\documentclass[11pt,reqno]{amsart}
\usepackage{bbm}
\usepackage{}
\usepackage{amssymb}
\usepackage{mathrsfs}
\usepackage{amsfonts}
\usepackage{amsfonts,amssymb,amsmath,amsthm}
\usepackage{url}
\usepackage{enumerate}
\usepackage[pdftex,bookmarksnumbered]{hyperref}
\usepackage{graphicx,xcolor}
\newcommand{\ds}{\displaystyle}
\newtheorem{theorem}{Theorem}[section]
\newtheorem{lemma}{Lemma}[section]
\newtheorem{proposition}{Proposition}[section]
\newtheorem{corollary}[theorem]{Corollary}
\theoremstyle{definition}
\newtheorem{definition}{Definition}[section]
\newtheorem{remark}{Remark}[section]
\numberwithin{equation}{section}

\newtheorem{example}{Example}[section]

\newtheorem{assumption}{Assumption}

\usepackage[left=2.0cm, right=2.0cm, top=2.0cm, bottom=2.0cm]{geometry}

\usepackage{graphicx}%
\usepackage{color}
\usepackage{cite}
\usepackage{hyperref}
\hypersetup{
	colorlinks = true,%
	citecolor = blue,
	filecolor=red,%
	linkcolor = [rgb]{0.65,0.0,0.0},%
	anchorcolor = red,
	pagecolor = red,
	urlcolor= [rgb]{0.65,0.0,0.0}
	linktocpage=true,
	pdfpagelabels=true,
	bookmarksnumbered=true,
}

\DeclareMathOperator{\id}{Id}

\DeclareMathOperator{\vol}{vol}
\DeclareMathOperator{\Ca}{Cap}
\DeclareMathOperator{\loc}{loc}

\DeclareMathOperator{\di}{div}\DeclareMathOperator{\Exp}{Exp}

\DeclareMathOperator{\F}{\mathsf{F}}

\DeclareMathOperator{\lo}{loc}

\DeclareMathOperator{\dd}{d}

\DeclareMathOperator{\sgn}{sgn}

\DeclareMathOperator{\Ric}{Ric}

\DeclareMathOperator{\supp}{supp}

\DeclareMathOperator{\tr}{{tr}}

\DeclareMathOperator{\dvol}{dvol}
\DeclareMathOperator{\dii}{div}

\DeclareMathOperator{\dmu}{d\mu}

\DeclareMathOperator{\MM}{\mathbb{M}}

\DeclareMathOperator{\VV}{\mathbb{V}}
\DeclareMathOperator{\B}{\mathsf{B}}
\DeclareMathOperator{\HH}{\mathsf{H}}
\DeclareMathOperator{\T}{\mathsf{T}}
\newcommand{\Lip}{\mathsf{Lip}}
\newcommand{\dil}{\mathsf{dil}}

\newcommand{\e}{{\varepsilon}}
\newcommand{\pa}{{\partial}}

\newcommand{\la}{{\langle}}
\newcommand{\ra}{{\rangle}}

\newcommand{\Bn}{{\bf n}}

\newcommand{\na}{{\nabla}}

\newcommand{\tJ}{\tilde{J}}

\author{Ningwei Cui}
\address{
School of Mathematics\\
Southwest Jiaotong University\\
610031 Chengdu, China}
\email{ningweicui@swjtu.edu.cn}

\author{Alexandru Krist\'aly}
\address{Department of Economics\\
	Babe\c s-Bolyai University\\
	400591 Cluj-Napoca, Romania \&  Institute of Applied Mathematics\\
 \'Obuda University\\
 1034 Budapest, Hungary}
  \email{alex.kristaly@econ.ubbcluj.ro; kristaly.alexandru@nik.uni-obuda.hu}

\author{Wei Zhao}
\address{
School of Mathematics\\
East China University of Science and Technology\\
200237 Shanghai, China}
\email{wzhao@ecust.edu.cn}


\thanks{A. Krist\'aly is  supported by the
	Excellence Researcher Program \'OE-KP-2-2022 of \'Obuda University, Hungary.  W. Zhao
is supported by Natural Science Foundation of Shanghai (No. 21ZR1418300).}

\begin{document}

\title[]{Sharp Hardy inequalities involving distance functions from submanifolds of Riemannian manifolds}

\begin{abstract} We establish various Hardy inequalities involving the distance function from submanifolds of Riemannian manifolds, where the natural weights are expressed in terms of bounds of the mean curvature of the submanifold and sectional/Ricci curvature of the ambient Riemannian manifold. Our approach is based on subtle Heintze--Karcher-type Laplace comparisons of the distance function and on a D'Ambrosio--Dipierro-type weak divergence formula for suitable vector fields, providing Barbatis--Filippas--Tertikas-type Hardy inequalities in the curved setting. Under very mild assumptions, we also establish the sharpness and non-existence of extremal functions within the Hardy inequalities and -- depending on the geometry of the ambient manifold -- their extensibility to various function spaces.  Several examples are provided by showing the applicability of our approach; in particular, well-known Hardy inequalities appear as limit cases of our new inequalities.
\end{abstract}
\maketitle

	\tableofcontents

\section{Introduction}
The theory of Hardy inequalities still hide interesting and unexplored questions, despite the fact that the first versions of such inequalities were formulated more than a hundred years ago, see Hardy \cite{Hardy}. The development of Hardy inequalities is strongly motivated by describing a large class of nonlinear phenomena involving certain singularities.  Up-to-date expositions of this still flourishing theory can be found in the monographs by Ghoussoub
and Moradifam \cite{GM-book}, and Ruzhansky and Suragan \cite{RS}.

Given $p>1$, an open domain $\Omega$ embedded into an ambient measure space $(M,{\rm d}\mathfrak{m})$, and two non-negative potentials $V,W$ on $\Omega $ with  $W(x)\to \infty$ (and perhaps $V(x)\to \infty$ as well) as $x\to x_0$ for some $x_0\in \Omega,$ a general Hardy inequality can be phrased as
\begin{equation}\label{Hardy-initial}
	\int_\Omega V|\nabla u|^p{\rm d}\mathfrak{m}\geq  \int_\Omega W| u|^p{\rm d}\mathfrak{m},\ \ \forall u\in C_0^\infty(\Omega).
\end{equation}
The most familiar form of \eqref{Hardy-initial} is obtained  when $V\equiv 1$ and $W(x)=\left(\frac{m-p}{p}\right)^p\frac{1}{|x|^p}$ with $p<m$, where $\Omega\subseteq \mathbb R^m$ is any domain containing the origin; moreover,  the constant  $\left(\frac{m-p}{p}\right)^p $ is sharp and never achieved.

A prominent setting for Hardy inequalities  is when \eqref{Hardy-initial} has a geometric flavor that is reflected in the geometry of the ambient space. The first key contribution in this direction   is provided by Carron \cite{Ca}, who extended  \eqref{Hardy-initial} to $n(\geq 3)$-dimensional Cartan-Hadamard manifolds (simply connected complete Riemannian manifolds with nonpositive sectional curvature) in the case when  $V\equiv 1$ and $W(x)=\left(\frac{n-2}{2}\right)^2\frac{1}{r(x)^2}$, where $r(x)$ is the Riemannian-distance of the point $x$ to a fixed point  $x_0$. Carron's work generated a huge wave of research on Hardy inequalities for Cartan-Hadamard manifolds;  without completeness, we mention the works by Berchio, Ganguly and Grillo \cite{BGG}, Berchio,  Ganguly,  Grillo and Pinchover \cite{BGGP},
D'Ambrosio and Dipierro \cite{DD}, Kombe and \"Ozaydin \cite{KM-1,KM-2}, Krist\'aly \cite{Kristaly},  Nguyen \cite{Nguyen}, Yang, Su
and Kong \cite{YSK}, etc.

An abstract approach to \eqref{Hardy-initial}  in the Euclidean setting is developed by Ghoussoub and Moradifam \cite{GM-1,GM-2}, who stated  that \eqref{Hardy-initial} holds if and only if $(V,W)$ forms a Bessel pair, which is based on the  solvability of a linear second-order Bessel-type ordinary differetial equation involving the potentials $V$ and $W$. Bessel pairs are extended to Cartan-Hadamard manifolds by Flynn, Lam and Lu \cite{FLL}, 
Flynn, Lam, Lu and Mazumdar \cite{FLLM}. In the same curved setting, a very recent paper by Kaj\'ant\'o, Krist\'aly, Peter and Zhao \cite{KKPZ} provides an alternative approach  to investigate \eqref{Hardy-initial} by using Riccati pairs.

The efficient development of Hardy inequalities on Cartan-Hadamard manifolds is not surprising, since they deeply exploit fundamental properties of such geometric objects, e.g.\ the emptiness of the cut locus of any point (in particular, the exponential map is a diffeomorphism), the availability of sharp Laplace/volume comparison principles with respect to the model Euclidean space, etc.

The present paper is (at least) two-fold. First, we intend to establish Hardy inequalities on Riemannian manifolds by handling not only the case when the sectional curvature  is bounded from above, but also the setting when the Ricci curvature is controlled from below. We notice that such bounds are not necessarily zeros;  in particular, in the latter case the manifold might be even compact. Second, we treat Hardy inequalities where the singularities of the potentials $V$ and $W$  may occur within a whole submanifold of the ambient space; as expected, a careful control on the mean curvature of the submanifold is needed. In  this way, the potentials $V$ and $W$ encode the bounds of both the Ricci/sectional curvature of the ambient space and the mean curvature of the submanifold. As far as we know, these are the first natural Hardy inequalities  involving nonvanishing `double-curvature' in terms of the bounds of the Ricci or sectional curvature of the ambient manifold as well as the mean curvature of the submanifold. In limit cases, we recover classical Hardy inequalities.

In the sequel, we shall provide a more detailed picture about our results; to do so, let $(M,g)$ be an $m$-dimensional complete Riemannian manifold and  $i:\Sigma\hookrightarrow M$ be an $n$-dimensional  submanifold.

 In \S  \ref{section-preliminaries} we present some preliminary notions and results which are indispensable to establish sharp Hardy inequalities. First, in \S \ref{Fermipremin}, some important aspects of the submanifold geometry together with the Fermi coordinates on $\Sigma$ are presented. Then, in \S\ref{section-Heintze--Karcher}, two Heintze--Karcher-type Laplacian estimates for the distance function $r(x):=r_\Sigma(x)=d(\Sigma,x)$ are provided, where $d$ is the natural metric function on $(M,g)$:     Theorem \ref{Laplacian-comparison-Sigma} is a slightly finer version than the original Heintze--Karcher theorem which is valid when the sectional curvature of $(M,g)$ is bounded from below, see Heintze and Karcher \cite{HK}, while  Theorem \ref{main-rem-2} -- which is new in the Riemannian geometry -- handles the counterpart case when the sectional  curvature  is bounded from above. Finally, in \S \ref{section-weak}, we recall and adapt to our setting the weak divergence notion of D'Ambrosio and Dipierro \cite{DD}.

In \S \ref{basichardy}, we state a generic inequality of the form \eqref{Hardy-initial} on any open domain $\Omega\subset M$ which verifies a Barbatis--Filippas--Tertikas-type geometric condition, see \cite{BFT}, i.e.,  $\Sigma=\pa\Omega$ if $n=m-1$ or  $\Sigma\cap \Omega\neq \emptyset$ if $0\leq n\leq m-2$. Under this geometric condition and a mild set of assumptions formulated in terms of joint-integrability of the pair of functions $(\phi,\psi)$, see Assumption \ref{Basciasumptionfunction}, we  prove in  \S \ref{section-construction} that
\begin{equation}
	\int_{\Omega}|\na u|^p \frac{\phi(r)^{p-1}|\psi'(r)|}{ |\nabla(\log \psi(r))|^{p}}\dvol_g\geq p^{-p}\int_{\Omega} |u|^p\phi(r)^{p-1}|\psi'(r)|\dvol_g,\label{intro-0-0}
\end{equation}
for every $u\in C^\infty_0(\Omega_\Sigma)$, where $\Omega_\Sigma=\Omega\setminus \Sigma$ and $\dvol_g$ is the canonical measure in $(M,g)$, see Theorem \ref{Hardy-thm-general}. The proof of \eqref{intro-0-0} relies on the weak divergence of D'Ambrosio and Dipierro \cite{DD}. In addition, in \S \ref{sharpnessfirst}  and \S \ref{nonexsorigin} we study the sharpness and the nonexistence of extremal functions in \eqref{intro-0-0}, respectively.

Since \eqref{intro-0-0} is expected to hold not only on $C^\infty_0(\Omega_\Sigma)$ but also on $C^\infty_0(\Omega)$, \S \ref{bigspaces} is devoted  to extend inequality \eqref{intro-0-0} to some larger function spaces. First, in \S \ref{section-4-1}, we can perform such an extension to the expected space $C^\infty_0(\Omega)$ by requiring a slightly
stronger assumption than in \S \ref{basichardy}, see Assumption \ref{Basciasumptionfunction3} and Theorem \ref{Hardy-thm-2}; the sharpness and nonexistence of extremal functions are also investigated in $C^\infty_0(\Omega)$, see Theorems \ref{sharp-u-general} and \ref{nonexiglobal}. In addition, we also observe that \eqref{intro-0-0} cannot be always extended  to   $C^\infty_0(\Omega)$, for instance, when $\Omega=M$ is compact; indeed, by using constant functions as test functions in \eqref{intro-0-0} -- which in such particular case clearly belong   to $C^\infty_0(\Omega)$ -- gives a contradiction. A natural choice to extend  \eqref{intro-0-0} from $C^\infty_0(\Omega_\Sigma)$ is the space $C_0^\infty(\Omega,\Sigma):=\{u\in C_0^\infty(\Omega):u(\Sigma)=0 \}.$ In this way, in \S \ref{Cinftyomegasigma}, we establish the extension of \eqref{intro-0-0} to  $C_0^\infty(\Omega,\Sigma)$ under the Assumption \ref{Basciasumptionfunction2}, see Theorem \ref{generalhardclosed}.

In \S \ref{applicII}, -- as a direct application of the previous results -- we establish sharp Hardy inequalities on manifolds with upper curvature bounds. In fact, we work under the `double-curvature' condition in the sense that  $(M,\Sigma,\Omega)$ satisfies the upper curvature bound condition $(\lambda,\kappa)$, i.e. the sectional curvature of $(M,g)$ and the mean curvature of $\Sigma$ are controlled from above by $\lambda$ and $\kappa$, respectively. Thus, if $(M,\Sigma,\Omega)$ satisfies the upper curvature bound condition $(\lambda,\kappa)$, and $\Sigma$ is a totally umbilical submanifold of $(M,g)$,   in \S \ref{section-5-1}
we prove the following Hardy inequality
\begin{equation}\label{Hardy-Ineq-curvature-condition52-0}
	\int_{\Omega}|\na u|^p\frac{\mathbf{c}_\Lambda(r)^{1-p}\mathbf{s}_\Lambda(r)^{p+\beta}}{(\mathbf{c}_\Lambda(r)-K \mathbf{s}_\Lambda(r))^n}\dvol_g\geq \left({\beta+m-n\over p}\right)^p\int_{\Omega} |u|^p\frac{\mathbf{c}_\Lambda(r)\mathbf{s}_\Lambda(r)^\beta}{(\mathbf{c}_\Lambda(r)-K \mathbf{s}_\Lambda(r))^n} \dvol_g,
\end{equation}
for any $u\in C^\infty_0(\Omega)$, see Theorem \ref{distance-weight-hardy-genera255}, where $\beta>-(m-n)$, $\lambda\leq \Lambda \leq 0$ and $\kappa \leq K$ with either $K\leq 0$ or
$K\geq  0$  and $K^2\leq -\Lambda$, while $\mathbf{s}_\Lambda$ and $\mathbf{c}_\Lambda$ are solutions to the ordinary differential equation $f''(t)+\Lambda f(t)=0$, $t>0$, see \eqref{s-lambda-Def}.
The proof of \eqref{Hardy-Ineq-curvature-condition52-0} is based on the verification of Assumption \ref{Basciasumptionfunction3} (see Theorem \ref{Hardy-thm-2}) for the  pair of functions $(\phi,\psi)$  defined by
\begin{equation}\label{newphipsi-2-0}
	\phi(t):=(\mathbf{c}_\Lambda(t)-K\mathbf{s}_\Lambda(t))^{-\frac{n}{p-1}}\mathbf{s}_\Lambda(t)^{-\frac{m-n-1}{p-1}}, \quad \psi(t):=\mathbf{s}_\Lambda(t)^{\beta+m-n}, \quad \forall\,t\in (0,\min\{\mathfrak{r}_{\Lambda},\mathfrak{t}_{\Lambda,K} \}),
\end{equation}
where the constants $\mathfrak{r}_{\Lambda}$ and $\mathfrak{t}_{\Lambda,K}$ are related to the first positive zeros of $\mathbf{s}_\Lambda$ and $\mathbf{c}_\Lambda-K\mathbf{s}_\Lambda$, respectively, if they do exist. In addition, it follows that  \eqref{Hardy-Ineq-curvature-condition52-0} is sharp if either $\Sigma$ is compact with $\Sigma\subset \Omega$ or $\Omega$ is  bounded with $\partial\Omega=\Sigma,$ and there is no extremal in either case. In particular, when $\beta=-p$,  $\lambda=\Lambda= \kappa = K=0$ and $\Sigma=\{x_0\}$ for some $x_0\in \Omega$ (thus $n=0$), inequality \eqref{Hardy-Ineq-curvature-condition52-0} reduces to the sharp Hardy inequality of Carron \cite{Ca}; other choices of parameters provide further particular forms of the Hardy inequality. In \S \ref{section-log-1},  sharp Hardy inequalities involving logarithmic double-curvature weights are established, see Theorem \ref{log-weight-thm-global}.

In \S \ref{applicI}, we  study Hardy inequalities on manifolds with lower curvature bounds.\ Although  such inequalities seem to be simple counterparts of the ones from \S \ref{applicII}, it turns out that genuine differences appear between the two settings. For instance, Bonnet--Myers theorem implies the compactness of the manifold  whenever the Ricci curvature is bounded from below by a positive constant, thus Hardy inequalities may yield undesirable phenomena, see \S
\ref{Cinftyomegasigma}. In a similar manner as in \S \ref{applicII},
we consider the `double-curvature' condition in the sense that  $(M,\Sigma,\Omega)$ satisfies the lower curvature bound condition $(\lambda,\kappa)$, i.e.\ the Ricci curvature of $(M,g)$ and the mean curvature of $\Sigma$ are controlled from below by $\lambda$ and $\kappa$, respectively. Under such curvature assumptions,  we establish in \S \ref {section-6-1} the counterpart of  \eqref{Hardy-Ineq-curvature-condition52-0}  for any $u\in C^\infty_0(\Omega_\Sigma)$ and $u\in C^\infty_0(\Omega,\Sigma)$ (instead of  $C^\infty_0(\Omega)$),
see Theorem \ref{distance-weight-hardy-general}/(i)\&(ii), respectively. The sharpness  and  nonexistence of extremal functions are also discussed in both  spaces $C^\infty_0(\Omega_\Sigma)$ and $C^\infty_0(\Omega,\Sigma)$. The proof of Theorem \ref{distance-weight-hardy-general} relies on Theorem \ref{generalhardclosed} by checking Assumption \ref{Basciasumptionfunction2}.
If $\lambda=\Lambda= \kappa = K=0$, Theorem \ref{distance-weight-hardy-general} reduces to the sharp Hardy inequality of Chen, Leung and Zhao \cite{CLZ} and  Meng, Wang and Zhao \cite{MWZ}.
Finally,  a sharp Hardy inequality is  stated with logarithmic double-curvature weights in \S \ref{section-6-2}, see Theorem \ref{log-weight-thm-general}.

\section{Preliminaries}\label{section-preliminaries}

\subsection{Submanifold geometry \& Fermi coordinates}\label{Fermipremin}
We  recall some definitions and properties of submanifolds, focusing mainly on Fermi coordinates; see Chavel \cite[Sections 3.6 \& 7.3]{Cha}, Heintze and Karcher \cite{HK},  and Dajczer and Tojeiro \cite{DT}
for  details. We first introduce the following notations.

Given a constant $\lambda\in \mathbb{R}$, define
\begin{eqnarray}
	\mathbf{s}_\lambda(t):=\left\{
	\begin{array}{lll}
		\frac{\sin(\sqrt{\lambda}t)}{\sqrt{\lambda}}, && \text{ if }\lambda>0,\\
		\ \ \ \ t, && \text{ if }\lambda=0,\\
		\frac{\sinh(\sqrt{-\lambda }t)}{\sqrt{-\lambda}}, && \text{ if }\lambda<0,
	\end{array}
	\right.\qquad
	\mathbf{c}_\lambda(t):=\left\{
	\begin{array}{lll}
		{\cos(\sqrt{\lambda}t)}, && \text{ if }\lambda>0,\\
		\ \ \ \ \  1, && \text{ if }\lambda=0,\\
		{\cosh(\sqrt{-\lambda}t)}, && \text{ if }\lambda<0,
	\end{array}
	\right.
	\label{s-lambda-Def}
\end{eqnarray}
where $0\leq t\leq 2\mathfrak{r}_{\lambda}$ with

\begin{equation}  \mathfrak{r}_\lambda:=\left\{
	\begin{array}{lll}
		\pi\over 2\sqrt{\lambda}, && \text{ if }\lambda>0,\\
		\\
		+\infty, && \text{ if }\lambda\leq0.
	\end{array}
	\right.
	\label{r-0-Def}
\end{equation}
It is easy to check that $\mathbf{s}_\lambda$ is strictly increasing on $[0,\mathfrak{r}_{\lambda}]$ and
both $\mathbf{s}_\lambda$ and $\mathbf{c}_\lambda$ are solutions to the ODE $f''(t)+\lambda f(t)=0$ with additional properties
\begin{equation}
	\mathbf{s}'_{\lambda}(t)=\mathbf{c}_\lambda(t), \quad \mathbf{c}'_{\lambda}(t)=-\lambda \mathbf{s}_\lambda(t), \quad  \lambda \mathbf{s}_\lambda(t)^2+\mathbf{c}_\lambda(t)^2=1.\label{relations}
\end{equation}

Given $\lambda, \kappa\in \mathbb{R}$, let $\mathfrak{c}_{\lambda,\kappa}$ be the first positive zero of $\mathbf{c}_\lambda-\kappa\mathbf{s}_\lambda$ (if it exists), i.e.,
\begin{equation}\label{cklambdafoot}
	\mathfrak{c}_{\lambda,\kappa}:=\sup\left\{t>0\,:\,    \mathbf{c}_\lambda(s)-\kappa\mathbf{s}_\lambda(s)  >0 \text{ for $s\in (0,t)$}  \right\}.
\end{equation}
A direct calculation yields that
\begin{eqnarray}\label{estimatecklambda}
	\mathfrak{c}_{\lambda,\kappa}=\left\{
	\begin{array}{lll}
		\frac{1}{\kappa}<\mathfrak{r}_\lambda, && \text{ if $\lambda=0$ and $\kappa> 0$},\\
		\\
		\frac{1}{\sqrt{\lambda}}\arctan \frac{\sqrt{\lambda}}{|\kappa|}<\mathfrak{r}_\lambda, && \text{ if $\lambda> 0$ and $\kappa> 0$},\\
		\\
		\frac{\pi}{2\sqrt{\lambda}}=\mathfrak{r}_\lambda, && \text{ if $\lambda> 0$ and $\kappa=0$},\\
		\\
		\frac{1}{\sqrt{\lambda}}\left(\pi-\arctan \frac{\sqrt{\lambda}}{|\kappa|}\right)\in (\mathfrak{r}_\lambda,2\mathfrak{r}_\lambda), && \text{ if $\lambda> 0$ and $\kappa< 0$},\\
		\\
		\frac{1}{\sqrt{-\lambda}}\text{artanh} \frac{\sqrt{-\lambda}}{|\kappa|}<\mathfrak{r}_\lambda, && \text{ if $\lambda< 0$ and $\kappa> 0$ with $\kappa^2> -\lambda$},\\
		\\
		+\infty=\mathfrak{r}_\lambda, && \text{ otherwise}.
	\end{array}
	\right.
\end{eqnarray}
In particular, it follows that
\begin{equation}\label{relationcandr}
	\mathfrak{c}_{\lambda,\kappa}\leq 2\mathfrak{r}_{\lambda}.
\end{equation}
Moreover,  it is immediate that
\[
\frac{\partial}{\partial \lambda}\left[\frac{\mathbf{c}_\lambda(t)}{ \mathbf{s}_\lambda(t) }-\kappa \right]< 0,\quad\frac{\partial}{\partial \kappa}\left[\frac{\mathbf{c}_\lambda(t)}{ \mathbf{s}_\lambda(t) }-\kappa \right]<0,\quad \forall\,t\in (0,2\mathfrak{r}_\lambda),
\]
which implies
\begin{equation}\label{footcomprison}
	\mathfrak{c}_{\Lambda,K}\leq \mathfrak{c}_{\lambda,\kappa}  \quad \text{ for any }\lambda\leq \Lambda, \   \kappa\leq K.
\end{equation}

The following result plays an important role in Sections \ref{applicII}--\ref{applicI}; for completeness, the proof is given in Appendix \ref{monotG}.

\begin{proposition}\label{gestiamtegra}Given $m,n\in \mathbb{N}$ and  $\lambda,\kappa\in \mathbb{R}$, let
	\begin{equation}\label{tlkdefin} \mathfrak{t}_{\lambda,\kappa}:=\mathfrak{t}_{n,\lambda,\kappa}:=\left\{
		\begin{array}{lll}
			\mathfrak{c}_{\lambda,\kappa}, && \text{ if }n>0,\\
			\\
			2\mathfrak{r}_{\lambda}, && \text{ if }n=0.
		\end{array}
		\right.
	\end{equation}
	Define
	a function $G_{\lambda,\kappa}:(0,\mathfrak{t}_{\lambda,\kappa})\rightarrow \mathbb{R}$ by
	\begin{equation}\label{Gfunctiondefined}  G_{\lambda,\kappa}(t):=G_{m,n,\lambda,\kappa}(t)=\left\{
		\begin{array}{lll}
			-n{\lambda\mathbf{s}_\lambda(t)+\kappa \mathbf{c}_\lambda(t)\over \mathbf{c}_\lambda(t)-\kappa \mathbf{s}_\lambda(t)}+(m-n-1){\mathbf{c}_\lambda(t)\over \mathbf{s}_\lambda(t)}, && \text{ if }n>0,\\
			\\
			(m-1){\mathbf{c}_\lambda(t)\over \mathbf{s}_\lambda(t)}, && \text{ if }n=0.
		\end{array}
		\right.
	\end{equation}
If $m\geq n+1$, then $G_{\lambda,\kappa}$ is decreasing in both $\lambda$ and $\kappa$, i.e., for any $\lambda\leq \Lambda$ and $ \kappa\leq K$,
	\[
	G_{\Lambda,K}(t)\leq G_{\lambda,\kappa}(t), \quad \forall\,t\in (0,\mathfrak{t}_{\Lambda,K}).
	\]
\end{proposition}

Throughout the paper,
$(M,g)$ is an $m(\geq2)$-dimensional complete (possibly compact) Riemannian manifold, and $i:\Sigma\hookrightarrow M$ is an $n$-dimensional {\it closed}  submanifold with $0\leq n\leq m-1$, i.e., $(i,\Sigma)$ is a regular submanifold of $M$ and $i(\Sigma)$ is a closed set of $M$.  The {\it distance function from $\Sigma$} is denoted  by $r(x):=d(\Sigma, x)$.

Let $\mathcal {V}\Sigma$ and  $\mathcal {V}S\Sigma$ be the {\it normal bundle} and the {\it unit normal bundle}  of $\Sigma$ in $M$, respectively. The fibers of $\mathcal {V}\Sigma$ and $\mathcal {V}S\Sigma$ at $x\in \Sigma$ are denoted by $\mathcal {V}_x\Sigma$ and $\mathcal {V}S_x\Sigma$, respectively. 
The {\it second fundamental
	form} $\mathfrak{B}$ of $\Sigma$ in $M$    is a vector-valued bilinear symmetric form $\mathfrak{B}:T_x\Sigma\times T_x\Sigma\rightarrow \mathcal {V}_x\Sigma$ for every $x\in \Sigma$, given by
\begin{equation}\label{secondfundadef}
	\mathfrak{B}(X,Y):=(\nabla_X\overline{Y})^\bot,
\end{equation}
where $\overline{Y}$ is any extension of $Y$ to a tangent vector field on $\Sigma$, $\nabla$ is the
Levi-Civita connection of the ambient space $(M,g)$, and the superscript $\bot$ denotes the projection onto the normal fiber $\mathcal {V}_x\Sigma$. A {\it totally geodesic submanifold} is characterized by the fact that
its second fundamental form vanishes.

Given $\mathbf{n}\in \mathcal {V}S_x\Sigma$, the {\it Weingarten map} $\mathfrak{A}^{\mathbf{n}}:T_x\Sigma\rightarrow T_x\Sigma$ is defined by
\begin{equation}\label{weingradef}
	\mathfrak{A}^{\mathbf{n}}(X):=-(\nabla_X\overline{\mathbf{n}})^\top,
\end{equation}
where $\overline{\mathbf{n}}$ is any extension of $\mathbf{n}$ to a normal vector field on $\Sigma$, and the superscript
$\top$ denotes the  projection onto $T_x\Sigma$.
The {\it mean curvature} is defined by
\[
H:=\frac{1}{n}\tr_{i^*g}\mathfrak{B},
\]
where $i^*g$ is the pull-back metric on $\Sigma.$

For any $\mathbf{n}\in \mathcal {V}S\Sigma$, one has that $\tr\mathfrak{A}^{\mathbf{n}}=n \langle H,\mathbf{n}\rangle$.
The submanifold $\Sigma$ is said to be {\it totally umbilical} if
\begin{equation}\label{defunbm}
	\mathfrak{B}(X,Y)=\langle X, Y\rangle\, H,
\end{equation}
for all vector fields $X,Y$ tangent to $\Sigma$.

Given $x\in M$, let $\exp_x:T_xM\rightarrow M$ be the standard exponential map of $(M,g)$. Thus,
the {\em exponential map of normal bundle} is defined as
\[
\text{Exp}: \mathcal{V}\Sigma\rightarrow M,\quad \Bn\mapsto \exp_{\pi(\Bn)}(\Bn),
\]
where  $\pi: \mathcal{V}\Sigma\rightarrow \Sigma$ is the natural {\it bundle projection}.
Due to Gray \cite{Gray}  and O'Neil \cite{ONeil}, for any $x\in \Sigma$, there is a neighborhood $U\subset \Sigma$ of $x$ and $\e_x>0$ such that $\text{Exp}|_{ \B(U,\e_x)}$ is a diffeomorphism, where
\[
\B(U,\e_x):=\left\{y\in \cup_{z\in U}\mathcal {V}_z\Sigma\,:\,0 \leq |y|<\e_x\right\}.
\]
Moreover, if $\Sigma$ is compact, there is a universal $\e>0$ such that
${\dd}\text{Exp}_{t\Bn}$ is nonsingular for all $t\in [0,\e)$ and $\Bn\in \mathcal{V}S\Sigma$.

Given $\mathbf{n}\in \mathcal {V}S\Sigma$,   let $\gamma_\mathbf{n}$ be the geodesic with $\gamma_\mathbf{n}'(0)=\mathbf{n}$, i.e., $ \gamma_\mathbf{n}(t)=\Exp(t\mathbf{n})$.
Then the {\it distance to the focal cut point}
of $\Sigma$ with respect to $\mathbf{n}$  is defined as
\[
c_{\mathcal {V}}(\mathbf{n}):=\sup\{t>0\,:\,d(\Sigma,\gamma_\mathbf{n}(t))=t\}.
\]
The {\it focal cut radius} of $\Sigma$ is defined by
\begin{equation}\label{cutradidfef}
	\mathfrak{c}_{\mathcal {V}}(\Sigma):=\inf_{\mathbf{n}\in \mathcal {V}S\Sigma}c_{\mathcal {V}}(\mathbf{n}).
\end{equation}
In addition, one may introduce
\begin{align*}
	&\mathcal {V}C\Sigma:=\{ c_{\mathcal {V}}(\mathbf{n})\mathbf{n}\,:\,c_{\mathcal {V}}(\mathbf{n})<+\infty,\mathbf{n}\in \mathcal {V}S\Sigma\},&\mathcal {V}\mathscr{C}\Sigma:=\Exp\mathcal {V}C\Sigma,\\
	&\mathcal {V}D\Sigma:=\left\{t\mathbf{n}\,:\,0\leq   t<c_{\mathcal {V}}(\mathbf{n}),\, \mathbf{n}\in \mathcal {V}S\Sigma\right\}, &\mathcal {V}\mathscr{D}\Sigma:=\Exp\mathcal {V}D\Sigma.
\end{align*}
The following result summarizes the main properties of the latter notions; see Chavel \cite[p.\,105ff, p.\,134]{Cha}:

\begin{lemma}\label{prop-1-3} Let $(M,g)$ be an $m$-dimensional complete Riemannian manifold and let $i:\Sigma\hookrightarrow M$ be an $n$-dimensional closed submanifold. Then the following statements  hold$:$
	\begin{enumerate}[{\rm(i)}]
		\item \label{geodesictmean} $\gamma_\mathbf{n}(t)$, $t\in [0,c_{\mathcal {V}}(\mathbf{n})]$ is a minimal geodesic from $\Sigma$ to $\gamma_{\mathbf{n}}(c_{\mathcal {V}}(\mathbf{n}))$. In particular, for any $t_0\in (0,c_{\mathcal {V}}(\mathbf{n}))$, $\gamma_\mathbf{n}$ is the unique minimal geodesic from $\Sigma$ to $\gamma_{\mathbf{n}}(t_0);$
		
		\smallskip
		
		\item\label{cutlocuszero}  $c_{\mathcal {V}}:\mathcal {V}S\Sigma\rightarrow (0,+\infty]$ is a continuous function and hence, $\mathcal {V}\mathscr{C}\Sigma$ is a set of Lebesgue measure zero$;$
		
		\smallskip
		
		\item  the focal cut domain $\mathcal{V}\mathscr{D}\Sigma$ is the largest starlike domain, i.e.,  $\Exp(t\Bn)\in \mathcal{V}\mathscr{D}\Sigma$  for any $\Bn\in \mathcal{V}S\Sigma$ and $t\in [0,c_\mathcal {V}(\Bn));$
		
		\smallskip
		
		\item\label{LEMM2.1-4}  the map $\Exp:\mathcal {V}D\Sigma \rightarrow\mathcal {V}\mathscr{D}\Sigma$ is a diffeomorphism, i.e., $\mathcal {V}\mathscr{D}\Sigma=M\backslash  \mathcal {V}\mathscr{C}\Sigma$.
	\end{enumerate}
\end{lemma}

Owing to Lemma \ref{prop-1-3}, one can introduce the {\it Fermi coordinate system}  by
\begin{equation}\label{Fermimap}
	\F: [0,+\infty)\times \mathcal{V}S\Sigma\rightarrow M,\quad (t,\Bn)\mapsto \text{Exp}(t\Bn).
\end{equation}
In particular, if $\Sigma$ is a point (i.e., of $0$-dimension), then the Fermi coordinate system is exactly the polar coordinate system.
Moreover, we have the following result; see Chavel \cite[p.~133]{Cha}:

\begin{lemma}\label{Fermifirstlemma} If $\F(t,\mathbf{n})$ is well-defined,  then
	\begin{itemize}
		
		
		\item[(i)] $\frac{\partial}{\partial t}:={\dd}{\F}_{(t,\mathbf{n})}\mathbf{n}={\gamma}'_\mathbf{n}(t)$, i.e., the tangent vector of  $\gamma_\mathbf{n}(t)$ at $t;$
		
		\smallskip
		
		\item[(ii)] given $X\in T_{\pi(\mathbf{n})}\Sigma$, $J_X(t):={\dd}{\F}_{(t,\mathbf{n})}X$ is a Jacobi field along $\gamma_\mathbf{n}$, orthogonal to $\gamma_{\mathbf{n}}$, with
		\[
		J_X(0)=X,\quad (\nabla_{\gamma'_\mathbf{n}(0)}J_X)(0)=-\mathfrak{A}^{\mathbf{n}}(X);
		\]

		\item[(iii)] given $Y\in \mathcal {V}_{\pi(\mathbf{n})}\Sigma\cap \mathbf{n}^\bot$, $J_Y(t):={\dd}{\F}_{(t,\mathbf{n})}Y$ is a Jacobi field along $\gamma_\mathbf{n}$,  orthogonal to $\gamma_{\mathbf{n}}$,  with
		\[
		J_Y(0)=0,\quad (\nabla_{\gamma'_\mathbf{n}(0)}J_Y)(0)=Y,
		\]
		where $\mathbf{n}^\bot:=\{Y\in T_{\pi(\mathbf{n})}M:\,g(\mathbf{n},Y)=0 \}$.
		
	\end{itemize}
\end{lemma}

Now we study the construction of such Jacobi fields.
Let $R$ be the {\it curvature tensor} of $(M,g)$, i.e.,
\[
R(U,V)W=\nabla_U\nabla_VW-\nabla_V\nabla_UW-\nabla_{[U,V]}W,
\]
where $U,V,W$ are three smooth vector fields on $M$. Given $\mathbf{n}\in \mathcal {V}S\Sigma$,
denote by $P_{t;\mathbf{n}}$ the {\it parallel transport} along
$\gamma_{\mathbf{n}}$ from $T_{\gamma_\mathbf{n}(0)}M$ to $T_{\gamma_\mathbf{n}(t)}M$  for all $t\geq 0$. Set
$R_{{\gamma}'_\mathbf{n}(t)}:=R\left(\,\cdot\,,{\gamma}'_\mathbf{n}(t)\right){\gamma}'_\mathbf{n}(t)$ and
\[
\mathcal {R}(t,\mathbf{n} ):=P^{-1}_{t;\mathbf{n}}\circ R_{{\gamma}'_\mathbf{n}(t)} \circ P_{t;\mathbf{n}}:
\mathbf{n}^{\bot}\rightarrow \mathbf{n}^{\bot}.\]
Let $\mathcal {A}(t,\mathbf{n})$ be the solution of the matrix (or linear
transformation) ordinary differential equation on $\mathbf{n}^\bot$:
\begin{equation}
	\left \{
	\begin{array}{lll}
		\mathcal {A}{''}(t,\mathbf{n})+\mathcal {R}(t,\mathbf{n})\,\mathcal {A}(t,\mathbf{n})=0,\\
		\\
		\mathcal {A}(0,\mathbf{n})|_{T_{\pi(\mathbf{n})}\Sigma}=\id|_{T_{\pi(\mathbf{n})}\Sigma},\ \mathcal {A}'(0,\mathbf{n})|_{T_{\pi(\mathbf{n})}\Sigma}=-\mathfrak{A}^{\mathbf{n}},\label{defineJacbi}\\
		\\
		\mathcal {A}(0,\mathbf{n})|_{\mathcal {V}_{\pi(\mathbf{n})}\Sigma\cap \mathbf{n}^\bot}=0,\ \mathcal {A}'(0,\mathbf{n})|_{\mathcal {V}_{\pi(\mathbf{n})}\Sigma\cap \mathbf{n}^\bot}=\id|_{\mathcal {V}_{\pi(\mathbf{n})}\Sigma\cap \mathbf{n}^\bot},
	\end{array}
	\right.
\end{equation}
where $\mathcal {A}'=\frac{\dd}{{\dd} t}\mathcal {A}$ and $\mathcal {A}''=\frac{{\dd}^2}{{\dd} t^2}\mathcal {A}$. On account of Lemma \ref{Fermifirstlemma}, we have
\begin{equation}\label{constructionofJacobifile}
	\left \{
	\begin{array}{lll}
		\frac{\partial}{\partial t}:={\gamma}'_{\mathbf{n}}(t)=P_{t;\mathbf{n}}\mathbf{n},\\
		\\
		J_X(t)={\dd}{\F}_{(t,\mathbf{n})}X=P_{t;\mathbf{n}}\mathcal {A}(t,\mathbf{n})X, \quad \forall\,X\in T_{\pi(\mathbf{n})}\Sigma,\\
		\\
		J_Y(t)={\dd}{\F}_{(t,\mathbf{n})}Y=P_{t;\mathbf{n}}\mathcal {A}(t,\mathbf{n})Y,\quad \forall\,Y\in \mathcal {V}_{\pi(\mathbf{n})}\Sigma\cap \mathbf{n}^\bot.
	\end{array}
	\right.
\end{equation}
It should be emphasized that $\det\mathcal {A}(t,\mathbf{n})>0$ for any $t\in (0,c_\mathcal {V}(\mathbf{n}))$.
Moreover,   (\ref{defineJacbi}) yields
\begin{align}
	\underset{t\rightarrow 0^+}{\lim}\frac{\det \mathcal {A}(t,\mathbf{n})}{t^{m-n-1}}=1,\quad \frac{(\det \mathcal {A}(t,\mathbf{n}))' }{\det \mathcal {A}(t,\mathbf{n})}=\frac{(m-n-1)}t- \tr\mathfrak{A}^{\mathbf{n}}+O(t)\ \ {\rm as}\ \ t\to 0.\label{smallsestimatedatA}
\end{align}

Using the Fermi coordinate system $(t,\Bn)$, the {\it Riemannian measure} $\dvol_g$ of $(M,g)$ can be represented as
\begin{align}
	{\dvol}_g(t,\mathbf{n}) =\det \mathcal {A}(t,\mathbf{n}) {\dd}t \cdot {\dvol}_{i^*g}(x) \cdot d\nu_x(\mathbf{n}),\label{dv-g-Def}
\end{align}
where ${\dvol}_{i^*g}$ is the {\it induced Riemannian measure} of $(\Sigma,i^*g)$ and ${\dd}\nu_x$ is the Riemannian measure of $\mathcal {V}S_{x}\Sigma$ with $x=\pi(\mathbf{n})$.
Thus,  for $f\in L^1(M)$, the change of variables formula reads as
\[
\int_M f \,{\dvol}_g=\int_{\Sigma}{\dvol}_{i^*g}(x)\int_{\mathcal {V}S_{x}\Sigma}d\nu_x(\mathbf{n})\int_0^{c_\mathcal {V}(\mathbf{n})}f(\Exp(t\mathbf{n}))\, \det \mathcal {A}(t,\mathbf{n}){\dd}t.\label{integration-f}
\]


Throughout the paper, we always use  $\T_s$ to denote the {\it $s$-tuber neighborhood of $\Sigma$}, i.e.,
\begin{equation}\label{stuber}
	\T_s:=\{x\in M\,:\, r(x)<s\}.
\end{equation}

\subsection{Heintze--Karcher-type Laplacian  estimates}\label{section-Heintze--Karcher}
In this subsection we establish two Heintze--Karcher-type Laplacian estimates for the distance function $r(x):=d(\Sigma,x)$. Given $x\in \mathcal {V}\mathscr{D}\Sigma\backslash\Sigma$, due to Lemma \ref{prop-1-3}/\eqref{LEMM2.1-4} and \eqref{Fermimap}, there is a unique $\mathbf{n}\in \mathcal {V}S\Sigma$ and $t=r(x)\in (0, c_\mathcal {V}(\mathbf{n}))$ such that $x=\F(t,\mathbf{n})$. A direct calculation (cf. Chen et al. \cite[(2.6) \& (2.9)]{CLZ}) yields
\begin{equation}
	\nabla r|_{(t,\mathbf{n})}= {\gamma}'_{\mathbf{n}}(t)=\left.\frac{\partial}{\partial t}\right|_{(t,\mathbf{n})},\quad \Delta r|_{(t,\Bn)}=\Delta t={(\det \mathcal {A}(t,\Bn))'\over \det \mathcal {A}(t,\Bn)}.\label{Laplacian-expression}
\end{equation}


In the sequel, we provide a slightly finer version of the Heintze--Karcher theorem when the curvature is controlled from below; compare with \cite[Section 3.2]{HK}:
\begin{theorem}\label{Laplacian-comparison-Sigma} Let $(M,g)$ be an $m$-dimensional complete Riemannian manifold with $\sec_M\geq\lambda$, and $i:\Sigma\hookrightarrow M$ be an $n$-dimensional closed submanifold. Then, for every $\Bn\in \mathcal{V}S\Sigma$, one has
	\begin{align}
		\det \mathcal {A}(t,\mathbf{n})\leq  \left[ \mathbf{c}_\lambda(t)-\frac{\tr \mathfrak{A}^\mathbf{n}}{n}\mathbf{s}_\lambda(t) \right]^n \mathbf{s}_\lambda(t)^{m-n-1},\quad \forall\, t\in [0,c_{\mathcal {V}}(\mathbf{n})]\label{Lap-deta-controll},\\
		\Delta r|_{(t,\mathbf{n})}\leq G_{\lambda,\frac{\tr \mathfrak{A}^\mathbf{n}}{n}}(t)=-n{  \lambda\mathbf{s}_\lambda(t)+\frac{\tr \mathfrak{A}^\mathbf{n}}{n}\mathbf{c}_\lambda(t)\over \mathbf{c}_\lambda(t)-\frac{\tr \mathfrak{A}^\mathbf{n}}{n}\mathbf{s}_\lambda(t)}+(m-n-1){\mathbf{c}_\lambda(t)\over \mathbf{s}_\lambda(t)},\quad \forall\, t\in (0,c_{\mathcal {V}}(\mathbf{n})).\label{Laplacian-comparison-Sigma-Ineq}
	\end{align}
	Moreover, provided either $n=0$ $($i.e., $\Sigma$ is a point$)$ or $n=m-1$ $($i.e., $\Sigma$ is a hypersurface$),$ then both \eqref{Lap-deta-controll} and \eqref{Laplacian-comparison-Sigma-Ineq}   remain  valid when $\Ric_M\geq (n-1)\lambda$ instead of $\sec_M\geq\lambda$.
\end{theorem}
\begin{proof} If $n=0$,  the statement is nothing but the well-known Bishop comparison theorem; see  Chavel \cite[Theorem 3.8]{Cha}. When $n=m-1$, the statement is exactly the result of Chen et al. \cite[Theorem 2.5]{CLZ}; see also \cite[Theorem 3.14]{Cha}. In the sequel, we show the general case, whose proof is divided into two steps.
	
	\textit{Step 1.} We first show that under the condition $\sum_{i=1}^n\kappa_i=n\kappa$ for some $\kappa\in \mathbb{R}$, one has
	\begin{equation}
		\sum_{i=1}^{n}{(\mathbf{c}_\lambda-\kappa_i\mathbf{s}_\lambda)'(t)\over (\mathbf{c}_\lambda-\kappa_i \mathbf{s}_\lambda)(t)}\leq n{(\mathbf{c}_\lambda-\kappa\mathbf{s}_\lambda)'(t)\over (\mathbf{c}_\lambda-\kappa\mathbf{s}_\lambda)(t)},\quad 0<t<\min_i \mathfrak{c}_{\lambda,\kappa_i}, \label{maximum-La}
	\end{equation}
	with equality if and only if $\kappa_1=\cdots=\kappa_n=\kappa$; here, $\mathfrak{c}_{\lambda,\kappa_i}$ is defined by \eqref{cklambdafoot}.

	\smallskip
	
Let $0<t<\min_i \mathfrak{c}_{\lambda,\kappa_i}$  be fixed; we prove \eqref{maximum-La} by using  Lagrange multipliers. Let us consider the Lagrange function
	\[
	\Phi(\kappa_1,\cdots,\kappa_n,\mu):=\sum_{i=1}^{n}{(\mathbf{c}_\lambda-\kappa_i \mathbf{s}_\lambda)'(t)\over (\mathbf{c}_\lambda-\kappa_i \mathbf{s}_\lambda)(t)}+\mu\left(\sum_{i=1}^n\kappa_i-n\kappa\right),
	\]
	where $\mu\in \mathbb{R}$ is the Lagrange multiplier. By using (\ref{relations}),  we obtain for any $1\leq j\leq n$ that
	\begin{eqnarray}
		{\pa \Phi\over \pa\kappa_j}&=&
		-{1\over (\mathbf{c}_\lambda(t)-\kappa_j \mathbf{s}_\lambda(t))^2}+\mu.\nonumber
	\end{eqnarray}
Therefore, the unique  critical point of $\Phi$ is $\kappa_1=\cdots=\kappa_n=\kappa$ and $\mu=(\mathbf{c}_\lambda(t)-\kappa\mathbf{s}_\lambda(t))^{-2}$. Moreover, the Hessian of $\Phi$ at this critical point is
	\[
	\left({\pa^2\Phi\over \pa \kappa_i\pa \kappa_j}\right)(\kappa,\ldots,\kappa,(\mathbf{c}_\lambda(t)-\kappa\mathbf{s}_\lambda(t))^{-2})=\left({-2\mathbf{s}_\lambda(t)\delta_{ij}\over (\mathbf{c}_\lambda(t)-\kappa \mathbf{s}_\lambda(t))^3}\right)_{i,j=1,...,n},
	\]
	which is negative definite. Accordingly, this critical point is the maximum of $\Phi$ under the condition $\sum_{i=1}^n\kappa_i=n\kappa$; hence (\ref{maximum-La}) follows.
	
	\medskip

	\textit{Step 2.} We now complete the proof of Theorem \ref{Laplacian-comparison-Sigma}.
		Given $\mathbf{n}\in \mathcal{V}S\Sigma$,  it follows from the proof the Heintze--Karcher comparison theorem, see \cite[Theorem 7.5]{Cha} or  \cite[Theorem 2.4]{CLZ}, that $c_\mathcal {V}(\mathbf{n})\leq \min_i \mathfrak{c}_{\lambda,\kappa_i} \leq \mathfrak{c}_{\lambda,\frac{\tr \mathfrak{A}^\mathbf{n}}{n}}$ and
	\begin{align}
		{(\det \mathcal {A}(t,\Bn))'\over \det \mathcal {A}(t,\Bn)}&\leq \sum_{i=1}^{n}{(\mathbf{c}_\lambda-\kappa_i \mathbf{s}_\lambda)'(t)\over (\mathbf{c}_\lambda-\kappa_i \mathbf{s}_\lambda)(t)}+(m-n-1){\mathbf{s}_\lambda'(t)\over \mathbf{s}_\lambda(t)},\quad \forall\, t\in (0,c_\mathcal {V}(\Bn)),\label{Heintze-Karcher-Ineq}
	\end{align}
	where $\{\kappa_i\}_{i=1}^n$ are   the eigenvalues   of $\mathfrak{A}^{\mathbf{n}}$ with respect to some orthonormal basis of $T_{\pi(\mathbf{n})}\Sigma$.
	Therefore, \eqref{Lap-deta-controll} and \eqref{Laplacian-comparison-Sigma-Ineq} follow by \eqref{maximum-La},  \eqref{Heintze-Karcher-Ineq}, \eqref{smallsestimatedatA} and \eqref{Laplacian-expression}$_2$; here, \eqref{Laplacian-expression}$_2$ denotes the second relation in \eqref{Laplacian-expression}.
\end{proof}

We now establish a counterpart of Theorem \ref{Laplacian-comparison-Sigma} when the curvature is bounded from above.
\begin{theorem}\label{main-rem-2} Let $(M,g)$ be an $m$-dimensional complete Riemannian manifold with $\sec_M\leq \lambda$, and  $i:\Sigma\hookrightarrow M$ be an $n$-dimensional closed submanifold. Given $\mathbf{n}\in \mathcal {V}S\Sigma$, if the eigenvalues $\{\kappa_i\}_{i=1}^n$ of $\mathfrak{A}^{\mathbf{n}}$ with respect to some orthonormal basis of $T_{\pi(\mathbf{n})}\Sigma$ satisfy $\kappa_i\leq \kappa$ for all $1\leq i\leq n$, then
\begin{align}
\det \mathcal {A}(t,\Bn)\geq  (\mathbf{c}_\lambda(t)-\kappa\mathbf{s}_\lambda(t))^n\,\mathbf{s}_\lambda(t)^{m-n-1}, \quad \forall\, t\in[0,\min\{c_\mathcal {V}(\mathbf{n}),\mathfrak{t}_{\lambda,\kappa}\}],\label{det-comparison-reverse2}\\
\Delta r|_{(t,\mathbf{n})}\geq G_{\lambda,\kappa}(t)=-n{\lambda\mathbf{s}_\lambda(t)+\kappa \mathbf{c}_\lambda(t)\over \mathbf{c}_\lambda(t)-\kappa \mathbf{s}_\lambda(t)}+(m-n-1){\mathbf{c}_\lambda(t)\over \mathbf{s}_\lambda(t)},\quad\forall\, t\in(0,\min\{c_\mathcal {V}(\mathbf{n}),\mathfrak{t}_{\lambda,\kappa}\}).\label{Laplace-comparison-reverse2}
\end{align}
\end{theorem}

In order to prove Theorem \ref{main-rem-2}, we need some  auxiliary  notions and results.
Given $\mathbf{n}\in \mathcal {V}S\Sigma$, let $\gamma_\mathbf{n}$ be a geodesic with ${\gamma}'_\mathbf{n}(0)=\mathbf{n}$. A Jacobi field $J$ along $\gamma_{\mathbf{n}}$ is called {\it transverse} if it is point-wisely orthogonal to $\gamma_{\mathbf{n}}$ with
\begin{equation}\label{deftransver}
	J(0)\in T_{\pi(\mathbf{n})}\Sigma,\quad (\nabla_{{\gamma}'_\mathbf{n}}J)(0)+\mathfrak{A}^{\mathbf{n}}J(0)\bot T_{\pi(\mathbf{n})}\Sigma.
\end{equation}
Let $\Upsilon_{\mathbf{n}}$ be the collection of piece-wisely smooth vector fields $X$ along $\gamma_{\mathbf{n}}$, point-wisely orthogonal to $\gamma_\mathbf{n}$ and $X(0)\in T_{\pi(\mathbf{n})}\Sigma$. Given $b>0$, the {\it index form} $I_b(\cdot,\cdot)$ on $\Upsilon_{\mathbf{n}}$ is defined by
\begin{equation}
	I_b(X,Y):=-\langle  {\mathfrak{A}}^{\mathbf{n}}X(0),Y(0)\rangle+\int^b_0 \Big(\la\na_{{\gamma}'_{\mathbf{n}}}X,\na_{{\gamma_{\mathbf{n}}}'}Y\ra-\la R_{\gamma'_{\mathbf{n}}}(X),Y\ra \Big) {\dd}t,\quad \forall\,X,Y\in \Upsilon_{\mathbf{n}}.\label{Index-Form-general-0}
\end{equation}

The standard Jacobi criteria  reads as follows, see Chavel \cite[Section 7.3]{Cha}:
\begin{lemma}[Jacobi Criteria]\label{INDEX-CRITERIA-modified-version} Let $(M,g)$ be an $m$-dimensional complete Riemannian manifold and $i:\Sigma\hookrightarrow M$ be an $n$-dimensional closed submanifold. Given $\mathbf{n}\in \mathcal {V}S\Sigma$,
	let $\gamma_\mathbf{n}:[0,b]\rightarrow M$ be a geodesic satisfying ${\gamma}'_\mathbf{n}(0)=\mathbf{n}$, $b\in (0,c_\mathcal {V}(\mathbf{n}))$ and
	let $J$ be a transverse Jacobi field along $\gamma_{\mathbf{n}}$. Then
	\[
	I_b(X,X)\geq I_b(J,J),\quad \forall\, X\in \Upsilon_{\mathbf{n}} \text{ with } X(b)=J(b),
	\]
	with equality if and only if $X=J$.
\end{lemma}

We are also  going to apply the following result by  Dajczer and Tojeiro {\cite[Proposition 1.20]{DT}:
	
	\begin{lemma}\label{umbilical-lem} Let ${\MM}_{\lambda}$ be an $m$-dimensional simply connected complete Riemannian manifold of constant sectional curvature $\lambda$  and let $\widetilde{o}\in {\MM}_{\lambda}$ be a fixed point. Then for every $n(\geq 1)$-dimensional proper vector subspace $\VV\subset T_{\widetilde{o}}{\MM}_{\lambda}$ and every vector $\HH\in  T_{\widetilde{o}}  {\MM}_{\lambda}$ orthogonal to $\VV$, there exists a unique $n$-dimensional complete totally umbilical submanifold $\widetilde{\Sigma}$ through $\widetilde{o}$  satisfying
		\begin{itemize}
			\item $T_{\widetilde{o}}\widetilde{\Sigma}=\VV$ and $\HH$ is  the  mean curvature vector of $\widetilde{\Sigma}$ at $\widetilde{o};$
			
			\item $\widetilde{\Sigma}$ is extrinsic sphere, i.e., its mean curvature is parallel$;$
			
			\item  the sectional curvature $\sec_{\widetilde{\Sigma}}\equiv\lambda+|\HH|^2$ and particularly, ${\widetilde{\Sigma}}$ is totally geodesic if and only if  $\HH=0$. 
		\end{itemize}
	\end{lemma}
	\begin{remark}\label{principlecurvauteform1}
		Let $\widetilde{H}$ denote the mean curvature of $\widetilde{\Sigma}$. Since $\widetilde{\Sigma}$ is an extrinsic sphere, we have $ |\widetilde{H}|(\tilde{x})\equiv|\HH|$ for any $\tilde{x}\in \widetilde{\Sigma}$. Moreover, for any $\tilde{\mathbf{n}}\in \mathcal {V}S_{\tilde{x}}\widetilde{\Sigma}$, it follows by \eqref{secondfundadef}--\eqref{defunbm} that the Weingarten map satisfies
		\[
		\widetilde{\mathfrak{A}}^{\tilde{\mathbf{n}}}=\langle \widetilde{H},\tilde{\mathbf{n}}   \rangle\id_{T_{\tilde{x}}\widetilde{\Sigma}}.
		\]
		In particular, given   $\kappa\in \mathbb{R}$ with $|\kappa|=|\widetilde{H}|(\tilde{x})$,  we have $\widetilde{\mathfrak{A}}^{\tilde{\mathbf{n}}}=\kappa\id_{T_{\tilde{x}}\widetilde{\Sigma}}$ by choosing
		\begin{equation}\label{unitvectorchoose}
			\tilde{\mathbf{n}}:=\left \{
			\begin{array}{lll}
				\text{sgn}(\kappa)\frac{\widetilde{H}(\tilde{x})}{|\widetilde{H}(\tilde{x})|},&\text{ if }\kappa\neq0,\\
				\\
				\text{any unit normal vector in $\mathcal {V}S_{\tilde{x}}\widetilde{\Sigma}$},&\text{ if }\kappa=0.
			\end{array}
			\right.
		\end{equation}
	\end{remark}
	
	{\it Proof of Theorem \ref{main-rem-2}.}  Given $\mathbf{n}\in \mathcal {V}S\Sigma$, let $\gamma_\mathbf{n}$ be a geodesic satisfying ${\gamma}'_\mathbf{n}(0)=\mathbf{n}$. By assumption, one has that
	\begin{itemize}
		
		\item $\sec_M({\gamma}'_{\mathbf{n}}(t),\cdot)\leq \lambda$;
		
		\smallskip
		
		\item  the eigenvalues $\{\kappa_i\}_{i=1}^n$ of $\mathfrak{A}^{\mathbf{n}}$ with respect to some orthonormal basis of $T_{\pi(\mathbf{n})}\Sigma$ satisfy $\kappa_i\leq \kappa$ for all $1\leq i\leq n$.
		
	\end{itemize}
	
%
%
%
%

		In the sequel, we   keep the index range of
		\[
		1\leq i,j,k,...\leq n,\quad   n+1\leq A,B,...\leq m-1, \quad 1\leq \alpha,\beta, ...\leq m-1.
		\]
		For convenience, let $\gamma:=\gamma_{\mathbf{n}}$ and  $o:=\pi(\mathbf{n})\in \Sigma$. Let
		$\{\mathbf{n},e_\alpha\}_{\alpha=1}^{m-1}$ be an orthonormal basis for $T_oM$ such that $\{e_i\}_{i=1}^n\subset T_o\Sigma$ are the eigenvectors  with $\mathfrak{A}^{\mathbf{n}}(e_i)=\kappa_i e_i$. Denote by $E_\alpha(t):=P_{t;\mathbf{n}}e_\alpha$ the parallel vector field along $\gamma$ for $1\leq \alpha\leq m-1$. The rest of the proof is divided into three steps.
		
		\smallskip
		
		{\it Step 1}. Given $t_0\in (0,\min\{c_\mathcal {V}(\mathbf{n}),\mathfrak{t}_{\lambda,\kappa}\})$,  we construct a family of transverse Jacobi fields $\{J_\alpha\}$  along $\gamma$.
		For each $1\leq i\leq n$ and $n+1\leq A\leq m-1$, it follows from  \eqref{constructionofJacobifile}, Lemma  \ref{Fermifirstlemma} and the standard theory of Jacobi fields (or ODE) that $Y_i(t)=P_{t;\mathbf{n}}\mathcal {A}(t,\mathbf{n})e_i$ is a unique transverse Jacobi field  along $\gamma$ with
		\begin{equation}\label{Yicontrol}
			Y_i(0)=e_i,\quad \na_{{\gamma}'}Y_i(0)=-\mathfrak{A}^{\mathbf{n}}(e_i)=-\kappa_i e_i,
		\end{equation}
		while $Y_A(t):=P_{t;\mathbf{n}}\mathcal {A}(t,\mathbf{n})e_A$ is the unique transverse Jacobi field  along $\gamma$ with
		\begin{equation}\label{YAcontrol}
			Y_A(0)=0,\quad \nabla_{{\gamma}'}Y_A(0)=e_A.
		\end{equation}
		Since $\{{\gamma}', Y_\alpha\}$ is a frame field along $t\mapsto \gamma(t)$ for $t\in (0,t_0)$, see Chavel \cite[Section 7.3]{Cha}, there exists a non-singular matrix $(f_\beta^\alpha(t_0))$ such that  $E_\alpha(t_0)=f^\beta_\alpha(t_0) Y_{\beta}(t_0)$. Now we can define a family of transverse Jacobi fields
		\begin{equation}\label{Jbetaconstr}
			J_\alpha(t):=f^\beta_\alpha(t_0)Y_\beta(t),\quad 0\leq t\leq t_0.
		\end{equation}
		It follows from Lemma \ref{INDEX-CRITERIA-modified-version}  that $J_\alpha$ is the unique transverse Jacobi field with $J_\alpha(t_0)=E_\alpha(t_0)$.

		\smallskip

		{\it Step 2}. We construct a ``comparing submanifold" in the space form. To do this, let $\MM_\lambda$ be the space form as in Lemma \ref{umbilical-lem}.
		Fixing a point $\tilde{o}\in \MM_\lambda$, it follows by
		Lemma \ref{umbilical-lem} and Remark \ref{principlecurvauteform1} that there is
		an $n$-dimensional totally umbilical submanifold $\widetilde{\Sigma}$ in $\MM_\lambda$ such that $\widetilde{\mathfrak{A}}^{\tilde{\mathbf{n}}}=\kappa\id_{T_{\tilde{o}}\widetilde{\Sigma}}$ for some $\tilde{\mathbf{n}}\in \mathcal {V}_{\tilde{o}}S\widetilde{\Sigma}$.
		Take an orthonormal basis $\{\tilde{\mathbf{n}},\tilde{e}_1,\ldots,\tilde{e}_{m-1}\}$ for  $T_{\tilde{o}}\MM_\lambda$, where $\{\tilde{e}_1,\ldots,\tilde{e}_{n}\}$ is an orthonormal basis for $T_{\tilde{o}}\widetilde{\Sigma}$. Clearly, $\widetilde{\mathfrak{A}}^{\tilde{\mathbf{n}}}(\tilde{e}_i)=\kappa \tilde{e}_i$ for $1\leq i\leq n$.
		Let $\tilde{\gamma}$ be the geodesic in $\MM_\lambda$ with ${\tilde{\gamma}}'(0)=\tilde{\mathbf{n}}$.
		Denote by $\widetilde{E}_\alpha$ the parallel vector field  along $\tilde{\gamma}$ with $\widetilde{E}_\alpha(0)=\tilde{e}_\alpha$ for $1\leq \alpha\leq m-1$.

		Now consider the transverse  Jacobi fields $\tJ_\alpha(t)$, $0\leq t\leq t_0$ along $\tilde{\gamma}$, see \cite[p.~321]{Cha}:
		\begin{equation}
			\begin{cases} \tJ_i(t):={\mathbf{c}_\lambda(t)-\kappa \mathbf{s}_\lambda(t)\over \mathbf{c}_\lambda(t_0)-\kappa \mathbf{s}_\lambda(t_0)}\widetilde{E}_i(t),&  1\leq i\leq n,\\
				\\
				\tJ_A(t):={\mathbf{s}_\lambda(t)\over \mathbf{s}_\lambda(t_0)}\widetilde{E}_A(t),&  n+1\leq A\leq m-1.\end{cases}\label{Jacobi-fields-tM}
		\end{equation}
		Thus, $\tJ_\alpha(t)$'s  are well-defined for $t\in [0,t_0]$ since $t_0<\mathfrak{t}_{\lambda,\kappa}$. In what follows, we will use $\tilde{*}$ to denote the geometric quantity $*$ in the space form $\MM_\lambda$.
		
		\smallskip

		{\it Step 3}. In this step we complete the proof.
			On the one hand,
		in $(M,g)$  we express $J_\alpha(t)=:J^\beta_\alpha(t)E_\beta(t)$. Since $E_\alpha(t_0)=J_\alpha(t_0)=J^\beta_\alpha(t_0)E_\beta(t_0)$, we have $J^\beta_\alpha(t_0)=\delta^\beta_\alpha$. Furthermore, $J^A_\alpha(0)=0$ follows by \eqref{Yicontrol}--\eqref{Jbetaconstr} directly.
		On the other hand,
		construct a family of smooth vector fields in $\MM_\lambda$ along $\tilde{\gamma}$ by
		\begin{equation}\label{Jacobi-fields-tM-trans}
			\widetilde{X}_\alpha(t):=J^\beta_\alpha(t) \widetilde{E}_\beta(t).
		\end{equation}
		It is easy to see that
		\[
		\widetilde{X}_\alpha(0)= J^i_\alpha(0) \tilde{e}_i\in  T_{\tilde{o}}\widetilde{\Sigma},\quad     \widetilde{X}_\alpha(t_0)=\widetilde{E}_\alpha(t_0)=\tJ_\alpha(t_0).
		\]
		Thus Lemma \ref{INDEX-CRITERIA-modified-version} together with the equation of Jacobi field yields
		\begin{align}
			\tilde{I}_{t_0}(\widetilde{X}_\alpha,\widetilde{X}_\alpha)&\geq \tilde{I}_{t_0}(\tJ_\alpha,\tJ_\alpha)=-\langle \widetilde{\mathfrak{A}}^{\tilde{\mathbf{n}}}(\tJ_\alpha(0)), \tJ_\alpha(0)  \rangle+\langle \tJ_\alpha,\widetilde{\nabla}_{{\tilde{\gamma}'}}\tJ_\alpha  \rangle|^{t_0}_0\notag\\
			&=\langle \tJ_\alpha(t_0),\widetilde{\nabla}_{{\tilde{\gamma}'}}\tJ_\alpha (t_0) \rangle-\langle \tJ_\alpha(0),\widetilde{\nabla}_{{\tilde{\gamma}'}}\tJ_\alpha (0)+\widetilde{\mathfrak{A}}^{\tilde{\mathbf{n}}}(\tJ_\alpha(0)) \rangle=\langle \tJ_\alpha(t_0),\widetilde{\nabla}_{{\tilde{\gamma}'}}\tJ_\alpha (t_0) \rangle,\label{lowerhiicontroll}
		\end{align}
		where the last equality follows by the definition of transverse Jacobi field, see \eqref{deftransver}.
		On the other hand, by recalling that $\kappa_i\leq \kappa$ and $J_\alpha$'s are transverse Jacobi fields, we obtain
		\begin{align*}
			\langle \widetilde{X}_\alpha(0),\widetilde{\nabla}_{{\tilde{\gamma}'}}\widetilde{X}_\alpha (0)+\widetilde{\mathfrak{A}}^{\tilde{\mathbf{n}}}(\widetilde{X}_\alpha(0)) \rangle&=\sum_{i=1}^nJ^i_\alpha(0)(({J}^i_\alpha)'(0)+\kappa J^i_\alpha(0))\geq \sum_{i=1}^nJ^i_\alpha(0)(({J}^i_\alpha)'(0)+\kappa_i J^i_\alpha(0))\\
			&=\langle J_\alpha(0),{\nabla}_{{{\gamma}'}}J_\alpha (0)+{\mathfrak{A}}^{{\mathbf{n}}}(J_\alpha(0)) \rangle=0.
		\end{align*}
		
		A similar calculation combined with the construction of $\widetilde{X}_\alpha$ and $\sec_M({\gamma}'_{\mathbf{n}}(t),\cdot)\leq \lambda$ furnishes
		\begin{align*}
			\tilde{I}_{t_0}(\widetilde{X}_\alpha,\widetilde{X}_\alpha)=&\langle \widetilde{X}_{\alpha}(t_0),\widetilde{\nabla}_{{\tilde{\gamma}'}}\widetilde{X}_{\alpha} (t_0) \rangle-\langle \widetilde{X}_{\alpha}(0),\widetilde{\nabla}_{{\tilde{\gamma}'}}\widetilde{X}_{\alpha} (0)+\widetilde{\mathfrak{A}}^{\tilde{\mathbf{n}}}(\widetilde{X}_{\alpha}(0)) \rangle\\
&-\int^{t_0}_0  \left(\langle \widetilde{\nabla}_{{\tilde{\gamma}'}}\widetilde{\nabla}_{{\tilde{\gamma}'}}\widetilde{X}_{\alpha} , \widetilde{X}_\alpha \rangle +\lambda |{\tilde{\gamma}'}\wedge \widetilde{X}_{\alpha}|^2\right) {\dd}t\\
			\leq & \langle {J_{\alpha}}(t_0),{\nabla}_{{{\gamma}'}}{J_{\alpha}} (t_0) \rangle-\int^{t_0}_0  \left(\langle {\nabla}_{{{\gamma}'}}{\nabla}_{{{\gamma}'}}{J_{\alpha}} , J_\alpha \rangle + \langle R_{\gamma'}(J_\alpha),J_\alpha \rangle\right) {\dd}t= \langle {J_{\alpha}}(t_0),{\nabla}_{{{\gamma}'}}{J_{\alpha}} (t_0) \rangle,
		\end{align*}
		which together with \eqref{lowerhiicontroll} and a standard calculation, see Chavel \cite[p.~325]{Cha}, yields
		\begin{equation}\label{controlldetlap}
			\left.\frac12{(\det \mathcal {A}(t,\Bn))'\over \det \mathcal {A}(t,\Bn)}\right|_{t=t_0}=\sum_{\alpha=1}^{m-1}\langle {J_{\alpha}}(t_0),{\nabla}_{{{\gamma}'}}{J_{\alpha}} (t_0) \rangle\geq \sum_{\alpha=1}^{m-1}\langle \tJ_\alpha(t_0),\widetilde{\nabla}_{{\tilde{\gamma}'}}\tJ_\alpha (t_0) \rangle= \left.\frac12{(\det \widetilde{\mathcal {A}}(t,\tilde{\Bn}))'\over \det \widetilde{\mathcal {A}}(t,\tilde{\Bn})}\right|_{t=t_0},
		\end{equation}
		where $\det \widetilde{\mathcal {A}}(t,\tilde{\Bn}):=(\mathbf{c}_\lambda(t)-\kappa\mathbf{s}_\lambda(t))^n\,\mathbf{s}_\lambda(t)^{m-n-1}$. Therefore, \eqref{Laplace-comparison-reverse2}} follows by \eqref{Laplacian-expression} and \eqref{controlldetlap}, while  \eqref{smallsestimatedatA} combined with \eqref{controlldetlap} implies   \eqref{det-comparison-reverse2}.
\hfill $\square$\\


We conclude this subsection by recalling the totally geodesic submanifolds in space forms, see Cannon et al. \cite{CFKP},  Dajczer and Tojeiro \cite{DT} and  O'Neil\cite{ONeil}.

\begin{example}\label{space-form-phi-psi1} Let $\mathbb{M}_\lambda$ be an $m(\geq 2)$-dimensional simply connected complete Riemannian manifold with constant sectional curvature $\lambda\in \{-1,0,1\}$ and let  $i:\Sigma\hookrightarrow \MM_\lambda$ be an $n$-dimensional {totally geodesic} submanifold with $1\leq n\leq m-1$.
Thus,
\begin{itemize}

\item[(i)] if $\lambda=1$, then $\mathbb{M}_\lambda=\mathbb{S}^m$  and $\Sigma$ is a great $n$-sphere $\mathbb{S}^n$; in particular when $n=m-1$,
\[
c_{\mathcal{V}}(\Bn)=\mathfrak{c}_{\lambda,0}=\mathfrak{r}_{\lambda}=\pi/2,\  \forall\,\mathbf{n}\in \mathcal {V}S\Sigma,\qquad \mathcal{V}\mathscr{C}\Sigma=\{\text{antipodal points w.r.t. }\Sigma\};
\]

\smallskip

\item[(ii)] if $\lambda=0$, then $\mathbb{M}_\lambda=\mathbb{R}^m$ and $\Sigma$ is an affine $n$-plane. Thus,
\begin{equation}\label{spaceformcnisinfite}
 c_{\mathcal{V}}(\Bn)=\mathfrak{c}_{\lambda,0}=\mathfrak{r}_{\lambda}=\infty,\  \forall\,\mathbf{n}\in \mathcal {V}S\Sigma,\qquad \mathcal{V}\mathscr{C}\Sigma=\emptyset;
\end{equation}

\smallskip

\item[(iii)] if $\lambda=-1$, then $\mathbb{M}_\lambda=\mathbb{H}^m$ is the hyperbolic space (modeled on $\mathbb{R}^m_+$) and $\Sigma$ is
\begin{itemize}
\item[$\bullet$] either an upper half affine $n$-subspace vertical to the hyperplane $x^m=0$;

\item[$\bullet$] or a half-$n$-sphere with an arbitrary radius that intersect orthogonally the hyperplane $x^m=0$;
\end{itemize}
in particular, \eqref{spaceformcnisinfite} remains valid in both cases.


\end{itemize}

In all these cases, due to Chavel \cite[p.~321]{Cha}, for any $t\in (0,\mathfrak{t}_{\lambda,0})=(0,\mathfrak{r}_{\lambda})$ and $\mathbf{n}\in \mathcal {V}S\Sigma$, one has
\begin{align}
\det \mathcal {A}(t,\Bn)=\mathbf{c}_\lambda(t)^n\mathbf{s}_\lambda(t)^{m-n-1},  \quad \Delta r|_{(t,\Bn)}=G_{\lambda,0}(t)=-\lambda n{ \mathbf{s}_\lambda(t)\over  \mathbf{c}_\lambda(t)}+(m-n-1){\mathbf{c}_\lambda(t)\over \mathbf{s}_\lambda(t)},\label{narspaceform}
 \end{align}
 which coincide with the conclusions of Theorems \ref{Laplacian-comparison-Sigma} $\&$ \ref{main-rem-2} in their limit cases.
\end{example}

\subsection{Weak divergence}\label{section-weak}
In our approach, weak divergence plays an important role in establishing Hardy inequalities; first we recall its definition, see   D'Ambrosio \cite{DA} and  D'Ambrosio and Dipierro \cite{DD}:

\begin{definition} Given an open subset $U\subset M$, let $L^1_{\lo}(TU)$ and $L^1_{\lo}(U)$ denote the locally integrable vector fields and functions on $U$ (w.r.t. $\vol_g$), respectively. That is, for any compact set $\mathscr{K}\subset U$,
\[
\int_\mathscr{K} |X|\dvol_g <\infty, \ \forall\, X\in L^1_{\lo}(TU); \quad \int_\mathscr{K} |f|\dvol_g <\infty, \ \forall\, f\in L^1_{\lo}(U).
\]
Given a vector field
$X\in L^1_{\lo}(TU)$ and  a non-negative function $f_X\in L^1_{\lo}(U)$, we say   $f_X\leq \di X$ {\it in the weak sense} if
\[
 \int_{U} u f_X {\dvol}_g\leq  -\int_{U} \langle \nabla u, X \rangle {\dvol}_g, \quad \forall \,u\in C^\infty_0(U) \text{ with }u\geq 0.
 \]
\end{definition}

Proceeding as in  D'Ambrosio \cite[Theorem 2.5]{DA} or D'Ambrosio and  Dipierro \cite[Lemma 2.10]{DD}, we have the following result.
\begin{lemma}\label{divlemf}Let $U$ be an open set of $M$.
Given $p>1$, if a vector field $X\in L^1_{\lo}(T{U})$    and a non-negative function $f_X\in L^1_{\lo}({U})$ satisfy $|X|^p/f_X^{p-1}\in L^1_{\lo}({U})$ and $f_X\leq \di X$ in the weak sense, then
\begin{align*}
p^p \ds\ds\int_{U}|\nabla u|^p\frac{|X|^p}{f_X^{p-1}} {\dvol}_g    \geq \ds\ds\int_{U} |u|^p f_X{\dvol}_g,\quad \forall\,u\in C^\infty_0({U}).
\end{align*}
\end{lemma}

In order to adapt to our setting,  we extend the weak divergence to Lipschitz functions, see Proposition \ref{generadiverge}.
%
For convenience, we use $\Lip(V)$ (resp., $\Lip_{\lo}(V))$ to denote the collection of Lipschitz (resp., locally Lipschitz) functions on $V$. Moreover, let $\Lip_0(V):=C_0(V)\cap \Lip(V)$, i.e., the collection of Lipschitz functions on $V$ with compact support. Given a Lipschitz function $f$,
the minimal $C>0$ satisfying the  inequality
\[
|f(x)-f(y)|\leq C d(x,y),\quad \forall\,x,y\in \text{ the domain of $f$},
\]
is called the {\it dilatation} of $f$, denoted by $\dil(f)$, see  Burago,  Burago and Ivanov \cite{DYS}.

\begin{lemma}\label{localtogloaLip} Let $U\subset M$ be an open set. Then $C^\infty_0(U)\subset \Lip_0(U)$ and $hu\in \Lip_0(U)$ for every   $h\in \Lip_{\lo}(U)$ and $u\in C^\infty_0(U)$.
\end{lemma}

The proof of Lemma \ref{localtogloaLip} directly follows from the definitions, thus we omit it. By Rademacher's theorem and local coordinate systems, we know that every Lipschitz function is differentiable $\vol_g$-a.e. Moreover, we have the following approximation result:

\begin{lemma}\label{lipsconverppax} Let $U\subset M$ be an open set and let $w:U\rightarrow \mathbb{R}$ be a Lipschitz function with compact support. Then  there exists a sequence of  Lipschitz functions $w_n\in C^\infty_0(U)\cap \Lip_0(U)$  such that
\[
\supp w_n\subset \mathscr{K},\quad |w_n|\leq C,\quad |\nabla w_n|\leq   \dil(w_n)\leq C,\quad w_n\rightrightarrows w,\quad \|w_n-w \|_{W^{1,1}(U)}\rightarrow 0,
\]
where $\mathscr{K}$ is a compact set in $U$ and $C>0$ is  a constant. In particular, if $w$ is nonnegative, so are $w_n$'s.
\end{lemma}

The proof of Lemma \ref{lipsconverppax} is based on standard approximation arguments, combining convolutions of suitable mollifiers with smooth partition of unity; therefore, we omit its proof.

The following proposition plays a crucial role in the proof of our main general result, see Theorem \ref{Hardy-thm-general}.

\begin{proposition}\label{generadiverge} Let $f_X\in L^1_{\loc}(U)$ and $X\in L^1_{\loc}(TU)$ satisfying $f_X\leq  \dii X$ in the weak sense. Then
 for every nonnegative
 $h\in \Lip_{\lo}(U)$
we have
\begin{equation}\label{lipdivergence}
\int_{U} (h u)\, f_X \dvol_g\leq -\int_{U} \langle \nabla ( h u), X   \rangle\dvol_g,\quad \forall\,u\in C^\infty_0(U) \text{ with }u\geq 0.
\end{equation}
\end{proposition}
\begin{proof}  Lemma \ref{localtogloaLip} implies that $w=hu\in \Lip_0(U);$ thus, it suffices to show
	\begin{equation}\label{stronglipdiver}
		\int_{U} w f_X {\dvol}_g\leq  -\int_{U} \langle\nabla w, X \rangle {\dvol}_g, \quad \forall \,w\in \Lip_0(U) \text{ with }w\geq 0.
	\end{equation}
If $w\in \Lip_0(U)$ is a nonnegative function,  Lemma
\ref{lipsconverppax} yields a sequence of  functions $w_n\in C^\infty_0(U)\cap \Lip_0(U)$ satisfying
\begin{equation}\label{boundedcontroll}
\supp w_n\subset \mathscr{K},\quad 0\leq w_n\leq C,\quad |\nabla w_n|\leq \dil(w_n)\leq C,\quad w_n\rightrightarrows w,\quad \|w_n-w\|_{W^{1,1}(U)}\rightarrow 0,
\end{equation}
 for some  compact $\mathscr{K}\subset U$ and constant $C\geq 0$.
Thus,  by passing to a subsequence, we may assume that $w_n\rightarrow w$ and $\nabla w_n \rightarrow  \nabla w$ point-wisely for $\vol_g$-a.e. Since $f_X\leq \dii X$ in the weak sense, we have
 \begin{equation}\label{weakldivevncon}
 \int_{U} w_n f_X \dvol_g\leq -\int_{U} \langle \nabla w_n, X   \rangle\dvol_g.
 \end{equation}
Note that \eqref{boundedcontroll} and the assumption imply that
\[
|w_n f_X|\leq \chi_{\mathscr{K}} C|f_X|\in L^1(U), \quad |\langle \nabla w_n, X   \rangle|\leq  \chi_{\mathscr{K}} C|X|\in L^1(U),
\]
where $\chi_{\mathscr{K}}$ stands for the characteristic function of $\mathscr{K}$.
Therefore  \eqref{stronglipdiver} follows by \eqref{weakldivevncon} and the dominated convergence theorem.
\end{proof}

\section{A general Hardy inequality}\label{basichardy}

In this section, we will   establish a   Hardy inequality involving a pair of parametric functions.

\subsection{Abstract construction}\label{section-construction}
Let   $M$ be an $m$-dimensional complete Riemannian manifold, $i: \Sigma\hookrightarrow M$ be an $n$-dimensional closed submanifold, and $r$ be the distance function from $\Sigma$. Inspired by the seminal paper of Barbatis, Filippas and Tertikas \cite{BFT} -- where the ambient space is the Euclidean space $\mathbb R^m$  --  throughout the present paper,  $\Omega\subset M$ is always a non-empty domain (i.e., a connected open set)  satisfying
\begin{equation}
 \text{(a)\, $\Sigma=\pa\Omega$ if $n=m-1$; \qquad (b)\, $\Sigma\cap \Omega\neq \emptyset$ if $0\leq n\leq m-2$.}\label{natural-domain}
 \end{equation}
For convenience, set $\Omega_\Sigma:=\Omega\setminus\Sigma$.
Now
 we introduce the following structural assumption.

\begin{assumption}\label{Basciasumptionfunction} 
Given $p>1$, let $(\phi,\psi)$   be a pair of two nonnegative $C^1$-functions defined on  some subinterval of $(0, +\infty)$  such that
\begin{enumerate}[{\rm(a)}]
 \item\label{smoothofphi} $\phi(r),\psi(r)$ are well-defined   in $\Omega_\Sigma$;

\smallskip

\item\label{smoothofpsi}  $\psi$ is monotone and  $\psi(r)$ is locally Lipschitz on $\Omega_\Sigma$;

\smallskip

\item\label{nansigcontroll} {$\ds\dii(\sgn(\psi'(r))\,\phi(r)^{p-1}\na r)\geq 0$} in the weak sense in $\Omega_\Sigma$;

\smallskip

\item\label{integrable} $\phi(r)^{p-1}, \phi(r)^{p-1}\psi(r),\phi(r)^{p-1}|\psi'(r)|,\frac{\phi(r)^{p-1}|\psi'(r)|}{ |\nabla(\log \psi(r))|^{p}}\in L^1_{\loc}(\Omega_\Sigma)$;

\smallskip

\item\label{negligbileset} $\{x\in \Omega_\Sigma\,:\, \phi(r(x))\,\psi(r(x))\,\psi'(r(x))=0\}$ is a $\vol_g$-negligible set.



\end{enumerate}
\end{assumption}

\begin{remark}\label{assup1explain} 

Assumption \eqref{smoothofphi} implies that $\phi'(r),\psi'(r)$ are well-defined   in $\Omega_\Sigma$. We notice that there is a close relation between $\dii(\phi(r)^{p-1}\na r)$ and $p$-Laplacian. In fact, let $\Phi$ be the primitive function of $\phi$ (i.e., $\Phi'=\phi$), which is usually locally defined. Then $\dii(\phi(r)^{p-1}\na r)$ is exactly the $p$-Laplacian $\Delta_p\Phi(r):=\dii(|\na \Phi(r)|^{p-2}\na \Phi(r))$.
Since $\psi$ is monotone, then
   \eqref{nansigcontroll} implies that $\Delta_p\Phi(r)\geq 0$ or $\Delta_p\Phi(r)\leq 0$ in the weak sense.
In particular,  a direct calculation yields
\begin{equation}\label{plaplacephi}
\dii(\sgn(\psi'(r))\,\phi(r)^{p-1}\na r)=\sgn(\psi'(r))\,\phi(r)^{p-1}\left(  \Delta r+(p-1)\frac{\phi'(r)}{\phi(r)}  \right) \quad \text{ on }\mathcal {V}\mathscr{D}\Sigma\backslash\Sigma.
\end{equation}
Since $M\backslash \mathcal {V}\mathscr{D}\Sigma=\Sigma\cup \mathcal {V}\mathscr{C}\Sigma$ is a $\vol_g$-negligible set, one may use the Laplacian comparison (Theorems \ref{Laplacian-comparison-Sigma}--\ref{main-rem-2}) to determine the sign of $\dii(\sgn(\psi'(r))\,\phi(r)^{p-1}\na r)$ directly.

We notice that the structural assumptions \eqref{smoothofphi}--\eqref{integrable} are enough to establish Hardy inequalities, see e.g. \eqref{Hardy-Inequ-general} below;    the role of \eqref{negligbileset}  is to make  \eqref{Hardy-Inequ-general} nontrivial. Moreover,  \eqref{negligbileset} is useful to investigate the non-existence of extremal functions in Hardy-type inequalities, see Section \ref{nonexsorigin}.
\end{remark}

Infinitely many pairs of $(\phi,\psi)$ can be produced on space forms; the following example provides such constructions.

\begin{example}\label{space-form-phi-psi2} Let $\mathbb{M}_\lambda$ be  the space form and $\Sigma\subset M$ be the totally geodesic submanifold as in Example \ref{space-form-phi-psi1}.
Choose a domain $\Omega\subset \MM_\lambda$ such that
\begin{itemize}
\item    $\Omega=\mathbb{S}^m_+$ (i.e., the upper hemisphere) and  $\Sigma=\pa\Omega$ (i.e., $n=m-1$) whenever $\lambda=1$;

\item  $\Omega\cap \Sigma\neq\emptyset$ whenever $\lambda\in \{-1,0\}$.

\end{itemize}
In view of \eqref{r-0-Def}, given $p>1$ and $\beta< -(m-n)$, the pair  $(\phi ,\psi )$ is defined as
\begin{align}
 \phi(t):=\mathbf{c}_\Lambda(t)^{-{n\over p-1}}\mathbf{s}_\Lambda(t)^{-{m-n-1\over p-1}},\quad   \psi(t):=\mathbf{s}_\Lambda(t)^{\beta+m-n},\quad \forall\,t\in (0,\mathfrak{r}_{\Lambda}),\label{phi-psi-M-lambda}
 \end{align}
 where $\Lambda\leq \lambda$ and particularly, $\Lambda<\lambda$ if $\lambda >0$.
 Then Assumption \ref{Basciasumptionfunction} is satisfied. 

Indeed, since $\mathfrak{r}_{\lambda}=c_{\mathcal{V}}(\Bn)$ for every $\mathbf{n}\in \mathcal {V}S\Sigma$, see Example \ref{space-form-phi-psi1}, we see that $\phi(r),\psi(r),\psi'(r)$ are well-defined and nonzero in $\Omega_\Sigma$. Therefore, Assumption \ref{Basciasumptionfunction}/\eqref{smoothofphi}\eqref{smoothofpsi}\eqref{integrable}\eqref{negligbileset} are true.
Finally, \eqref{plaplacephi} together with \eqref{narspaceform} and Proposition \ref{gestiamtegra}  yield
 \begin{equation}\label{divdireccalution}
\dii (\sgn(\psi'(r))\phi(r)^{p-1}\na r)=-\phi(r)^{p-1}\left(  \Delta r+(p-1)\frac{\phi'(r)}{\phi(r)}  \right)=-\phi(r)^{p-1}\left(  G_{\lambda,0}(r)-G_{\Lambda,0}(r) \right)\geq 0
 \end{equation}
in $\mathcal {V}\mathscr{D}\Sigma\backslash \Sigma$. Thus Assumption \ref{Basciasumptionfunction}/\eqref{nansigcontroll} is satisfied as well.
\end{example}
\begin{remark}\label{basiclambda=Lambdacase}
Note that if $\Lambda=\lambda\leq 0$, then \eqref{phi-psi-M-lambda} satisfies Assumption \ref{Basciasumptionfunction} for all $\beta\neq -(m-n)$.
\end{remark}

Under Assumption \ref{Basciasumptionfunction}, we have the following generic Hardy inequality.

\begin{theorem}\label{Hardy-thm-general} Let $(M,g)$ be an $m$-dimensional complete Riemannian manifold, $i:\Sigma \hookrightarrow M$ be an $n$-dimensional closed submanifold  and $\Omega$ be a non-empty domain in $M$ satisfying \eqref{natural-domain}.  Let $r(x):=d(\Sigma,x)$ be the distance function from $\Sigma$.  Given $p>1$, if the pair $(\phi,\psi)$ satisfies Assumption \ref{Basciasumptionfunction},  then
\begin{equation}
\int_{\Omega_\Sigma}|\na u|^p\frac{\phi(r)^{p-1}|\psi'(r)|}{ |\nabla(\log \psi(r))|^{p}}\dvol_g\geq {p^{-p}}\int_{\Omega_\Sigma} |u|^p\phi(r)^{p-1}|\psi'(r)|\dvol_g\label{Hardy-Inequ-general}
\end{equation}
 for all $u\in C^\infty_0(\Omega_\Sigma)$.
\end{theorem}

\begin{proof} 

Without loss of generality, we may consider any function  $u\in C^\infty_0(\Omega_\Sigma)$ with $u\geq 0$. Then, by Assumption \ref{Basciasumptionfunction}/\eqref{smoothofpsi}\eqref{nansigcontroll}\eqref{integrable}  and \eqref{lipdivergence}  (choosing $U=\Omega_\Sigma$, $f_X=0$, $X=\sgn(\psi'(r))\phi(r)^{p-1}\na r$ and $h=\psi(r)$), we have
\[
0\leq -\int_{\Omega_\Sigma} \langle \nabla(\psi(r)u), \sgn(\psi'(r))\phi(r)^{p-1}\nabla r  \rangle\dvol_g,
\]
which combined with $|\psi'(r)|=\sgn(\psi'(r))\psi'(r)$ and the eikonal equation $|\nabla r|=1$ a.e. on $\Omega_\Sigma$ yields
\begin{align*}
\int_{\Omega_\Sigma} u\phi(r)^{p-1}|\psi'(r)| \dvol_g\leq -\int_{\Omega_\Sigma} \la \na u,\sgn(\psi'(r))\phi(r)^{p-1}\psi(r)\na r\ra \dvol_g.
\end{align*}
Hence,  $\phi(r)^{p-1}|\psi'(r)|\leq {\rm div }(\sgn(\psi'(r))\phi(r)^{p-1}\psi(r)\na r)$ in the weak sense.
Now by considering
\[
U=\Omega_\Sigma, \quad f_X=\phi(r)^{p-1}|\psi'(r)|, \quad X=\sgn(\psi'(r))\phi(r)^{p-1}\psi(r)\na r,
\]
inequality \eqref{Hardy-Inequ-general} follows at once by Lemma \ref{divlemf} and Assumption \ref{Basciasumptionfunction}/\eqref{integrable}.
\end{proof}

\subsection{Sharpness}\label{sharpnessfirst}

Let $U$ be an open set of $\Omega$, $u\in C^\infty_0(U)\setminus \{0\}$ and for a pair $(\phi,\psi)$ satisfying Assumption \ref{Basciasumptionfunction}, define
\begin{equation}\label{functinoalshapr}
	\mathcal {J}_{p,U}(u):={\ds\int_{U}|\na u|^p \frac{\phi(r)^{p-1} |\psi'(r)|}{|\nabla(\log \psi(r))|^{p}}\dvol_g\over \ds\int_{U}|u|^p \phi(r)^{p-1}|\psi'(r)|\dvol_g }.
\end{equation}
Due to Theorem \ref{Hardy-thm-general},  it is meaningful to investigate the  sharpness of  (\ref{Hardy-Inequ-general}), i.e., to study the possibility of
\begin{equation}\label{sharpfunctionexp}
	\inf_{u\in C^\infty_0(\Omega_\Sigma)\setminus \{0\}}\mathcal {J}_{p,\Omega_\Sigma}(u)=p^{-p}.
\end{equation}
In order to do this, we recall  the $s$-tuber neighborhood
$\T_{s}$  defined by \eqref{stuber}. Thus,
\begin{equation}\label{domaincontts}
	{\T_{s}}\cap \, \Omega=\{x\in \Omega\,:\, r(x)<s\},\quad \Omega\,\backslash {\T_{s}}=\{x\in \Omega\,:\, r(x)\geq s\}.
\end{equation}
Inspired by D'Ambrosio \cite{DA}, we have the following result.

\begin{lemma}\label{Ambrosio-lem}
	Let $(M,g)$, $\Sigma$, $\Omega$ and $(\phi,\psi)$  be as in Theorem \ref{Hardy-thm-general}. Assume that
	\begin{enumerate}[{\rm(i)}]
		\item there exists some $s_0>0$ such that  $\T_{s_0}\cap \, \Omega$ and $\Omega\,\backslash {\T_{s_0}}$ have piece-wise regular boundaries$;$
		
		\smallskip
		
		\item there exists $\e_1>0$ such that  for any $\varepsilon\in (0,\varepsilon_1)$,
		\begin{equation}
			\psi(r)^{1+\varepsilon}\phi(r)^{p-1}|\psi'(r)|\in L^1({\T_{s_0}}\cap\,\Omega),\quad \psi(r)^{-(1+\varepsilon)}\phi(r)^{p-1}|\psi'(r)|\in L^1(\Omega\backslash{\T_{s_0}}).\label{Ambrosio-condition-1}
		\end{equation}
		
	\end{enumerate}
	
	\noindent Then for each $\varepsilon\in (0,\varepsilon_1)$, by setting $c(\varepsilon):=(1+\varepsilon)/p$, the function
	\begin{align}
		\nu_\e(x):=\begin{cases}
			\Big({\psi(r(x))\over \psi(s_0)}\Big)^{c(\e)},& \text{ if }x\in \T_{s_0}\cap \,\Omega,\\
			\\
			\Big({\psi(r(x))\over \psi(s_0)}\Big)^{-c(\e/2)},& \text{ if }x\in \Omega\,\backslash \T_{s_0},
		\end{cases} \label{v-Def-1}
	\end{align}
	satisfies the following inequality
	\begin{equation}
		\int_{\Omega_\Sigma}|\na \nu_\e|^p \frac{\phi(r)^{p-1}|\psi'(r)|}{ |\nabla(\log \psi(r))|^{p}} \dvol_g<c(\e)^p\int_{\Omega_\Sigma} |\nu_\e|^p\phi(r)^{p-1}|\psi'(r)|\dvol_g.\label{inverse-ineq}
	\end{equation}
\end{lemma}

\begin{proof} Owing to \eqref{domaincontts}, \eqref{Ambrosio-condition-1} and Assumption \ref{Basciasumptionfunction}/\eqref{negligbileset},  a direct calculation yields
	\begin{align}
		&c(\e)^p\int_{\Omega_\Sigma} |\nu_\e|^p\phi(r)^{p-1}|\psi'(r)|\dvol_g\nonumber\\
		=&|\psi(s_0)|^{-c(\e)p}c(\e)^p\int_{r<s_0} |\psi(r)^{c(\e)}|^p\phi(r)^{p-1}|\psi'(r)|\dvol_g\nonumber\\
		&+|\psi(s_0)|^{c(\e/2)p}c(\e)^p\int_{r\geq s_0} |\psi(r)^{-c(\e/2)}|^p\phi(r)^{p-1}|\psi'(r)|\dvol_g\nonumber\\
		=&\int_{r<s_0} |\na \nu_\e|^p\frac{\phi(r)^{p-1}|\psi'(r)|}{ |\nabla(\log \psi(r))|^{p}} \dvol_g+{c(\e)^p\over c(\e/2)^p}\int_{r\geq s_0} |\na \nu_\e|^p \frac{\phi(r)^{p-1}|\psi'(r)|}{ |\nabla(\log \psi(r))|^{p}} \dvol_g,\nonumber
	\end{align}
	which combined with $c(\e)/c(\e/2)>1$ yields  \eqref{inverse-ineq}.
\end{proof}

The same argument also yields a companion version of Lemma \ref{Ambrosio-lem}.
\begin{lemma}\label{Ambrosio-lem2}
	Let $(M,g)$, $\Sigma$, $\Omega$ and $(\phi,\psi)$  be as in Theorem \ref{Hardy-thm-general}. Assume that
	\begin{enumerate}[{\rm(i)}]
		\item there exists some $s_0>0$ such that  $\T_{s_0}\cap \, \Omega$ and $\Omega\,\backslash {\T_{s_0}}$ have piecewise regular boundaries$;$
		
		\smallskip
		
		\item there exists $\e_1>0$ such that  for any $\varepsilon\in (0,\varepsilon_1)$,
		\begin{equation}
			\psi(r)^{-(1+\varepsilon)}\phi(r)^{p-1}|\psi'(r)|\in L^1(\T_{s_0}\cap\,\Omega),\quad     \psi(r)^{1+\varepsilon}\phi(r)^{p-1}|\psi'(r)|\in L^1(\Omega\backslash{\T_{s_0}}).                        \label{Ambrosio-condition-2}
		\end{equation}
		
	\end{enumerate}
	
	\noindent Then for each $\varepsilon\in (0,\varepsilon_1)$, by setting $c(\varepsilon):=(1+\varepsilon)/p$, the following function satisfies \eqref{inverse-ineq}
	\begin{align}
		\nu_\e(x):=\begin{cases}
			\Big({\psi(r(x))\over \psi(s_0)}\Big)^{-c(\e/2)},& \text{ if }x\in {\T_{s_0}}\cap \,\Omega,\\
			\\
			\Big({\psi(r(x))\over \psi(s_0)}\Big)^{c(\e)},& \text{ if }x\in \Omega\setminus \T_{s_0}.
		\end{cases}  \label{v-Def-2}
	\end{align}
\end{lemma}

Now we discuss Lemmas \ref{Ambrosio-lem} $\&$ \ref{Ambrosio-lem2} in the space form $\MM_\lambda$.

\begin{example}\label{space-form-phi-psi3} Let $\MM_\lambda$, $\Sigma$, $\Omega$ and $(\phi ,\psi )$ be as in Example \ref{space-form-phi-psi2}.
	It follows by Example  \ref{space-form-phi-psi1} that   $\T_{s_0}\cap \, \Omega$ and $\Omega\,\backslash {\T_{s_0}}$ have  regular boundaries for every  $s_0\in (0,\mathfrak{r}_{\lambda})$. We claim that
	\begin{equation}\label{conditionspaceform1}
		\text{vol}_{i^*g}(\Sigma\cap \overline{\Omega})<+\infty,  \ \beta<-(m-n)  \quad \Longrightarrow \quad   \text{\eqref{Ambrosio-condition-2}}.
	\end{equation}

	We first discuss the case $\lambda=1$. Since $\Lambda<\lambda=1$, Theorem \ref{Laplacian-comparison-Sigma} combined with \eqref{narspaceform}$_1$ yields
	\[
	 \det \mathcal {A}(t,\Bn)=\mathbf{c}_\lambda(t)^n\mathbf{s}_\lambda(t)^{m-n-1}\leq \mathbf{c}_\Lambda(t)^n\mathbf{s}_\Lambda(t)^{m-n-1}, \quad t\in [0, \mathfrak{r}_\lambda].
	\]
	Thus, for any $\varepsilon \in (0,1)$, a direct calculation together with \eqref{domaincontts}, \eqref{phi-psi-M-lambda}, \eqref{dv-g-Def} and \eqref{narspaceform}   provides
	\begin{eqnarray*}
		\int_{{\T_{s_0}}\cap\Omega}\psi(r)^{-(1+\varepsilon)}\phi(r)^{p-1}|\psi'(r)|\dvol_g&\leq& C |\alpha| \int^{s_0}_0\mathbf{s}_\Lambda(t)^{-(m-n)-\alpha\varepsilon}\mathbf{c}_\Lambda(t)^{-n+1}\mathbf{c}_\lambda(t)^n\mathbf{s}_\lambda(t)^{m-n-1}{\dd}t\\
		&\leq&  C |\alpha| \int^{s_0}_0\mathbf{s}_\Lambda(t)^{-1-\alpha\varepsilon}\mathbf{c}_\Lambda(t) {\dd}t= \frac{C}{\varepsilon}\mathbf{s}_\Lambda(s_0)^{-\alpha\varepsilon}<+\infty,
	\end{eqnarray*}
	where $\alpha:=\beta+m-n< 0$ and $C:=\vol(\mathbb{S}^{m-n-1})\text{vol}_{i^*g}(\Sigma\cap \overline{\Omega})$.
	A similar computation  yields
	\begin{align*}
		\int_{\Omega\backslash{\T_{s_0}}}  \psi(r)^{1+\varepsilon}\phi(r)^{p-1}|\psi'(r)|\dvol_g\leq C |\alpha| \int^{\mathfrak{r}_{\lambda}}_{s_0} \mathbf{s}_\Lambda(t)^{\alpha(2+\varepsilon)-1}\mathbf{c}_\Lambda(t){\dd}t=\frac{-C}{2+\varepsilon}\mathbf{s}_\Lambda(t)^{\alpha(2+\varepsilon)}\Big|^{\mathfrak{r}_{\lambda}}_{s_0}<+\infty.
	\end{align*}
	Therefore, \eqref{conditionspaceform1}  holds.
	
	For the case $\lambda\in \{0,-1\}$, the arguments are similar and even easier than before; thus, we omit the details.
	In particular, in view of Remark \ref{basiclambda=Lambdacase}, if $\Lambda=\lambda\leq 0$, then \eqref{conditionspaceform1} can be strengthened as
	\begin{equation}\label{conditionspaceform12}
		\text{vol}_{i^*g}(\Sigma\cap \overline{\Omega})<+\infty, \ \beta>-(m-n) \text{ (resp., $\beta<-(m-n)$)}   \quad \Longrightarrow \quad \eqref{Ambrosio-condition-1} \text{ (resp., \eqref{Ambrosio-condition-2})}.
	\end{equation}
\end{example}

In order to establish the sharpness result, we introduce the following space.
\begin{definition}\label{D1pspace}
	Given an open set $U\subset M$, the {\it weighted Sobolev space} $D^{1,p}(U,\phi,\psi)$ is defined as the completion of $C^\infty_0(U)$ under the following norm
	\begin{equation*}
		\|u\|_{D}:=\left(\int_{U}|u|^p \phi(r)^{p-1}|\psi'(r)|\dvol_g+\int_{U}|\na u|^p \frac{\phi(r)^{p-1} |\psi'(r)|}{|\nabla(\log \psi(r))|^{p}}\dvol_g\right)^{1/ p}.\label{D-norm}
	\end{equation*}
\end{definition}

\begin{proposition}\label{sharpness-Hardy-0} Let $(M,g)$, $\Sigma$, $\Omega$, $(\phi,\psi)$ and $\nu_\e$  be as    in either Lemma \ref{Ambrosio-lem} or  Lemma \ref{Ambrosio-lem2}. If $\nu_\e\in D^{1,p}(\Omega_\Sigma,\phi,\psi)$ for any small $\varepsilon>0$, then
	\eqref{Hardy-Inequ-general} is sharp $($i.e., \eqref{sharpfunctionexp} holds$).$
\end{proposition}

\begin{proof}
	
	Since $\nu_\e\in D^{1,p}(\Omega_\Sigma,\phi,\psi)$, there exists a sequence $u_i\in C^\infty_0(\Omega_\Sigma)\setminus \{0\}$ such that
	\begin{align}   \lim_{i\rightarrow\infty}\int_{\Omega_\Sigma}|u_i|^p \phi(r)^{p-1}|\psi'(r)|\dvol_g&= \int_{\Omega_\Sigma}|\nu_\e|^p \phi(r)^{p-1}|\psi'(r)|\dvol_g,\nonumber\\
		\lim_{i\rightarrow\infty}\int_{\Omega_\Sigma}|\na u_i|^p \frac{\phi(r)^{p-1} |\psi'(r)|}{|\nabla(\log \psi(r))|^{p}}\dvol_g&=\int_{\Omega_\Sigma}|\na \nu_\e|^p \frac{\phi(r)^{p-1} |\psi'(r)|}{|\nabla(\log \psi(r))|^{p}}\dvol_g,\nonumber
	\end{align}
	which together with  \eqref{Hardy-Inequ-general} and \eqref{inverse-ineq} yields
	\begin{align}
		p^{-p}\leq& \inf_{u\in C^\infty_0(\Omega_\Sigma)\backslash\{0\}} {\ds\int_{\Omega_\Sigma}|\na u|^p \frac{\phi(r)^{p-1} |\psi'(r)|}{|\nabla(\log \psi(r))|^{p}}\dvol_g\over \ds\int_{\Omega_\Sigma}|u|^p \phi(r)^{p-1}|\psi'(r)|\dvol_g }\leq   \lim_{i\rightarrow\infty}{\ds\int_{\Omega_\Sigma}|\na u_i|^p \frac{\phi(r)^{p-1} |\psi'(r)|}{|\nabla(\log \psi(r))|^{p}}\dvol_g\over \ds\int_{\Omega_\Sigma}|u_i|^p \phi(r)^{p-1}|\psi'(r)|\dvol_g}\nonumber\\
		=&{\ds\int_{\Omega_\Sigma}|\na \nu_\e|^p \frac{\phi(r)^{p-1} |\psi'(r)|}{|\nabla(\log \psi(r))|^{p}}\dvol_g\over \ds\int_{\Omega_\Sigma}|\nu_\e|^p \phi(r)^{p-1}|\psi'(r)|\dvol_g}<c(\e)^p.\nonumber
	\end{align}
	Letting $\e\rightarrow0^+$, the sharpness of \eqref{Hardy-Inequ-general} follows due to the fact that $c(\e)^p\rightarrow p^{-p}$ as $\e\to 0$.
\end{proof}

The above proposition points out the importance of $\nu_\e\in D^{1,p}(\Omega_\Sigma,\phi,\psi)$.  Note that the domain of  $\nu_\e$ is $\Omega$ and then we have the following result.
\begin{corollary}\label{finsharpcorollary}
	Let $(M,g)$, $\Sigma$,  $\Omega$, $(\phi,\psi)$ and $\nu_\e$  be as    in either Lemma \ref{Ambrosio-lem} or  Lemma \ref{Ambrosio-lem2}.
	Suppose that  $\nu_\e$  can be extended to $\overline{\Omega}$ and satisfies the following two conditions
	\begin{itemize}
		
		\item $\nu_\e$ vanishes on $\Sigma\cap \Omega$ and $\partial\Omega$;
		
		\smallskip
		
		\item if $\Omega$ is unbounded, then $\nu_\e(x)\rightarrow 0$ as $r(x)\rightarrow +\infty$.
		
	\end{itemize}
	Then $\nu_\e\in D^{1,p}(\Omega_\Sigma,\phi,\psi)$ and hence, \eqref{Hardy-Inequ-general} is sharp.
\end{corollary}
\begin{proof}
	Due to Proposition \ref{sharpness-Hardy-0}, it suffices to show $\nu_\e\in D^{1,p}(\Omega_\Sigma,\phi,\psi)$. On the one hand, in view of  Lemmas \ref{Ambrosio-lem} $\&$ \ref{Ambrosio-lem2},  the function $\nu_\varepsilon$   satisfies \eqref{inverse-ineq}  and
	\begin{equation}
		|\nu_\e|^p\phi(r)^{p-1}|\psi'(r)|,\ |\na \nu_\e|^p \frac{\phi(r)^{p-1} |\psi'(r)|}{|\nabla(\log \psi(r))|^{p}} \in L^1(\Omega_\Sigma).\label{L-1-bound}
	\end{equation}
		On the other hand,
	set $\nu_{\e,\iota}:=\max\{\nu_\e-\iota,0\}$ for small $\iota>0$.   By the properties of $\nu_\e$, it is not hard to check that $\nu_{\e,\iota}$ is a Lipschitz function with compact support in $\Omega_\Sigma$. By using Lemma \ref{lipsconverppax}, there exists a sequence of smooth Lipschitz functions
	$w_n\in C^\infty_0(\Omega_\Sigma)\cap \Lip_0(\Sigma)$  such that
	\begin{equation}\label{fnvelconverge}
		\supp w_n\subset \mathscr{K}_\iota,\quad |w_n|\leq C_\iota,\quad |\nabla w_n|\leq     \dil(w_n)\leq C_\iota,\quad w_n\rightrightarrows \nu_{\e,\iota},\quad \|w_n-\nu_{\e,\iota} \|_{W^{1,1}(\Omega_\Sigma)}\rightarrow 0,
	\end{equation}
	where $\mathscr{K}_\iota$ is a compact set in $\Omega_\Sigma$ and $C_\iota>0$ is  a constant. Due to $\nu_{\e,\iota}\in \Lip_0(\Omega_\Sigma)$,  by choosing eventually larger $C_\iota$ and $\mathscr{K}_\iota$, and by passing to a subsequence, we may assume that
	\begin{equation}\label{nucondtion}
		\supp \nu_{\e,\iota}\subset \mathscr{K}_\iota\quad |\nu_{\e,\iota}|\leq C_\iota, \quad |\nabla \nu_{\e,\iota}|\leq  \dil(\nu_{\e,\iota})\leq C_\iota, \quad \nabla w_n \rightarrow \nabla \nu_{\e,\iota} \text{ $\vol_g$-a.e. pointwise}.
	\end{equation}
	Thus, a direct calculation yields
	\begin{align*}
		\left\| w_n- \nu_{\e,\iota}  \right\|_D^p= \int_{\mathscr{K}_\iota} \left| w_n- \nu_{\e,\iota}\right|^{p}  \phi(r)^{p-1}|\psi'(r)|  \dvol_g +\int_{\mathscr{K}_\iota} \left|\nabla w_n-\nabla \nu_{\e,\iota}\right|^{p} \frac{\phi(r)^{p-1} |\psi'(r)|}{|\nabla(\log \psi(r))|^{p}} \dvol_g,
	\end{align*}
	which together with \eqref{fnvelconverge}, \eqref{nucondtion}, Assumption \ref{Basciasumptionfunction}/\eqref{integrable} and Lebesgue's dominated convergence theorem furnishes $\left\| w_n- \nu_{\e,\iota}  \right\|_D\rightarrow 0$ as $n\rightarrow \infty$, i.e., $\nu_{\e,\iota}\in D^{1,p}(\Omega_\Sigma,\phi,\psi)$.
	In a similar way, we have
	\begin{align}
		\|\nu_{\e,\iota}-\nu_\e\|_D^p=&\int_{\Omega_\Sigma}|\nu_{\e,\iota}-\nu_\e|^p \phi(r)^{p-1}|\psi'(r)|\dvol_g+\int_{\Omega_\Sigma}|\na \nu_{\e,\iota}-\na\nu_\e|^p \frac{\phi(r)^{p-1} |\psi'(r)|}{|\nabla(\log \psi(r))|^{p}}\dvol_g\nonumber\\
		=&\int_{\Omega_\Sigma}\chi_{\{0\leq \nu_\e\leq \iota\}}|\nu_\e|^p\phi(r)^{p-1}|\psi'(r)|\dvol_g+\int_{\Omega_\Sigma}\chi_{\{\nu_\e>\iota\}}|\iota|^p\phi(r)^{p-1}|\psi'(r)|\dvol_g\nonumber\\
		&+\int_{\Omega_\Sigma}\chi_{\{0\leq \nu_\e\leq \iota\}}|\na \nu_\e|^p \frac{\phi(r)^{p-1} |\psi'(r)|}{|\nabla(\log \psi(r))|^{p}}\dvol_g.\label{difference-between-two-vs}
	\end{align}
	Since $\chi_{\{\nu_\e>\iota\}}|\iota|^p\phi(r)^{p-1}|\psi'(r)|\leq |\nu_\e|^p\phi(r)^{p-1}|\psi'(r)|$,  Lebesgue's dominated convergence theorem combined with \eqref{L-1-bound} and \eqref{difference-between-two-vs} furnishes $\|\nu_{\e,\iota}-\nu_\e\|_D\rightarrow0$ as $\iota\rightarrow0^+$,
	which means $\nu_\e\in D^{1,p}(\Omega_\Sigma,\phi,\psi)$.
\end{proof}

By recalling \eqref{cutradidfef}, we have the following consequence.

\begin{corollary}\label{geomertricmeaningofsharpness2}
	Let $(M,g)$ be an $m$-dimensional complete Riemannian manifold, $i:\Sigma \hookrightarrow M$ be an $n$-dimensional closed submanifold with $\mathfrak{c}_\mathcal {V}(\Sigma)>0$, $\Omega$ be a domain such that either $\Omega=M$ or $\partial\Omega=\Sigma$ and let $(\phi,\psi)$   satisfy Assumption \ref{Basciasumptionfunction}.
	
	\smallskip
	
	{\noindent $\bullet$}  If $\Omega$ is bounded,  suppose that one of the following   conditions hold$:$
	\begin{enumerate}[{ \rm(i)}]
		
		\item \label{phie1233412} $\psi(0^+)=0$, $\psi'(t)>0$  and \eqref{Ambrosio-condition-1}$;$
		
		\item \label{phie1233432} $\psi(0^+)=+\infty$, $\psi'(t)<0$ and \eqref{Ambrosio-condition-2}.
		
	\end{enumerate}

	{\noindent $\bullet$} If $\Omega$ is unbounded,  assume that one of the following conditions hold$:$
	\begin{enumerate}[{ \rm(i')}]
		
		\item \label{phie123341} $\psi(0^+)=0$, $\psi'(t)>0$, $\lim_{t\rightarrow +\infty}\psi(t)=+\infty$ and \eqref{Ambrosio-condition-1}$;$
		
		\item \label{phie123343} $\psi(0^+)=+\infty$, $\psi'(t)<0$,  $\lim_{t\rightarrow +\infty}\psi(t)=0$ and \eqref{Ambrosio-condition-2}.
		
	\end{enumerate}

	\noindent Then \eqref{Hardy-Inequ-general} is sharp.
\end{corollary}

\begin{proof}
	Since $\mathfrak{c}_\mathcal {V}(\Sigma)>0$, choose $s_0\in (0,\mathfrak{c}_\mathcal {V}(\Sigma))$ small enough so that $\T_{s_0}\cap~\Omega$ an $\Omega\setminus \T_{s_0}$ are non-empty with piece-wise regular boundaries.
	
	Suppose that  $\Omega$ is bounded.   If Condition \eqref{phie1233412} (resp., Condition \eqref{phie1233432})  holds, we define $\nu_\e$ by \eqref{v-Def-1} (resp., by \eqref{v-Def-2}). Since either $\Omega=M$ or $\partial\Omega=\Sigma$ holds,  $\nu_\e$  can be extended to $\overline{\Omega}$ and
	vanishes on $\Sigma\cup \partial\Omega$. Hence,  Corollary \ref{finsharpcorollary} implies the sharpness of \eqref{Hardy-Inequ-general}.
	The proof  of  unbounded $\Omega$ is similar and hence, we omit it.
\end{proof}

In the next two examples we apply Corollary \ref{geomertricmeaningofsharpness2} in space forms $\MM_\lambda$.
\begin{example}[$\lambda=0$]\label{space-form-phi-psi4}
	Let $\MM_0=\mathbb{R}^m$ and $\Sigma$ be an affine $n$-plane. In particular, $\mathfrak{c}_{\mathcal {V}}(\Sigma)=+\infty$.
	On the one hand,   if $\text{vol}_{i^*g}(\Sigma\cap \overline{\Omega})<+\infty$ with either $\Omega=\mathbb{R}^m$ or $\partial\Omega=\Sigma$, we have only one choice (see Example \ref{space-form-phi-psi3}):
	\[
	\Sigma \text{ is a point, say the origin $\mathbf{0}$,  i.e., $n=0$, and $\Omega_\Sigma=\mathbb{R}^m\backslash\{\mathbf{0}\}$}.
	\]
	Now
	let $(\phi,\psi)$ be as in Example \ref{space-form-phi-psi2} with $\Lambda=\lambda=0$ (see also Remark \ref{basiclambda=Lambdacase}). In the case when $n=0$, owing to \eqref{conditionspaceform12}, we have that  $\beta>-m$ (resp., $\beta<-m$), which implies Condition (\ref{phie123341}') (resp., Condition (\ref{phie123343}')) in Corollary \ref{geomertricmeaningofsharpness2}.

	In particular, Corollary \ref{geomertricmeaningofsharpness2} implies the classical sharp Hardy inequality: for any $p>1$ and $\beta\neq-m$,
	\begin{equation}\label{punthardyonRm}
		\int_{\mathbb{R}^m} |\nabla u|^p |x|^{p+\beta}{\dd}x\geq \left|\frac{\beta+m}{p}\right|^p\int_{\mathbb{R}^m}|u|^p |x|^\beta{\dd}x,\quad \forall\,u\in C^\infty_0(\mathbb{R}^m\backslash\{\mathbf{0}\}),
	\end{equation}
	where ${\dd}x$ denotes  the Lebesgue measure and   $|x|$ denotes the norm of $x$.
\end{example}

\begin{example}[$\lambda=1$]\label{smbalsigsharp} Let   $\MM_1=\mathbb{S}^m$ ($m\geq 2$) and $\Sigma$ be a great $\mathbb{S}^{m-1}$ (i.e., $n=m-1$). Thus, $\mathfrak{c}_{\mathcal {V}}(\Sigma)=\pi/2>0$. Let $\Omega=\mathbb{S}^m_+$ be a hemisphere with $\partial \Omega=\Sigma$ and let $(\phi,\psi)$ be defined by \eqref{phi-psi-M-lambda} with $\Lambda=0$, i.e.,
	\[
	\phi(t):=t^{-\frac{m-n-1}{p-1}}=1,\quad \psi(t):=t^{\beta+m-n}=t^{\beta+1},\quad\beta<-1.
	\]
	Due to \eqref{conditionspaceform1},  Condition \eqref{phie1233432} in Corollary \ref{geomertricmeaningofsharpness2} is satisfied. Therefore, Corollary \ref{geomertricmeaningofsharpness2} together with  Example \ref{space-form-phi-psi2} furnishes a sharp Hardy inequality: for any $p>1$ and $\beta<-1$,
	\begin{equation}
		\int_{\mathbb{S}^m_+}|\nabla u|^p r^{p+\beta}\dvol_g\geq \left|  \frac{\beta+1}{p} \right|^p \int_{\mathbb{S}^m_+}|u|^p r^\beta \dvol_g,\quad \forall\,u\in C^\infty_0(\mathbb{S}^m_+).\label{hardyineqshshop}
	\end{equation}
\end{example}

\subsection{Non-existence of extremal functions}\label{nonexsorigin}   In this subsection, we investigate the existence of extremal functions of \eqref{Hardy-Inequ-general}  in $C^\infty_0(\Omega_\Sigma)$.
To do this, we recall the inequality of Lindqvist \cite{Lindqvist}: for every $X,Y\in \mathbb R^m$ one has
\begin{equation}
 |X+Y|^p\geq |X|^p+p|X|^{p-2}\la X,Y\ra+
 \begin{cases} C_p{|Y|^2\over (|X|+|Y|)^{2-p}},& \text{ if }p\in(1,2),\\
 \\
 C_p|Y|^p,& \text{ if } p\geq2,
 \end{cases}\label{main-inequality-vectors}
 \end{equation}
where $C_p$ is a positive constant only depending on $p$. We have the following estimate.

\begin{proposition}\label{prop-bv-improve}  Let $(M,g)$, $\Sigma$, $\Omega$ and $(\phi,\psi)$  be as in Theorem \ref{Hardy-thm-general}. Then for any $u\in C^\infty_0({\Omega}_\Sigma)\setminus\{0\}$,   one has \begin{align}
 \int_{\Omega_\Sigma} |\na u|^p {\dd}\mu\geq&\,p^{-p}\int_{\Omega_\Sigma} |u|^p\phi(r)^{p-1}|\psi'(r)|\dvol_g\notag\\
 &+\begin{cases} C_p\displaystyle\int_{\Omega_\Sigma}{\psi(r)^{-1}|\na h|^2\over (p^{-1}|\na(\log \psi(r))||h|+|\na h|)^{2-p}}{\dd}\mu,& \text{ if } p\in(1,2),\\
\\
C_p\displaystyle\int_{\Omega_\Sigma} |\na h|^p\psi(r)^{-1}{\dd}\mu,& \text{ if } p\geq2,
\end{cases}\label{io0-9-thm}
\end{align}
where ${\dd}\mu:=\phi(r)^{p-1}|\psi'(r)||\nabla(\log \psi(r))|^{-p}\dvol_g$ and $h(x):=\psi(r(x))^{1/p}u(x)$. In particular, each term in the right-hand side of \eqref{io0-9-thm} is finite and nonnegative.

\end{proposition}

\begin{proof} Without loss of generality, we may consider only nonnegative test functions $u\in C^\infty_0(\Omega_\Sigma)\setminus\{0\}$ in \eqref{io0-9-thm}; otherwise, we consider standard approximation for $|u|$ by nonnegative elements of $C^\infty_0(\Omega_\Sigma).$ Let $\Psi(r):=\psi(r)^{-1}$ and $u(x)=:\Psi(r(x))^{1/ p}h(x)\geq 0$. Since Assumption \ref{Basciasumptionfunction}/\eqref{negligbileset} implies that $\{\psi(r)=0\}$ is a $\vol_g$-negligible set, we have  $\supp u=\supp h$.

 We are going to apply inequality \eqref{main-inequality-vectors} for the  vector fields
\[
X:=p^{-1}\Psi(r)^{1/p-1}\Psi'(r)h\na r,\quad Y:=\Psi(r)^{1/p}\na h,
\]
which are well-defined for $\vol_g$-a.e. in $\Omega_\Sigma$; after an elementary computation we obtain
\begin{align}
\int_{\Omega_\Sigma}|\na u|^p{\dd}\mu\geq&\,\frac{1}{p^p}\int_{\Omega_\Sigma} |u|^p\phi(r)^{p-1}|\psi'(r)|\dvol_g-\frac{p}{p^p}\int_{\Omega_\Sigma}\sgn(\psi'(r))\la \na r,\na |h|^p\ra \phi(r)^{p-1}\dvol_g\nonumber\\
&+\begin{cases} C_p\displaystyle\int_{\Omega_\Sigma}{\psi(r)^{-1}|\na h|^2\over (p^{-1}|\na(\log\Psi(r))||h|+|\na h|)^{2-p}}{\dd}\mu, & \text{ if } p\in(1,2),\\
\\
C_p \displaystyle\int_{\Omega_\Sigma}|\na h|^p\psi(r)^{-1}{\dd}\mu, & \text{ if } p\geq2.\end{cases}\label{io0-9}
\end{align}
 It remains to show the nonnegativity of the   second term in the RHS of  (\ref{io0-9}). Since $p>1$, and both $u\in C^\infty_0(\Omega_\Sigma)$ and $h$ are nonnegative,    Assumption \ref{Basciasumptionfunction}/\eqref{smoothofpsi}\eqref{nansigcontroll} combined with  \eqref{lipdivergence} furnish
\[
0\leq -\int_{\Omega_\Sigma} \langle \nabla (u^p \psi(r)), \sgn(\psi'(r))\,\phi(r)^{p-1}\na r   \rangle\dvol_g=-\int_{\Omega_\Sigma}\sgn(\psi'(r))\la \na r,\na h^p\ra \phi(r)^{p-1}\dvol_g,
\]
which together with \eqref{io0-9} concludes the proof of \eqref{io0-9-thm}.
\end{proof}

The following result implies the nonexistence of extremal functions in \eqref{Hardy-Inequ-general}; we recall the functional $\mathcal {J}_{p,U}$ from \eqref{functinoalshapr}.

\begin{theorem}\label{achieve-thm} Let $(M,g)$, $\Sigma$, $\Omega$ and $(\phi,\psi)$  be as in Theorem \ref{Hardy-thm-general}. Thus,  for every $p>1$, there holds
\[
\mathcal {J}_{p,\Omega_\Sigma}(u)>p^{-p},\quad \forall\,u\in C^\infty_0(\Omega_\Sigma)\backslash\{0\}.
\]
\end{theorem}


\begin{proof}  In view of \eqref{Hardy-Inequ-general}, suppose by contradiction that for some $u\in C^\infty_0(\Omega_\Sigma)\setminus \{0\}$ we have $\mathcal {J}_{p,\Omega_\Sigma}(u)=p^{-p}$. 
According  to Proposition \ref{prop-bv-improve}, since every item of the RHS of (\ref{io0-9-thm}) is nonnegative, we necessarily have
\[
0=\begin{cases} \displaystyle\int_{\Omega_\Sigma}{\psi(r)^{-1}|\na h|^2\over (p^{-1}|(\log\psi(r))'||h|+|\na h|)^{2-p}}{\dd}\mu,& \text{ if } p\in(1,2),\\
\\
\displaystyle\int_{\Omega_\Sigma} |\na h|^p\psi(r)^{-1}{\dd}\mu,& \text{ if } p\geq2.
\end{cases}
\]
Due to Assumption \ref{Basciasumptionfunction}, the measure $\mu$ shares the same negligible sets with $\vol_g$ in $\Omega_\Sigma$. Thus,  $\na h=\na(\psi(r(x))^{1/p}u(x))=0$ for $\vol_g$-a.e.  $x\in \Omega_\Sigma$. Since $x\mapsto \psi(r(x))^{1/p}u(x)$ is continuous in $\Omega_\Sigma$, an elementary argument  yields that $h$ is constant on each connected component of $\Omega_\Sigma$, which is an open set in $M$.

Since $u\neq0$, let $x_0\in \Omega_\Sigma$ be such that  $u(x_0)\neq0$ and $U$ denote the connected component of $\Omega_\Sigma$ containing $x_0$.   Let $C_U$ be the constant satisfying
$$ h(x)=\psi(r(x))^{1/p}u(x)=C_U,\quad  \forall x\in U.$$
First, since $\Omega$ is connected (a domain), we necessarily have $\partial U\cap \Sigma\neq\emptyset$, which combined with the fact that $u=0$ near $\Sigma$ implies $C_U=0$. Second, since  $U$ is an open set containing $x_0$ and $\{\psi(r)=0 \}$ is a negligible set, see Assumption \ref{Basciasumptionfunction}/\eqref{negligbileset}, there exists a point $x\in  U$ close to $x_0$ such that $\psi(r(x))\neq0$ and $u(x)\neq0$; in particular, $C_U=\psi(r(x))^{1/p}u(x)\neq0$, a contradiction.
 \end{proof}

\section{Extensions to larger function spaces}\label{bigspaces}

The aim of this section is to extend the Hardy inequality \eqref{Hardy-Inequ-general} to some larger spaces  than $C^\infty_0(\Omega_\Sigma)$.

\subsection{Extension to the space  $C^\infty_0(\Omega)$}\label{section-4-1}

In this subsection, we are going to extend Hardy inequality \eqref{Hardy-Inequ-general} to functions $u\in C^\infty_0(\Omega)$. In order to do this, we introduce the following assumption, which is a slightly stronger than Assumption \ref{Basciasumptionfunction}.

\begin{assumption}\label{Basciasumptionfunction3} 
 Given $p>1$, let $(\phi,\psi)$   be a pair of two nonnegative $C^1$-functions defined on  some subinterval of $(0, +\infty)$  such that
\begin{enumerate}[{\rm(a)}]
 \item\label{smoothofphi3} $\phi(r)$ is well-defined on $\Omega_\Sigma$ while $\psi(r)$ can be continuously extended to $\Omega$;

\smallskip

\item\label{smoothofpsi3} $\psi$ is monotone  and  $\psi(r)$ is locally Lipschitz on $\Omega_\Sigma$;

\smallskip

\item\label{derivativeintegr} either $\lim_{t\rightarrow 0^+}\psi'(t)<+\infty$ or $\lim_{t\rightarrow 0^+}\psi'(t)=+\infty$ with $\psi'>0$ and $\phi(r)^{p-1}\psi(r)^p\in L^1_{\loc}(\Omega)$;

\smallskip

\item\label{nansigcontroll3} {$\ds\dii(\sgn(\psi'(r))\,\phi(r)^{p-1}\na r)\geq 0$} in the weak sense in $\Omega$;

\smallskip

\item\label{integrable3} $\phi(r)^{p-1}, \phi(r)^{p-1}\psi(r),\phi(r)^{p-1}|\psi'(r)|,\frac{\phi(r)^{p-1}|\psi'(r)|}{ |\nabla(\log \psi(r))|^{p}}\in L^1_{\loc}(\Omega)$;

\smallskip

\item\label{negligbileset3} $\{x\in \Omega_\Sigma\,:\, \phi(r(x))\,\psi(r(x))\,\psi'(r(x))=0\}$ is a $\vol_g$-negligible set.

\end{enumerate}
\end{assumption}

\begin{remark}\label{addtiondervaliporemark}
In view of \eqref{derivativeintegr}, if $\lim_{t\rightarrow 0^+}\psi'(t)<+\infty$, then it follows by Lagrange's mean value theorem and \eqref{smoothofpsi3} that  $\psi(r)$ is locally Lipschitz on the whole domain $\Omega$.

\end{remark}


\begin{example}\label{space-form-phi-psi5}
Given $\lambda\in \{0,-1\}$, let $\mathbb{M}_\lambda$ be  the space form and $\Sigma$ be the totally geodesic submanifold as in Example \ref{space-form-phi-psi1}. For $p>1$ and
 $\beta>-(m-n)$, the pair  $(\phi ,\psi )$ is defined   by
 \begin{align}
 \phi(t):=\mathbf{c}_\lambda(t)^{-{n\over p-1}}\mathbf{s}_\lambda(t)^{-{m-n-1\over p-1}},\quad   \psi(t):=\mathbf{s}_\lambda(t)^{\beta+m-n},\quad \forall\,t\in (0,+\infty).
 \end{align}
Thus, for any $\Omega \subset \mathbb{M}_\lambda$ with $\Omega \cap \Sigma\neq\emptyset$,  Assumption \ref{Basciasumptionfunction3} is satisfied.

Indeed, Conditions \eqref{smoothofphi3},\eqref{smoothofpsi3},\eqref{negligbileset3} trivially hold. Moreover, since $\Sigma$ is a $\vol_g$-negligible set, Remark \ref{basiclambda=Lambdacase} implies Condition \eqref{nansigcontroll3} (also see \eqref{divdireccalution}). In addition, due to \eqref{narspaceform}$_1$, we have $\phi(t)^{p-1}=(\det \mathcal {A}(t,\mathbf{n}))^{-1}$ and hence, if $\alpha:=\beta+m-n>0$, then
 \begin{align*}
&\phi(t)^{p-1}\psi(t)^s\det \mathcal {A}(t,\mathbf{n})=\mathbf{s}_\lambda(t)^{s\alpha },\quad
 \phi(t)^{p-1}|\psi'(t)|\det \mathcal {A}(t,\mathbf{n})=\alpha\mathbf{s}_\lambda(t)^{\alpha-1}\mathbf{c}_\lambda(t),\\
 &\frac{\phi(t)^{p-1}|\psi'(t)|}{ |  (\log \psi(t))'|^{p}}\det \mathcal {A}(t,\mathbf{n})=\alpha^{1-p}\mathbf{c}_\lambda(t)^{1-p}\mathbf{s}_\lambda(t)^{\alpha+p-1},
 \end{align*}
which together with \eqref{dv-g-Def} provides Conditions \eqref{derivativeintegr}\eqref{integrable3}.

\end{example}

Now we have the following inequality.

\begin{theorem}\label{Hardy-thm-2}  Let $(M,g)$ be an $m$-dimensional complete Riemannian manifold, $i:\Sigma \hookrightarrow M$ be an $n$-dimensional closed submanifold and $\Omega$ be a {\rm noncompact} domain in $M$ satisfying \eqref{natural-domain}.
  Let $r(x):=d(\Sigma,x)$ be the distance function  from $\Sigma$.  For any pair $(\phi,\psi)$ satisfying Assumption \ref{Basciasumptionfunction3}, there holds
\begin{equation}
 \int_{\Omega}|\na u|^p \frac{\phi(r)^{p-1}|\psi'(r)|}{ |\nabla(\log \psi(r))|^{p}}\dvol_g\geq p^{-p}\int_{\Omega} |u|^p\phi(r)^{p-1}|\psi'(r)|\dvol_g,\quad \forall\,u\in C^\infty_0(\Omega).\label{Hardy-Inequ-sec-upper-bound}
 \end{equation}
 \end{theorem}
\begin{proof} In view of Assumption \ref{Basciasumptionfunction3}/\eqref{derivativeintegr},
if $\lim_{t\rightarrow 0^+}\psi'(t)<+\infty$, Remark \ref{addtiondervaliporemark} implies that  $\psi(r)$ is locally Lipschitz on $\Omega$. Thus, \eqref{Hardy-Inequ-sec-upper-bound} follows by the same proof of  Theorem \ref{Hardy-thm-general}.

In the sequel, we consider the case when $\lim_{t\rightarrow 0^+}\psi'(t)=+\infty$ with $\psi'>0$ and $\phi(r)^{p-1}\psi(r)^p\in L^1_{\loc}(\Omega)$.
For every small $\varepsilon>0$ we define the function
\[
\psi_\varepsilon(t):=\int^{t}_0\left(  \frac{1}{\psi'(s)}+\varepsilon \right)^{-1}{\dd}s   +   \psi(0).
\]
It is easy to see that $\psi'_\varepsilon\leq \psi'$ and $\psi_\varepsilon\leq \psi$. In particular, $\psi'_\varepsilon(0^+)$ exists and hence, $\psi_\varepsilon(r)$ is a locally Lipschitz function on $\Omega$. Accordingly, the first step implies that the inequality \eqref{Hardy-Inequ-sec-upper-bound} remains valid for the pair $(\phi,\psi_\varepsilon)$, i.e.,
\begin{equation}
 \int_{\Omega}|\na u|^p \frac{\phi(r)^{p-1}|\psi_\varepsilon'(r)|}{ |\nabla(\log \psi_\varepsilon(r))|^{p}}\dvol_g\geq p^{-p}\int_{\Omega} |u|^p\phi(r)^{p-1}|\psi_\varepsilon'(r)|\dvol_g,\quad \forall\,u\in C^\infty_0(\Omega).\label{Hardy-Inequ-sec-upper-bound00}
 \end{equation}
Due to the fact that $|u|^p\phi(r)^{p-1}|\psi_\varepsilon'(r)|\leq |u|^p\phi(r)^{p-1}|\psi'(r)|\in L^1(\Omega)$, the dominated convergence theorem implies
\begin{equation}\label{conlimi1}
\lim_{\varepsilon\rightarrow 0^+}\int_{\Omega} |u|^p\phi(r)^{p-1}|\psi_\varepsilon'(r)|\dvol_g=\int_{\Omega} |u|^p\phi(r)^{p-1}|\psi'(r)|\dvol_g.
\end{equation}
On the other hand, by Assumption \ref{Basciasumptionfunction3}/\eqref{integrable3} and   $\phi(r)^{p-1}\psi(r)^p\in L^1_{\loc}(\Omega)$, we obtain
\begin{eqnarray*}
\frac{\phi(r)^{p-1}|\psi_\varepsilon'(r)|}{ |\nabla(\log \psi_\varepsilon(r))|^{p}}&=&\frac{\phi(r)^{p-1}|\psi_\varepsilon(r) |^p}{|\psi'_\varepsilon(r)|^{p-1}}\leq \frac{\phi(r)^{p-1}\psi(r)^p}{|\psi'_\varepsilon(r)|^{p-1}}=\phi(r)^{p-1}\psi(r)^p\left| \frac{1}{\psi'(r)}+\varepsilon \right|^{p-1}\\
&\leq& \frac{\phi(r)^{p-1}\psi(r)^p}{|\psi'(r)|^{p-1}} \left( 1+|\psi'(r)|  \right)^{p-1}\leq C(p)\left( \frac{\phi(r)^{p-1}\psi(r)^p}{|\psi'(r)|^{p-1}}+\phi(r)^{p-1}\psi(r)^p\right) \in L^1_{\lo}(\Omega),
\end{eqnarray*}
which combined with the dominated convergence theorem  yields
\begin{equation}\label{conlimi2}
\lim_{\varepsilon\rightarrow 0^+}\int_{\Omega}|\na u|^p \frac{\phi(r)^{p-1}|\psi_\varepsilon'(r)|}{ |\nabla(\log \psi_\varepsilon(r))|^{p}}\dvol_g=\int_{\Omega}|\na u|^p \frac{\phi(r)^{p-1}|\psi'(r)|}{ |\nabla(\log \psi(r))|^{p}}\dvol_g.
\end{equation}
Now \eqref{Hardy-Inequ-sec-upper-bound} directly follows by \eqref{Hardy-Inequ-sec-upper-bound00}--\eqref{conlimi2}.
\end{proof}

 \begin{remark}The non-compactness of $\Omega$   guarantees that ``$C^\infty_0(\Omega)$ does not contain constant functions". Note that \eqref{Hardy-Inequ-sec-upper-bound} fails for constant functions and hence, this condition is necessary here.
It is noticeable that the non-compactness of $\Omega$ {\it does not} mean that the closure $\overline{\Omega}$ is noncompact; for instance, consider a connected  proper open subset in $\mathbb{S}^m$.

We also emphasize that since $\Sigma \cap \partial\Omega_\Sigma\neq \emptyset$, nonzero constant functions cannot belong to $C^\infty_0(\Omega_\Sigma)$. Hence, such an assumption is unnecessary in both   Theorems \ref{Hardy-thm-general} and Theorem \ref{generalhardclosed} (see below).
\end{remark}

In the sequel we investigate the sharpness of (\ref{Hardy-Inequ-sec-upper-bound}).
Note that if $\Sigma=\partial \Omega$, then $\Omega_\Sigma=\Omega$, in which case the sharpness has been discussed (see Section~\ref{sharpnessfirst}). In the sequel, we focus on the case $\Omega\cap \Sigma\neq\emptyset$.

\begin{theorem}\label{sharp-u-general} Let $(M,g)$, $\Sigma$, $\Omega$, $(\phi,\psi)$  be as in Theorem \ref{Hardy-thm-2}.
Then \eqref{Hardy-Inequ-sec-upper-bound} is sharp if the following conditions hold$:$
\begin{enumerate}[{\rm (i)}]

\item\label{sigmcompactness} $\Sigma$ is compact and  $\Sigma\subset\Omega;$

\smallskip

\item\label{nonintegrable} there exists $\delta_0>0$   such that $\phi(r)^{p-1}|\nabla(\log \psi(r))|\notin L^1(\T_\delta)$ for any $\delta\in (0,\delta_0)$.
\end{enumerate}
\end{theorem}

\begin{proof} The compactness of  $\Sigma$ implies $\mathfrak{c}_{\mathcal {V}}(\Sigma)>0$.
Since $\Sigma\subset\Omega$, choose $\iota\in (0,4^{-1}\min\{\delta_0, \mathfrak{c}_{\mathcal {V}}(\Sigma)\})$ such that the closure of $4\iota$-tuber neighborhood $\overline{\T_{4\iota}}$  is compact and contained in $\Omega$.

Let $\eta\in C^\infty_0(M)$ be a cut-off function with $0\leq \eta\leq 1$ and
\begin{equation}\label{cutoffepsilionfunction}
\eta(x)=1\ \text{ if }x\in \T_{\iota};\qquad \eta(x)=0\ \text{ if }x\notin \T_{2   \iota   }.
\end{equation}
For any $\e\in (0,\iota)$, we define the function
\begin{equation}\label{boundaryfunctionpsi}
v_\e(x):=
\begin{cases}
\psi(\e)^{-1/p},& \text{ if } x\in\T_{\e};\\
\\
 \psi(r(x))^{-1/p},& \text{ if } x\in \Omega\backslash {\T_{\e}}.
\end{cases}
\end{equation}
 Note that $\eta  v_\e$ is  compactly supported in $\overline{\T_{2\iota}}$ and particularly,
\[
\supp |\na (\eta v_\e)|^p\subset \overline{{\T}_{2\iota}},\quad  \sup_{x\in {\T}_{2\iota}} |\na (\eta v_\e)|^p(x)<\infty,
\]
which combined with Assumption \ref{Basciasumptionfunction3}/\eqref{integrable3}, \eqref{cutoffepsilionfunction} and \eqref{boundaryfunctionpsi} yields
\begin{equation}\label{seondierisfinitevsna}
\int_{{\T}_{2\iota} \setminus {\T}_{\iota}}|\na (\eta v_\e)|^p \frac{\phi(r)^{p-1}|\psi'(r)|}{ |\nabla(\log \psi(r))|^{p}}\dvol_g \text{ is finite and independent of the choice of $\e$}.
\end{equation}
On the other hand,   \eqref{nonintegrable}    means
\begin{align}
\lim_{\e\rightarrow 0^+}\int_{\T_{\iota} \setminus {\T}_{\e} }\phi(r)^{p-1}|\nabla(\log \psi(r))|\dvol_g=+\infty.\label{infinitecondit}
\end{align}
From above, we have
\begin{align}
&\int_{\Omega}|\na (\eta v_\e)|^p \frac{\phi(r)^{p-1}|\psi'(r)|}{ |\nabla(\log \psi(r))|^{p}}\dvol_g\notag\\
=& \int_{\T_{\iota}}|\na v_\e|^p \frac{\phi(r)^{p-1}|\psi'(r)|}{ |\nabla(\log \psi(r))|^{p}}\dvol_g+\int_{{\T}_{2\iota} \setminus \T_{\iota}}|\na (\eta v_\e)|^p \frac{\phi(r)^{p-1}|\psi'(r)|}{ |\nabla(\log \psi(r))|^{p}}\dvol_g\nonumber\\
=&p^{-p}\int_{\T_{\iota}\setminus {\T}_{\e} }\phi(r)^{p-1}|\nabla(\log \psi(r))|\dvol_g+\int_{{\T}_{2\iota} \backslash \T_{\iota}}|\na (\eta v_\e)|^p \frac{\phi(r)^{p-1}|\psi'(r)|}{ |\nabla(\log \psi(r))|^{p}}\dvol_g,\label{fracl1}
\end{align}
and
\begin{align}
&\int_{\Omega} |\eta v_\e|^p\phi(r)^{p-1}|\psi'(r)|\dvol_g\geq \int_{\T_{\iota}} |v_\e|^p\phi(r)^{p-1}|\psi'(r)|\dvol_g
\geq\int_{\T_{\iota}\setminus {\T}_{\e}}\phi(r)^{p-1}|\nabla(\log \psi(r))|\dvol_g.\label{fracl2}
\end{align}
Since $\eta v_\e\in D^{1,p}(\Omega,\phi,\psi)$, see Definition \ref{D1pspace},
by  \eqref{functinoalshapr}, \eqref{Hardy-Inequ-sec-upper-bound}, \eqref{fracl1}, \eqref{fracl2}, \eqref{seondierisfinitevsna} and \eqref{infinitecondit} we obtain
\begin{align}  p^{-p}\leq&\inf_{u\in C^\infty_0(\Omega)\setminus\{0\}}\mathcal {J}_{p,\Omega}(u)\leq \mathcal {J}_{p,\Omega}(\eta v_\e)
\leq p^{-p}+{\int_{{\T}_{2\iota} \setminus {\T}_{\iota}}|\na (\eta v_\e)|^p \frac{\phi(r)^{p-1}|\psi'(r)|}{ |\nabla(\log \psi(r))|^{p}}\dvol_g\over \int_{{\T}_{\iota}\setminus {\T}_{\e}}\phi(r)^{p-1}|\nabla(\log \psi(r))|\dvol_g}\rightarrow p^{-p},\nonumber
\end{align}
as  $\varepsilon\to 0^+$, which concludes the proof.
\end{proof}

\begin{remark}\label{conditioncheck} A sufficient condition for \eqref{nonintegrable} in Theorem \ref{sharp-u-general} is that there exists $\delta_0>0$ such that
	\begin{equation}\label{infiniteintegral1}
		\lim_{\e\rightarrow0^+}\int_\e^{ \delta}\phi(t)^{p-1}|(\log \psi(t))'|\, t^{m-n-1} {\dd}t=+\infty,\quad \forall\delta\in (0,\delta_0).
	\end{equation}
		In fact, by \eqref{smallsestimatedatA}, we may assume  that $\iota>0$ satisfies   (see the beginning of the proof of Theorem \ref{sharp-u-general}):
	\[
	2^{-1}t^{m-n-1}\leq \det \mathcal {A}(t,\mathbf{n})\leq 2 t^{m-n-1}, \text{ for }(t,\mathbf{n})\in[0,4\iota]\times \mathcal {V}S\Sigma.
	\]
	Thus, \eqref{dv-g-Def} combined with \eqref{infiniteintegral1} yields
	\begin{align*}
		&\int_{\T_{\iota} \setminus {\T}_{\e} }\phi(r)^{p-1}|\nabla(\log \psi(r))|\dvol_g\nonumber\\
		\geq&~ \frac12{\vol(\mathbb{S}^{m-n-1})} \, {\vol}_{i^*g}(\Sigma)  \int_\varepsilon^{\iota}\phi(t)^{p-1}|(\log \psi(t))'|t^{m-n-1}{\dd}t\rightarrow +\infty, \ \text{ as }\e\rightarrow 0^+,
	\end{align*}
	which implies $\phi(r)^{p-1}|\nabla(\log \psi(r))|\notin L^1(\T_{\iota})$ and hence, \eqref{nonintegrable} follows.
	Moreover, due to the fact that
	\[
	\lim_{t\rightarrow 0^+}\frac{\mathbf{c}_\lambda(t)^n\mathbf{s}_\lambda(t)^{m-n-1}}{t^{m-n-1}}=1,
	\]
	it is easy to see that
	\eqref{infiniteintegral1} is equivalent to that for one (and hence, for every) $\lambda\in \mathbb{R}$, there exists some $\delta_\lambda>0$ such that
	\begin{equation}\label{infiniteintegral}
		\lim_{\e\rightarrow0^+}\int_\e^{\delta}\phi(t)^{p-1}|(\log \psi(t))'| \mathbf{c}_\lambda(t)^n\mathbf{s}_\lambda(t)^{m-n-1} {\dd}t=+\infty, \quad \forall\,\delta\in (0,\delta_\lambda).
	\end{equation}
	
\end{remark}

Finally, we also check the attainability of the best constant in  \eqref{Hardy-Inequ-sec-upper-bound} by any function $C^\infty_0(\Omega)\backslash\{0\}$.

\begin{theorem}\label{nonexiglobal} Let $(M,g)$, $\Sigma$, $\Omega$ and $(\phi,\psi)$  be as in Theorem \ref{Hardy-thm-2}.  Then
	for every $p>1$, there is no extremal functions for \eqref{Hardy-Inequ-sec-upper-bound}, i.e.,
	\begin{equation}\label{basicomeglarge2}
		\mathcal {J}_{p,\Omega}(u)>p^{-p},\quad \forall\,u\in C^\infty_0(\Omega)\backslash\{0\},
	\end{equation}
	where $\mathcal {J}_{p,\Omega}$ is defined by \eqref{functinoalshapr}.
\end{theorem}
\begin{proof}By Assumption \ref{Basciasumptionfunction3}, one can easily generalize Proposition \ref{prop-bv-improve} to $u\in C^\infty_0(\Omega)$. Due to Theorem \ref{Hardy-thm-2}, suppose by contradiction that $\mathcal {J}_{p,\Omega}(u)=p^{-p}$ for some $u\in C^\infty_0(\Omega)\backslash\{0\}$.
	Similarly to the proof of Theorem \ref{achieve-thm}, we have $\na(\psi(r(x))^{1/p}u(x))=0$ for $\vol_g$-a.e.  $x\in \Omega$. Since $\psi(r(x))^{1/p}u(x)$ is continuous in the connected domain $\Omega$, there is a constant $C$ such that $\psi(r(x))^{1/p}u(x)=C$ for every $x\in \Omega$.

	On the one hand, since $\Omega$ is noncompact, we have $\supp u\subsetneq \Omega$. Hence, there exists some point $x_0\in \Omega$ such that $u(x_0)=0$, which implies $C=0$.
	On the other hand,
	since $u\neq0$, there exists some point $x_1\in \Omega$ such that $u(x_1)\neq0$. Since $\{\psi(r)=0\}$ is a  $\vol_g$-negligible set,  the continuity of $u$ yields  some point $x_2\in \Omega$ such that $u(x_2)\neq0$ and $\psi(r(x_2))> 0$, which  means $C\neq0$, a contraction.
\end{proof}

\begin{example}
Let $\MM_\lambda, \Sigma$ and  $(\phi,\psi)$ be as in Example  \ref{space-form-phi-psi4}, i.e., $\MM_\lambda=\mathbb{R}^m$, $\Sigma=\{o\}$, $n=0$ and
\[
\phi(t)=t^{-\frac{m-n-1}{p-1}}=t^{-\frac{m-1}{p-1}},\quad \psi(t)=t^{\beta+m-n}=t^{\beta+m}.
\]
According to Example \ref{space-form-phi-psi5}, Assumption \ref{Basciasumptionfunction3} is satisfied when $p>1$ and $\beta>-m$. Moreover, for every $\delta>0$,
\[
\lim_{\e\rightarrow0^+}\int_\e^{ \delta}\phi(t)^{p-1}|(\log \psi(t))'|\, t^{m-1} {\dd}t=\lim_{\e\rightarrow0^+}\int_\e^{ \delta}\frac{(\beta+m)}{t}{\dd}t=+\infty.
\]
Hence, Theorems \ref{Hardy-thm-2}--\ref{sharp-u-general} and Remark \ref{conditioncheck} furnish the classical Hardy inequality (compare with \eqref{punthardyonRm}): for any $p>1$ and $\beta>-m$,
\begin{equation}\label{hardyonRm}
\int_{\mathbb{R}^m} |\nabla u|^p |x|^{p+\beta}{\dd}x\geq \left(\frac{\beta+m}{p}\right)^p\int_{\mathbb{R}^m}|u|^p |x|^\beta{\dd}x,\quad \forall\,u\in C^\infty_0(\mathbb{R}^m).
\end{equation}
In particular, \eqref{hardyonRm} is sharp and has no extremal functions in $C^\infty_0(\mathbb{R}^m)\backslash\{0\}$, see Theorem \ref{nonexiglobal}.
\end{example}

\subsection{Extension to the space  $C^\infty_0(\Omega,\Sigma)$}\label{Cinftyomegasigma}
Although expected, in some cases it is impossible to  extend \eqref{Hardy-Inequ-general}  to   $C^\infty_0(\Omega)$. Such phenomenon occurs for instance when $\Omega=M$ is compact; indeed, in this case the constant functions belong to $C^\infty_0(\Omega)$, and \eqref{Hardy-Inequ-sec-upper-bound} provides a contradiction. A possible way to surpass this problem is to extend \eqref{Hardy-Inequ-general} from $C^\infty_0(\Omega_\Sigma)$ to the space
\begin{equation}\label{comgsigspace}
C_0^\infty(\Omega,\Sigma):=\{u\in C_0^\infty(\Omega)\,:\,u(\Sigma)=0 \}.
\end{equation}
In this subsection, we discuss the validity of
\eqref{Hardy-Inequ-general} on $C_0^\infty(\Omega,\Sigma)$ under the following assumption.

\begin{assumption}\label{Basciasumptionfunction2}
Given $p > 1$, let $(\phi,\psi)$   be a pair of two nonnegative  $C^2$-functions defined on
some subinterval of $(0, +\infty)$  such that
\begin{enumerate}[{\rm(a)}]
 \item\label{smoothofphi2} $\phi(r),\psi(r)$ are well-defined   in $M\backslash\Sigma$;

\smallskip

\item\label{smoothofpsi2} $\psi$ is monotone and $\psi(r)$ is locally Lipschitz on $\Omega_\Sigma$;

\smallskip

\item\label{nansigcontroll2} {$\ds\dii(\sgn(\psi'(r))\,\phi(r)^{p-1}\na r)\geq 0$} in the weak sense in $\Omega_\Sigma$;

\smallskip

\item\label{integrable2} $\phi(r)^{p-1}, \phi(r)^{p-1}\psi(r),\phi(r)^{p-1}|\psi'(r)|\in L^1_{\loc}(\Omega_\Sigma)$ and $\frac{\phi(r)^{p-1}|\psi'(r)|}{ |\nabla(\log \psi(r))|^{p}}\in L^1_{\loc}(M)$;

\smallskip

\item\label{negligbileset2} $\{x\in {\Omega_\Sigma}\,:\, \phi(r(x))\,\psi(r(x))\,\psi'(r(x))=0\}$ is a $\vol_g$-negligible set;

\smallskip

\item\label{localboundedofnu} $\frac{\partial}{\partial r}\log\left(\frac{\phi(r)^{p-1}|\psi'(r)|}{ |\nabla(\log \psi(r))|^{p}}\right)$ is locally bounded (i.e., bounded on every compact set) in $M\backslash\Sigma$.

\end{enumerate}


\end{assumption}

\begin{remark}\label{assp2reason}  Obviously, \eqref{smoothofphi2}--\eqref{integrable2} are enough to obtain the Hardy inequality \eqref{Hardy-Inequ-general}; in particular,  Assumption \ref{Basciasumptionfunction2} is stronger than Assumption \ref{Basciasumptionfunction}.
More precisely, \eqref{smoothofphi2} means that $\phi'(r),\psi'(r), \psi''(r)$ are well-defined on $M\backslash\Sigma$, which is necessary to define $\frac{\partial}{\partial r}\left(\frac{\phi(r)^{p-1}|\psi'(r)|}{ |\nabla(\log \psi(r))|^{p}}\right)$. Furthermore,
the local boundedness of $\frac{\partial}{\partial r}\left(\frac{\phi(r)^{p-1}|\psi'(r)|}{ |\nabla(\log \psi(r))|^{p}}\right)$, i.e., \eqref{localboundedofnu}, is used to establish a divergence lemma \eqref{newdivelemm} and Corollary \ref{maxminlenew}, which plays an important role in a capacity   theory; for technical details, see  Appendix \ref{Soblevspace}. 
\end{remark}


We now introduce a new Sobolev space.
\begin{definition}\label{Sobolevspaces}
Let $(M,g)$ be a complete Riemannian manifold,   $i:\Sigma\hookrightarrow M$ be a closed submanifold    and let  $(\phi,\psi)$ be a pair of functions satisfying Assumption \ref{Basciasumptionfunction2}.
Given $p>1$ and  an open set  $U\subset M$, the {\it weighted Sobolev space} ${W^{1,p}_0}(U, \phi,\psi)$ (resp., ${W^{1,p}}(U, \phi,\psi)$) is defined as the completion of  $C^{\infty}_{0}(U)$ (resp., $C^{\infty}_{p}(U):=\{ u\in C^\infty(U)\,:\, \|u\|_{p}<\infty\}$) under the norm
\begin{equation}
\|u\|_{p}:=\left(\int_{U}|u|^p \frac{\phi(r)^{p-1}|\psi'(r)|}{ |\nabla(\log \psi(r))|^{p}}\dvol_g+\int_{U}|\na u|^p \frac{\phi(r)^{p-1} |\psi'(r)|}{|\nabla(\log \psi(r))|^{p}}\dvol_g\right)^{1/ p}.\label{w-norm}
\end{equation}

\end{definition}
In Appendix \ref{Soblevspace} we provide further properties of the spaces  $W^{1,p}_0(U, \phi,\psi)$ and $W^{1,p}(U, \phi,\psi)$. In view of \eqref{comgsigspace}, we have the following lemma.

\begin{lemma}\label{continuouslemma}
 Let $(M,g)$, $\Sigma$ and $(\phi,\psi)$  be as in Definition \ref{Sobolevspaces}. If
 $\Omega\subset M$ is a non-empty $($possibly compact$)$ domain  satisfying \eqref{natural-domain}, then $C^\infty_0(\Omega,\Sigma)\subset W^{1,p}_0(\Omega_\Sigma,\phi,\psi)$, that is,
\begin{equation*}
  u\in C^\infty_0(\Omega,\Sigma)\quad  \Longrightarrow\quad  u|_{\Omega_\Sigma}\in  W^{1,p}_0(\Omega_\Sigma,\phi,\psi).\label{functionsspacessubset}
\end{equation*}
\end{lemma}
\begin{proof}
Since the proof is trivial if $\Sigma= \partial\Omega$, it  suffice to prove the statement when
$\Sigma\cap \Omega\neq\emptyset$. Given $u\in C^\infty_0(\Omega,\Sigma)$,  the zero extension furnishes $u\in W^{1,p}(M,\phi,\psi)$.  Since $u$ is continuous in $M$ and $u=0$ in $M\backslash\Omega_\Sigma$, Theorem \ref{maintheoreminaapendix} together with Theorem \ref{propoertyfocap}/\eqref{Cap1} yields $u|_{\Omega_\Sigma}\in   W^{1,p}_0(\Omega_\Sigma,\phi,\psi)$.
\end{proof}


Now Lemma \ref{continuouslemma} together with Theorem \ref{Hardy-thm-general} yields the following result.
\begin{theorem}\label{generalhardclosed}
Let  $(M,g)$ be an $m$-dimensional complete Riemannian manifold, $i: \Sigma\hookrightarrow M$ be an $n$-dimensional closed submanifold and let $\Omega\subset M$ be a non-empty $($possibly compact$)$ domain  satisfying \eqref{natural-domain}. Given $p>1$, if  $(\phi,\psi)$ satisfies Assumption \ref{Basciasumptionfunction2}, then
\begin{equation}
\int_{\Omega}|\na u|^p\frac{\phi(r)^{p-1}|\psi'(r)|}{ |\nabla(\log \psi(r))|^{p}}\dvol_g\geq p^{-p}\int_{\Omega} |u|^p\phi(r)^{p-1}|\psi'(r)|\dvol_g,\ \ \forall u\in C^\infty_0(\Omega,\Sigma). \label{Hardy-Inequ-general2}
\end{equation}
\end{theorem}
\begin{proof}
Given $u\in W^{1,p}_0(\Omega_\Sigma, \phi,\psi)$, there is a sequence $u_i\in C^\infty_0(\Omega_\Sigma)$ such that $u_i$ converges to $u$ under $\|\cdot\|_{p}$.  By passing to a subsequence and using Lemma  \ref{nullmeaure}, we may assume that $u_i$ converges to $u$ point-wisely\ $\vol_g$-a.e. in $\Omega$.   Since $\Sigma$ is a $\vol_g$-negligible set, the inequality \eqref{Hardy-Inequ-general} together with  Fatou's lemma  yields
\begin{align*}
&\int_{\Omega}|\na u|^p\frac{\phi(r)^{p-1}|\psi'(r)|}{ |\nabla(\log \psi(r))|^{p}}\dvol_g=\underset{i\rightarrow \infty}{\lim\inf}\int_{\Omega}|\na u_i|^p\frac{\phi(r)^{p-1}|\psi'(r)|}{ |\nabla(\log \psi(r))|^{p}}\dvol_g\\
\geq& p^{-p} \underset{i\rightarrow \infty}{\lim\inf}\int_{\Omega}  |u_i|^p\phi(r)^{p-1}|\psi'(r)|\dvol_g\geq p^{-p} \int_{\Omega}  \underset{i\rightarrow \infty}{\lim\inf}|u_i|^p\phi(r)^{p-1}|\psi'(r)|\dvol_g\\
=&p^{-p}\int_{\Omega} |u|^p\phi(r)^{p-1}|\psi'(r)|\dvol_g,
\end{align*}
which combined with Lemma \ref{continuouslemma} yields
 \eqref{Hardy-Inequ-general2}.
\end{proof}

\begin{remark}\label{sharpnonextronclosed}
Since $C^\infty_0(\Omega_\Sigma)\subset C^\infty_0(\Omega,\Sigma)$, the sharpness results in Section \ref{sharpnessfirst} remain valid for \eqref{Hardy-Inequ-general2}. In particular, \eqref{Hardy-Inequ-general2} is sharp whenever \eqref{Hardy-Inequ-general} (under Assumption \ref{Basciasumptionfunction}) is sharp.
 Moreover, there is   no extremal functions for \eqref{Hardy-Inequ-general2} if
  one of the following conditions holds:
\begin{enumerate}[{\rm (i)}]

\item\label{limteexideoin2} $\Omega$ is unbounded;

\item\label{limteexideoin1}  $\lim_{t\rightarrow 0^+}\psi(t)<+\infty$;

\item\label{limteexideoin3}    for any open set $V\subset \Omega_\Sigma$ such that $\overline{V}\cap \Sigma$ is not a $\vol_{i^*g}$-negligible set, there holds
\[
\int_V\phi(r)^{p-1}|\nabla(\log \psi(r))|\dvol_g=+\infty.
\]
\end{enumerate}
We only deal with \eqref{limteexideoin3}, as the proofs of \eqref{limteexideoin2} and \eqref{limteexideoin1} are based on similar arguments as in Theorems \ref{nonexiglobal} and \ref{achieve-thm}, respectively. If there is some $u\in C^\infty_0(\Omega,\Sigma)\setminus \{0\}$ such that  $\mathcal {J}_{p,\Omega}(u)=p^{-p}$,
a similar argument as in Theorem \ref{achieve-thm} yields that  $x\mapsto \psi(r(x))^{1/p}u(x)$ is   constant on each connected component of $\Omega_\Sigma$; fix such a connected component $U$ and constant $C_U>0$ so that $\psi(r(x))^{1/p}u(x)=C_U$ in $U$.
%
%
%
Provided that \eqref{limteexideoin3} holds,
choose a bounded open set $V\subset \Omega_\Sigma$ such that $\overline{V}\cap \Sigma$ is not a $\vol_{i^*g}$-negligible set. For convenience, we may assume that $V$ is contained in the closure of $U$; otherwise,  consider $V\cap U$.
Since $K:=\overline{V}$  is a compact set,
   Assumption \ref{Basciasumptionfunction2}/\eqref{integrable2} yields that
$
|\na u|^p \frac{\phi(r)^{p-1} |\psi'(r)|}{|\nabla(\log \psi(r))|^{p}}\in L^1(K).$
On the other hand, since $\psi(r(x))^{1/p}u(x)=C_U$ in $U$, we have by \eqref{limteexideoin3} that
\[
\int_{K}|\na u|^p \frac{\phi(r)^{p-1} |\psi'(r)|}{|\nabla(\log \psi(r))|^{p}}\dvol_g=\frac{C_U^p}{p^p}\int_V  \phi(r)^{p-1}|\nabla \log \psi(r)|\dvol_g=+\infty,
\]
 a contradiction.

\end{remark}

Theorem \ref{generalhardclosed} is efficient to deal with  compact domains; in the sequel, we present an application.
\begin{example}
Let $\MM_\lambda=\mathbb{S}^m=\Omega$ and let $\Sigma=\mathbb{S}^{m-1}$, the great sphere in $\mathbb{S}^m$. Choose $(\phi,\psi)$ as
\[
\phi(t):=t^{-\frac{m-n-1}{p-1}}=1,\quad \psi(t):=t^{\beta+m-n}=t^{\beta+1},\quad -(p+1) <\beta<-1.
\]
Assumption \ref{Basciasumptionfunction2}/\eqref{smoothofphi2}\eqref{smoothofpsi2}\eqref{negligbileset2}\eqref{localboundedofnu}  obviously hold and \eqref{nansigcontroll2} has been proven in Example \ref{space-form-phi-psi2}.
Moreover,
\eqref{dv-g-Def} and \eqref{narspaceform}$_1$ imply $\frac{\phi(r)^{p-1}|\psi'(r)|}{ |\nabla(\log \psi(r))|^{p}}\in L^1(\mathbb{S}^m)$. Thus, Assumption \ref{Basciasumptionfunction2} is satisfied.

Theorem \ref{generalhardclosed} together with Remark \ref{sharpnonextronclosed}, Corollary \ref{geomertricmeaningofsharpness2} and \eqref{conditionspaceform1} yields the following sharp Hardy inequality: for any $p>1$ and $-1>\beta>-(p+1)$,
 \begin{equation*}
\int_{\mathbb{S}^m}|\nabla u|^p r^{\beta+p}\dvol_g\geq \left|  \frac{\beta+1}{p} \right|^p \int_{\mathbb{S}^m}|u|^p r^\beta \dvol_g,\quad \forall\,u\in C^\infty_0(\mathbb{S}^m,\mathbb{S}^{m-1}),
\end{equation*}
which  has no extremal functions in $C^\infty_0(\mathbb{S}^m,\mathbb{S}^{m-1})\backslash\{0\}$ due to   Condition \eqref{limteexideoin3} in Remark \ref{sharpnonextronclosed}.
\end{example}

\section{Application I: Sharp Hardy inequalities on manifolds with upper curvature bounds}\label{applicII}

As we pointed out, classical Hardy inequalities have been extended to Cartan--Hadamard manifolds, see Carron \cite{Ca}, Krist\'aly \cite{Kristaly},  Yang et al. \cite{YSK}. In this section, we consider a more general case when the curvature is  bounded from above; more precisely, by recalling \eqref{defunbm}, we consider the following setting:

\begin{definition}\label{Curvature-Upper} Let $(M,g)$ be an $m$-dimensional Riemannian manifold, $i:\Sigma\hookrightarrow M$ be an $n$-dimensional closed {\it totally umbilical} submanifold and $\Omega$ be a non-empty domain in $M$ satisfying \eqref{natural-domain}. Given $\lambda,\kappa\in \mathbb{R}$, the triple
 $(M,\Sigma,\Omega)$  satisfies the {\it  upper curvature bound condition $(\lambda,\kappa)$}  if one of the following  conditions holds:
\begin{enumerate}[{\rm (a)}]
\item\label{0diem2}    ${\sec}_M\leq \lambda$ and $n=0$;

\smallskip

\item\label{m-1diem2}   ${\sec}_M\leq \lambda$, $n=m-1$ and $\tr \mathfrak{A}^\mathbf{n}\leq n\kappa$ for any inward normal vector $\mathbf{n}$  w.r.t. $\Omega$;

\smallskip

\item\label{other-dime2}  $\sec_M\leq \lambda$, $1\leq n\leq m-2$ and  $\tr \mathfrak{A}^\mathbf{n}\leq n\kappa$ for any $\mathbf{n}\in \mathcal {V}S\Sigma$.

\end{enumerate}
\end{definition}
\begin{remark}\label{explprincipcuvrature1}

Since $\Sigma$ is a totally umbilical submanifold,  for a given $\mathbf{n}\in \mathcal {V}S\Sigma$  the eigenvalues $\{\kappa_i\}_{i=1}^n$ of $\mathfrak{A}^{\mathbf{n}}$ with respect to every orthonormal basis of $T_{\pi(\mathbf{n})}\Sigma$ are equal, i.e.,
$\kappa_1=\cdots=\kappa_n$.
Thus, $\tr \mathfrak{A}^\mathbf{n}\leq n\kappa$ implies each $\kappa_i\leq \kappa$. Moreover,
\begin{itemize}

\item \eqref{0diem2} means that $\Sigma$ is a point and hence, the Weingarten map always vanishes, in which case  $\kappa$ can be chosen as any real number (because it is meaningless);

\smallskip

\item \eqref{other-dime2} implies $\kappa\geq 0$ because $-n\kappa\leq -\tr \mathfrak{A}^\mathbf{n}=\tr \mathfrak{A}^\mathbf{-n}\leq n\kappa$. In particular, if $\kappa=0$ then $0=\tr \mathfrak{A}^\mathbf{n}=n\kappa_i$ for each $\kappa_i$, i.e., $\kappa_i=0$ and therefore, $\Sigma$ is actually totally geodesic.

\end{itemize}
According to the last statement -- although in the sequel we will not point out explicitly the sign of $\kappa$ -- we should keep in mind that $\kappa\geq 0$ whenever $1\leq n\leq m-2$.
\end{remark}

\subsection{Hardy inequalities involving double-curvature  weights
}\label{section-5-1}

In the sequel we use $(\lambda,\kappa)\leq (\Lambda,K)$ to denote $\lambda\leq \Lambda$ and $\kappa\leq K$, respectively.

\begin{theorem}\label{distance-weight-hardy-genera255} Let $(M,g)$ be an $m$-dimensional complete Riemannian manifold, $i:\Sigma \hookrightarrow M$ be an $n$-dimensional closed totally umbilical submanifold and $\Omega$ be a noncompact domain in $M$ satisfying \eqref{natural-domain}.  Denote $r(x):=d(\Sigma,x)$ by the distance function from $\Sigma$.
Suppose that $(M,\Sigma,\Omega)$ satisfies the upper curvature bound condition $(\lambda,\kappa)$.
Then for  any $p>1$, $\beta>-(m-n)$ and any pair $(\Lambda,K)$ satisfying one of the following conditions
\begin{enumerate}[{\quad \rm (i) }]

\item \label{LKUPPER2} $(\lambda,\kappa)\leq (\Lambda,K)\leq (0,0);$

\smallskip

\item \label{LKUPPER} $(\lambda,\kappa)\leq (\Lambda,K)$, $\Lambda\leq 0$, $K\geq  0$  and $K^2\leq -\Lambda$,

\end{enumerate}
there holds
\begin{equation}\label{Hardy-Ineq-curvature-condition52}
\int_{\Omega}|\na u|^p\frac{\mathbf{c}_\Lambda(r)^{1-p}\mathbf{s}_\Lambda(r)^{p+\beta}}{(\mathbf{c}_\Lambda(r)-K \mathbf{s}_\Lambda(r))^n}\dvol_g\geq \left({\beta+m-n\over p}\right)^p\int_{\Omega} |u|^p\frac{\mathbf{c}_\Lambda(r)\mathbf{s}_\Lambda(r)^\beta}{(\mathbf{c}_\Lambda(r)-K \mathbf{s}_\Lambda(r))^n} \dvol_g,
\end{equation}
for any $u\in C^\infty_0(\Omega)$.
Moreover, \eqref{Hardy-Ineq-curvature-condition52} is sharp if either $\Sigma$ is compact with $\Sigma\subset \Omega$ or $\Omega$ is  bounded with $\partial\Omega=\Sigma,$ and there is no extremal in either case.
\end{theorem}

\begin{proof}
Let  $(\phi,\psi)$ be the pair of functions defined as
\begin{equation}\label{newphipsi-2}
\phi(t):=(\mathbf{c}_\Lambda(t)-K\mathbf{s}_\Lambda(t))^{-\frac{n}{p-1}}\mathbf{s}_\Lambda(t)^{-\frac{m-n-1}{p-1}}, \quad \psi(t):=\mathbf{s}_\Lambda(t)^{\beta+m-n}, \quad \forall\,t\in (0,\min\{\mathfrak{r}_{\Lambda},\mathfrak{t}_{\Lambda,K} \}).
\end{equation}
It follows from \eqref{estimatecklambda}, \eqref{footcomprison} and \eqref{tlkdefin} that
$ \mathfrak{r}_{\Lambda}=+\infty= \mathfrak{t}_{\Lambda,K}$.

We claim that $(\phi,\psi)$ satisfies Assumption \ref{Basciasumptionfunction3}.
Indeed, Assumption \ref{Basciasumptionfunction3}/\eqref{smoothofphi3}\eqref{smoothofpsi3}\eqref{negligbileset3} are obviously true.
Besides,  \eqref{Laplace-comparison-reverse2} together with Proposition \ref{gestiamtegra} yields
 \begin{align*}
\Delta r\geq G_{\lambda,\kappa}(r)\geq G_{\Lambda,K}(r)=-(p-1)\frac{\phi'(r)}{\phi (r)},\quad \forall\,x\in  \mathcal {V}\mathscr{D}\Sigma\backslash\Sigma,
\end{align*}
which combined with \eqref{plaplacephi} and $\psi'>0$ furnishes Assumption \ref{Basciasumptionfunction3}/\eqref{nansigcontroll3}.
Moreover,  if $\alpha:=\beta+m-n>0$, a direct calculation yields
\begin{align*}
\phi(t)^{p-1}&\psi(t)^s\sim t^{-(m-n-1)+s\alpha},\quad \phi(t)^{p-1}|\psi'(t)|=\alpha(\mathbf{c}_\Lambda(t)-K\mathbf{s}_\Lambda(t))^{-n}\mathbf{s}_\Lambda(t)^\beta\mathbf{c}_\Lambda(t)\sim \alpha t^\beta,\\
&\frac{\phi(t)^{p-1}|\psi'(t)|}{|(\log \psi(t))'|^p}=\alpha^{1-p}(\mathbf{c}_\Lambda(t)-K\mathbf{s}_\Lambda(t))^{-n}\mathbf{s}_\Lambda(t)^{p+\beta}\mathbf{c}_\Lambda(t)^{1-p}\sim \alpha^{1-p}t^{p+\beta},
 \end{align*}
 as $t\to 0^+$. Thus, \eqref{smallsestimatedatA}$_1$ combined with $\mathfrak{r}_{\Lambda}=\mathfrak{t}_{\Lambda,K}=+\infty$ furnishes Assumption \ref{Basciasumptionfunction3}/\eqref{derivativeintegr}\eqref{integrable3}. Therefore, the claim is true and \eqref{Hardy-Ineq-curvature-condition52} follows by Theorem  \ref{Hardy-thm-2} directly.
 We now discuss the sharpness.

 {\color{blue}(a)} Suppose that $\Sigma$ is compact with  $\Sigma\subset \Omega$.
 Thus, owing to Theorem \ref{sharp-u-general} and Remark \ref{conditioncheck}, it suffices to check \eqref{infiniteintegral}. By choosing a small $\delta_\Lambda\in (0,c_{\mathcal{V}}(\Sigma))$ with
 \[
 \min_{t\in [0,\delta_\Lambda]}\left(\frac{\mathbf{c}_\Lambda(t)}{\mathbf{c}_\Lambda(t)-K \mathbf{s}_\Lambda(t)}\right)^n\geq \frac12,
 \]
 we obtain that for any  $\delta\in (0,\delta_\Lambda)$,
 \begin{align*}\label{infiniteintegral-2}
&\lim_{\e\rightarrow0^+}\int_\e^{\delta}\phi(t)^{p-1}|(\log \psi(t))'| \mathbf{c}_\Lambda(t)^n\mathbf{s}_\Lambda(t)^{m-n-1} {\dd}t\\
=&\alpha\lim_{\e\rightarrow0^+}\int_\e^{\delta}(\mathbf{c}_\Lambda(t)-K\mathbf{s}_\Lambda(t))^{-n}\mathbf{s}_\Lambda(t)^{-1}\mathbf{c}_\Lambda(t)^{n+1} {\dd}t\nonumber\\
\geq&\frac{\alpha}2\lim_{\e\rightarrow0^+}\int_\e^{\delta}\mathbf{s}_\Lambda(t)^{-1}\mathbf{c}_\Lambda(t) {\dd}t=+\infty,
\end{align*}
i.e., \eqref{infiniteintegral} holds.

 {\color{blue}(b)} Suppose that $\Omega$ is a bounded domain with $\partial\Omega=\Sigma$, in which case $n=m-1$, see \eqref{natural-domain}. Let $D:=\sup_{x\in \Omega}r(x)<+\infty$. The completeness of $M$ implies the compactness of $\Sigma$ and hence, $c_\mathcal {V}(\Sigma)>0$. Owing to \eqref{smallsestimatedatA}$_1$, there exists $s_0\in (0,\min\{D,c_\mathcal {V}(\Sigma)\})$ such that
\[
 \det \mathcal {A}(t,\mathbf{n})\leq 2 \mathbf{s}_\Lambda(t)^{m-n-1}=2, \text{ for }(t,\mathbf{n})\in[0,s_0]\times \mathcal {V}S\Sigma,
\]
which together with \eqref{dv-g-Def} yields for every $\varepsilon>0$ that
\begin{align}
&\int_{{\T_{s_0}}\cap\,\Omega}\psi(r)^{1+\varepsilon}\phi(r)^{p-1}|\psi'(r)|\dvol_g\leq 2\alpha \vol_{i^*g}(\Sigma) \int^{s_0}_0 (\mathbf{c}_\Lambda(t)-K\mathbf{s}_\Lambda(t))^{-n} \mathbf{s}_\Lambda(t)^{(2+\varepsilon)\alpha-1}\mathbf{c}_\Lambda(t){\dd}t\notag\\
\leq & 2\alpha \vol_{i^*g}(\Sigma)\max_{t\in [0,s_0]} (\mathbf{c}_\Lambda(t)-K\mathbf{s}_\Lambda(t))^{-n}\int^{s_0}_0 \mathbf{s}_\Lambda(t)^{(2+\varepsilon)\alpha-1}\mathbf{c}_\Lambda(t){\dd}t<+\infty.\label{psi1+eisfinite}
\end{align}
On the other hand, we have
\[
\max_{t\in [s_0,D] }\psi(t)^{-(1+\varepsilon)}\phi(t)^{p-1}|\psi'(t)|\leq  \alpha \mathbf{c}_\Lambda(D) \mathbf{s}_\Lambda(s_0)^{-\alpha\varepsilon-(m-n) }\max_{t\in [s_0,D]} (\mathbf{c}_\Lambda(t)-K\mathbf{s}_\Lambda(t))^{-n}   <+\infty,
\]
which together with the boundedness of $\Omega$ implies $\psi(r)^{-(1+\varepsilon)}\phi(r)^{p-1}|\psi'(r)|\in L^1(\Omega\backslash{\T_{s_0}})$. Therefore, we have the properties from  \eqref{Ambrosio-condition-1}, and the sharpness follows by Corollary \ref{geomertricmeaningofsharpness2}.

 The non-existence of extremal function is a direct consequence of Theorem \ref{nonexiglobal}.
\end{proof}

\begin{remark}
 Here we explain Condition \eqref{LKUPPER} from Theorem \ref{distance-weight-hardy-genera255}.
Since $\Omega$ can be unbounded, we need $\mathfrak{r}_{\Lambda}=\mathfrak{c}_{\Lambda,K}=+\infty$, in order to guarantee the well-definiteness of $\phi(r),\psi(r)$ on $\Omega_\Sigma$ constructed in \eqref{newphipsi-2}. Note that by \eqref{r-0-Def} and \eqref{footcomprison} we have  $\mathfrak{r}_\Lambda\leq \mathfrak{r}_\lambda$ and  $\mathfrak{c}_{\Lambda,K}\leq \mathfrak{c}_{\lambda,\kappa}$ whenever $(\lambda,\kappa)\leq (\Lambda, K)$, which means $\mathfrak{r}_{\lambda}=\mathfrak{c}_{\lambda,\kappa}=+\infty$.
Due to Condition \eqref{LKUPPER2} and Remark \ref{explprincipcuvrature1}, it is natural to assume $\kappa \geq 0$, which together with \eqref{estimatecklambda} implies $\lambda\leq 0$ and $\kappa^2\leq -\lambda$. Since $(\lambda,\kappa)\leq (\Lambda, K)$, the same reason yields Condition \eqref{LKUPPER}.
\end{remark}

\begin{example}\label{hypspexampe}
Let $\MM_\lambda=\mathbb{H}^m$ (i.e., $\lambda=-1$) and let $\Sigma$ be a single point (i.e., $n=0$).
Let $\Omega=\mathbb{H}^m$, which is noncompact. Obviously, $(M,\Sigma,\Omega)$ satisfies the upper curvature bound condition $(\lambda,\kappa)=(-1,0)$.

Given $p>1$ and $\beta>-m$, by choosing $K=0$ and $\Lambda\in \{0,-1\}$, Theorem \ref{distance-weight-hardy-genera255} provides
\begin{align*}
\int_{\mathbb{H}^m}|\nabla u|^p r^{p+\beta}\dvol_g&\geq \left(  \frac{\beta+m}{p} \right)^p \int_{\mathbb{H}^m}|u|^p r^\beta \dvol_g,\\
\int_{\mathbb{H}^m}|\na u|^p\cosh(r)^{1-p}\sinh(r)^{p+\beta}\dvol_g&\geq \Big({\beta+m\over p}\Big)^p\int_{\mathbb{H}^m} |u|^p\cosh(r)\sinh(r)^\beta \dvol_g,
\end{align*}
for any $u\in C^\infty_0(\mathbb{H}^m)$.
In particular, both of them are sharp and there are no extremal functions in $C^\infty_0(\mathbb{H}^m)$.
\end{example}

\begin{example}Given $p>1$ and $\beta>-1$, let $M:=\mathbb{R}^m$ and $\Sigma :=\mathbb{S}^{m-1}$, i.e., $n=m-1$. Let $\Omega:=\mathbb{R}^m\backslash\overline{\mathbb{B}^m}$. In this case, $\sec_M=0$ and $\tr \mathfrak{A}^\mathbf{n}=-(m-1)$ for any inward normal vector $\mathbf{n}$  w.r.t. $\Omega$.
By choosing
$\Lambda=\lambda=0$ and $K=\kappa=-1$, Theorem \ref{distance-weight-hardy-genera255} implies the following Hardy inequality
\[
\int_{\mathbb{R}^m\backslash\overline{\mathbb{B}^m}}|\nabla u|^p(x) \frac{|x|^{p+\beta}}{(1+|x|)^{m-1}}{\dd}x\geq \left( \frac{\beta+1}{p} \right)^p\int_{\mathbb{R}^m\backslash\overline{\mathbb{B}^m}}|u|^p(x) \frac{|x|^{\beta}}{(1+|x|)^{m-1}}{\dd}x,\quad \forall\,u\in C^\infty_0(\mathbb{R}^m\backslash\overline{\mathbb{B}^m}).
\]
\end{example}

\begin{remark}\label{barbatis-remark}
	Theorem \ref{distance-weight-hardy-genera255} is valid for all kinds of noncompact domains; however, if we restrict our attention to domains with
	\begin{equation}\label{r finite}
			\sup_{x\in \Omega} r(x)<+\infty,
	\end{equation}
	 an improved Hardy inequality can be established in the same flavor as in Theorem \ref{distance-weight-hardy-genera255}. Indeed, inspired by  Barbatis et al. \cite[Theorem A]{BFT} (see also Chen et al. \cite{CLZ}), if $(M,g)$,   $i:\Sigma \hookrightarrow M$ and  $\Omega$ are the same objects as in Theorem \ref{distance-weight-hardy-genera255} verifying \eqref{r finite}, then for any  $p>1$ and $\beta>-(m-n)$,
	there exists a constant $\mathcal{T}:=\mathcal{T}(p,\beta,m-n)>1$ such that for any $D>0$ with $\mathbf{s}_\lambda(D)\geq \mathcal {T}\, \mathbf{s}_\lambda( \sup_{x\in \Omega}r(x) ) $,
	\begin{align}\int_{\Omega}|\na u|^p\frac{\mathbf{c}_\lambda(r)^{1-p}\mathbf{s}_\lambda(r)^{p+\beta}}{(\mathbf{c}_\lambda(r)-\kappa\mathbf{s}_\lambda(r))^n }\dvol_g \geq& |\delta|^p\int_{\Omega} |u|^p\frac{\mathbf{c}_\lambda(r)\mathbf{s}_\lambda(r)^\beta}{(\mathbf{c}_\lambda(r)-\kappa\mathbf{s}_\lambda(r))^n} \dvol_g\nonumber\\
		&+{p-1\over 2p}|\delta|^{p-2}\int_{\Omega} |u|^p\frac{\mathbf{c}_\lambda(r)\mathbf{s}_\lambda(r)^\beta}{(\mathbf{c}_\lambda(r)-\kappa\mathbf{s}_\lambda(r))^n} \log^{-2}\left({\mathbf{s}_\lambda(D)\over \mathbf{s}_\lambda(r)}\right) \dvol_g,  \label{newsharpss344}
	\end{align}
	for any $u\in C^\infty_0(\Omega)$, where $\delta:=(m-n+\beta)/p$. Furthermore,
	\eqref{newsharpss344} is sharp in the following sense
	\begin{align*}
		|\delta|^p&=\inf_{u\in C^\infty_0(\Omega)\backslash\{0\}} \frac{\int_{\Omega}|\na u|^p\frac{\mathbf{c}_\lambda(r)^{1-p}\mathbf{s}_\lambda(r)^{p+\beta}}{(\mathbf{c}_\lambda(r)-\kappa\mathbf{s}_\lambda(r))^n }\dvol_g}{\int_{\Omega} |u|^p\frac{\mathbf{c}_\lambda(r)\mathbf{s}_\lambda(r)^\beta}{(\mathbf{c}_\lambda(r)-\kappa\mathbf{s}_\lambda(r))^n} \dvol_g}, \\
		{p-1\over 2p}|\delta|^{p-2}&=\inf_{u\in C^\infty_0(\Omega)\backslash\{0\}}\frac{\int_{\Omega}|\na u|^p\frac{\mathbf{c}_\lambda(r)^{1-p}\mathbf{s}_\lambda(r)^{p+\beta}}{(\mathbf{c}_\lambda(r)-\kappa\mathbf{s}_\lambda(r))^n }\dvol_g-|\delta|^p\int_{\Omega} |u|^p\frac{\mathbf{c}_\lambda(r)\mathbf{s}_\lambda(r)^\beta}{(\mathbf{c}_\lambda(r)-\kappa\mathbf{s}_\lambda(r))^n} \dvol_g}{\int_{\Omega} |u|^p\frac{\mathbf{c}_\lambda(r)\mathbf{s}_\lambda(r)^\beta}{(\mathbf{c}_\lambda(r)-\kappa\mathbf{s}_\lambda(r))^n} \log^{-2}\Big({\mathbf{s}_\lambda(D)\over \mathbf{s}_\lambda(r)}\Big) \dvol_g}.
	\end{align*}
The proof of \eqref{newsharpss344} and its sharpness are quite technical, following the Euclidean arguments from \cite{BFT} with suitable adaptation to our curved setting as in Theorem \ref{distance-weight-hardy-genera255}; we leave the details to the interested reader.
\end{remark}


\subsection{Hardy inequalities involving logarithmic double-curvature weights}\label{section-log-1} 

In this subsection, we will establish  a Hardy inequality whose weight involves logarithmic terms.
First, we need the following proposition.
\begin{proposition}\label{basicloganalysis}
Given $\Lambda, s_1,s_2\in \mathbb{R}$, $D\in (0,\mathfrak{r}_\Lambda]\cap \mathbb{R}$, $l_1\in(0,D)$ and $l_2\in (l_1,D]$, let
\begin{align*}
H_1(s_1,s_2):=\int^{l_1}_0 \left[\log\left({\mathbf{s}_\Lambda(D)\over \mathbf{s}_\Lambda(\tau)}\right)\right]^{s_1} \mathbf{s}_\Lambda(\tau)^{s_2}\mathbf{c}_\Lambda(\tau){\dd}\tau,\quad H_2(s_1,s_2):=\int^{l_2}_{l_1} \left[\log\left({\mathbf{s}_\Lambda(D)\over \mathbf{s}_\Lambda(\tau)}\right)\right]^{s_1} \mathbf{s}_\Lambda(\tau)^{s_2}\mathbf{c}_\Lambda(\tau){\dd}\tau.
\end{align*}
Then we have
\begin{itemize}
\item $H_1$ is well-defined if either $s_1\in \mathbb{R}$, $s_2>-1$ or $s_1<-1$, $s_2=-1;$

\smallskip

\item $H_2$ is well-defined if either $s_1,s_2\in \mathbb{R}$, $l_2<D$ or $s_1>-1$, $s_2\in \mathbb{R}$, $l_2=D$.

\end{itemize}
\end{proposition}
 \begin{proof} Since $\mathbf{s}_\Lambda$ is strictly increasing on $(0,\mathfrak{r}_\Lambda]\cap \mathbb{R}$, by changing variables $\sigma:=\mathbf{s}_\Lambda(\tau)$  we obtain
 \[
 H_1(s_1,s_2)=\int^{\mathbf{s}_\Lambda(l_1)}_{0}\left[\log\left({\mathbf{s}_\Lambda(D)\over \sigma}\right)\right]^{s_1} \sigma^{s_2}{\dd}\sigma,\quad H_2(s_1,s_2)=\int^{\mathbf{s}_\Lambda(l_2)}_{\mathbf{s}_\Lambda(l_1)}\left[\log\left({\mathbf{s}_\Lambda(D)\over \sigma}\right)\right]^{s_1} \sigma^{s_2}{\dd}\sigma,
 \]
which together with   Chen et al. \cite[Lemma 3.12]{CLZ} concludes the proof.
 \end{proof}

\begin{remark}
Note that $\mathfrak{r}_\Lambda<+\infty$ if $\Lambda>0$, in which case $D$ can be chosen as $\mathfrak{r}_\Lambda$. However, $\mathfrak{r}_\Lambda=+\infty$ if $\Lambda\leq 0$, in which case $D$ cannot be $\mathfrak{r}_\Lambda$. Therefore, we use $D\in (0,\mathfrak{r}_\Lambda]\cap \mathbb{R}$ to handle both cases.
\end{remark}

Given $\Lambda,K,s_1,s_2\in\mathbb{R}$ and $D\in (0,\min\{\mathfrak{r}_\Lambda,\mathfrak{t}_{\Lambda,K}\}]\cap \mathbb{R}$, let
\begin{equation}
\phi(t):=(\mathbf{c}_\Lambda(t)-K\mathbf{s}_\Lambda(t))^{-{n\over p-1}}\mathbf{s}_\Lambda(t)^{-{m-n-1\over p-1}},\quad \psi(t):=\psi(t,s_1,s_2):=\int^t_0 \left[\log\left({\mathbf{s}_\Lambda(D)\over \mathbf{s}_\Lambda(\tau)}\right)\right]^{s_1} \mathbf{s}_\Lambda(\tau)^{s_2}\mathbf{c}_\Lambda(\tau){\dd}\tau.\label{def-psi-log-global}
\end{equation}

Basic properties of the function $\psi$ are collected in the following lemma.


\begin{lemma}\label{computation-lem-global}
	The function $\psi$  in \eqref{def-psi-log-global} is well-defined and strictly increasing in $[0,D)$ if $s_1<-1$ and $s_2\geq -1$.
	In addition, if  $s>0$ is such that $ s_1+1+s\leq0$,
then $\psi$ defined by \eqref{def-psi-log-global} satisfies
\begin{equation}
\psi(t)\leq s^{-1}\left[\log\left({\mathbf{s}_\Lambda(D)\over \mathbf{s}_\Lambda(t)}\right)\right]^{s_1+1} \mathbf{s}_\Lambda(t)^{s_2+1}=s^{-1}\left[\log\left({\mathbf{s}_\Lambda(D)\over \mathbf{s}_\Lambda(t)}\right)\right]{\mathbf{s}_\Lambda(t)\over \mathbf{c}_\Lambda(t)}\psi'(t),\ \ t\in[0,D), \label{inequality-lemm-global}
\end{equation}
with equality if $s=-(s_1+1)$ and $s_2=-1$.
\end{lemma}
\begin{proof} Set
\[
F(t):=\psi(t)-s^{-1}\left[\log\left({\mathbf{s}_\Lambda(D)\over \mathbf{s}_\Lambda(t)}\right)\right]^{s_1+1} \mathbf{s}_\Lambda(t)^{s_2+1},\quad \forall\, t\in(0,D).
\]
Clearly, $F(0):=\lim_{t\rightarrow 0^+}F(t)=0$.
Since  $\mathbf{s}_\Lambda(t)\leq \mathbf{s}_\Lambda(D)$ for $t\in [0,D)$, we have
\[
F'(t)=s^{-1}\left[\log\left({\mathbf{s}_\Lambda(D)\over \mathbf{s}_\Lambda(t)}\right)\right]^{s_1} \mathbf{s}_\Lambda(t)^{s_2}\mathbf{c}_\Lambda(t)\left[s_1+1+s-(s_2+1)\log\left({\mathbf{s}_\Lambda(D)\over \mathbf{s}_\Lambda(t)}\right)\right]\leq0,
\]
which implies $F(t)\leq F(0)=0$ for any $t\in [0,D)$, i.e., (\ref{inequality-lemm-global}) follows. In particular, if  $s=-(s_1+1)$ and $s_2=-1$, then $F'(t)=0$ and hence, $F(t)=F(0)=0$.
\end{proof}

\begin{theorem}\label{log-weight-thm-global} Let $(M,g)$ be an $m$-dimensional complete Riemannian manifold, $i:\Sigma \hookrightarrow M$ be an $n$-dimensional closed totally umbilical submanifold and $\Omega$ be a noncompact domain  satisfying \eqref{natural-domain}.
Suppose that
\begin{itemize}
\item $(M,\Sigma,\Omega)$ satisfies the upper curvature bound condition $(\lambda,\kappa);$

\item there is a finite constant $D\in   (0,\min\{\mathfrak{t}_{\lambda,\kappa},\mathfrak{r}_\lambda\}]\cap \mathbb{R}$  such that $\Omega\subset \T_D$.
\end{itemize}
 Given  $\alpha,\beta\in \mathbb{R}$ and $s>0$ satisfying  $\alpha\geq p-(m-n)$ and $\beta+1+s\leq 0$,
then for any $(\Lambda,K)\geq (\lambda,\kappa)$ with $\min\{\mathfrak{r}_\Lambda,\mathfrak{t}_{\Lambda,K}\}\geq  D$,  we have
\begin{align}
& \int_\Omega|\na u|^p\left[\log\left({\mathbf{s}_\Lambda(D)\over \mathbf{s}_\Lambda(r)}\right)\right]^{p+\beta}\frac{\mathbf{c}_\Lambda(r)^{1-p}\mathbf{s}_\Lambda(r)^{\alpha}}{(\mathbf{c}_\Lambda(r)-K\mathbf{s}_\Lambda(r))^{n}} \dvol_g\nonumber\\
\geq& \left({s\over p}\right)^p\int_\Omega |u|^p\left[\log\left({\mathbf{s}_\Lambda(D)\over \mathbf{s}_\Lambda(r)}\right)\right]^\beta\frac{\mathbf{c}_\Lambda(r)\mathbf{s}_\Lambda(r)^{\alpha-p}}{(\mathbf{c}_\Lambda(r)-K\mathbf{s}_\Lambda(r))^{n}} \dvol_g,\quad \forall\, u\in C^\infty_0(\Omega).\label{hardy-log-weight-global}
\end{align}
Moreover,  if either $\Sigma\subset\Omega$ or $\Sigma=\partial\Omega$ with
$D<\min\{\mathfrak{r}_\Lambda,\mathfrak{t}_{\Lambda,K}\}$, and
\begin{equation}
	\text{$\Sigma$ is compact},\quad \alpha=p-(m-n),\quad s=|\beta+1|,\label{sharplogcondi}
\end{equation}
then \eqref{hardy-log-weight-global} is sharp.
%
%
%
\end{theorem}

\begin{proof}
Let $\phi,\psi$ be the functions defined on  $ (0,D)$  by \eqref{def-psi-log-global} with
\begin{equation}\label{paramesss}
s_1:=\beta<-1,\quad s_2:=\alpha-p+(m-n)-1\geq -1;
\end{equation}
due to  Lemma \ref{computation-lem-global}, both of them are well-defined.  Now we claim that Assumption \ref{Basciasumptionfunction3} is satisfied.
Obviously, Assumption \ref{Basciasumptionfunction3}/\eqref{smoothofphi3}\eqref{smoothofpsi3}\eqref{negligbileset3} trivially hold.  Moreover,  \eqref{Laplace-comparison-reverse2} together with Proposition \ref{gestiamtegra} yields
\[
\Delta r\geq G_{\lambda,\kappa}(r)\geq G_{\Lambda,K}(r)=-(p-1)\frac{\phi'(r)}{\phi(r)},\quad \forall\,x\in  \mathcal {V}\mathscr{D}\Sigma \cap \Omega_\Sigma,
\]
which combined with \eqref{plaplacephi} implies  Assumption \ref{Basciasumptionfunction3}/\eqref{nansigcontroll3}.

Given a compact set $\mathscr{K}\subset \Omega$ with $\sup_{x\in \mathscr{K}}r(x)=:R<D$, we have $\sup_{t\in [0,R]}\psi(t)<+\infty$ and $\phi(r)^{p-1}\in L^1_{}(\mathscr{K})$, which implies
$\phi(r)^{p-1}\psi(r)^s \in L^1(\mathscr{K})$ for any $s\geq 0$. Thus, Assumption \ref{Basciasumptionfunction3}/\eqref{derivativeintegr} follows. It remains to check Assumption \ref{Basciasumptionfunction3}/\eqref{integrable3}. The assumption together with \eqref{paramesss}  implies $s_1+1+s=\beta+1+s\leq 0$,
which combined with
Lemma \ref{computation-lem-global} furnishes \eqref{inequality-lemm-global}. Therefore,
we obtain
\begin{align}
\frac{\phi(r)^{p-1} |\psi'(r)|}{|\nabla(\log \psi(r))|^p}=\,&(\mathbf{c}_\Lambda(r)-K\mathbf{s}_\Lambda(r))^{-n}\mathbf{s}_\Lambda(r)^{-(m-n-1)}\frac{\psi(r)^p}{|\psi'(r)|^{p-1}}\notag\\
&\leq
s^{-p}\left[\log\left({\mathbf{s}_\Lambda(D)\over \mathbf{s}_\Lambda(r)}\right)\right]^{p+\beta}\frac{\mathbf{c}_\Lambda(r)^{1-p}\mathbf{s}_\Lambda(r)^{\alpha}}{(\mathbf{c}_\Lambda(r)-K\mathbf{s}_\Lambda(r))^{n}}, \label{pp-0p-001}\\
 \phi(r)^{p-1}|\psi'(r)|=&\left[\log\left({\mathbf{s}_\Lambda(D)\over \mathbf{s}_\Lambda(r)}\right)\right]^\beta\frac{\mathbf{c}_\Lambda(r)\mathbf{s}_\Lambda(r)^{\alpha-p}}{(\mathbf{c}_\Lambda(r)-K\mathbf{s}_\Lambda(r))^{n}}.\label{pp-0p-011}
\end{align}
Thus, Assumption  \ref{Basciasumptionfunction3}/\eqref{integrable3} is verified whenever
\begin{equation}\label{biggerlogesitamate}
\left[\log\left({\mathbf{s}_\Lambda(D)\over \mathbf{s}_\Lambda(r)}\right)\right]^{p+\beta}\frac{\mathbf{c}_\Lambda(r)^{1-p}\mathbf{s}_\Lambda(r)^{\alpha}}{(\mathbf{c}_\Lambda(r)-K\mathbf{s}_\Lambda(r))^{n}},\ \left[\log\left({\mathbf{s}_\Lambda(D)\over \mathbf{s}_\Lambda(r)}\right)\right]^\beta\frac{\mathbf{c}_\Lambda(r)\mathbf{s}_\Lambda(r)^{\alpha-p}}{(\mathbf{c}_\Lambda(r)-K\mathbf{s}_\Lambda(r))^{n}}\in L^1(\mathscr{K}).
\end{equation}
In order to check this, we notice that
the assumption on $D$ yields a finite constant $C>0$ such that on $\mathscr{K}$ we have
\begin{equation}\label{basicvvr}
{\max\left\{\frac{1}{(\mathbf{c}_\Lambda(r)-K\mathbf{s}_\Lambda(r))^{n}}, \frac{\mathbf{c}_\Lambda(r)^{ -p}}{(\mathbf{c}_\Lambda(r)-K\mathbf{s}_\Lambda(r))^{n}}\right\} \leq \frac{C}{\vol(\mathbb{S}^{m-n-1})}.}
\end{equation}
Moreover,
due to \eqref{smallsestimatedatA} and \eqref{Fermimap}, there exists a small $\epsilon>0$  such that
\begin{align*}
\det \mathcal {A}(t,\mathbf{n})\leq 2\, \mathbf{s}_\Lambda(t)^{m-n-1}, \text{ for }(t,\mathbf{n})\in([0,\epsilon]\times \mathcal {V}S\Sigma)\cap \F^{-1}(\mathscr{K}),
\end{align*}
which together with \eqref{dv-g-Def},  \eqref{basicvvr} and Proposition \ref{basicloganalysis} yield
\begin{align*}
& \int_{\T_\epsilon\cap \mathscr{K}}\left[\log\left({\mathbf{s}_\Lambda(D)\over \mathbf{s}_\Lambda(r)}\right)\right]^{p+\beta}\frac{\mathbf{c}_\Lambda(r)^{1-p}\mathbf{s}_\Lambda(r)^{\alpha}}{(\mathbf{c}_\Lambda(r)-K\mathbf{s}_\Lambda(r))^{n}}\dvol_g\leq 2 C \int^\epsilon_0\left[\log\left({\mathbf{s}_\Lambda(D)\over \mathbf{s}_\Lambda(r)}\right)\right]^{p+\beta}\mathbf{s}_\Lambda(r)^{s_2+p}\mathbf{c}_\Lambda(r){\dd}r<+\infty,\\
&\int_{\T_\epsilon\cap \mathscr{K}}\left[\log\left({\mathbf{s}_\Lambda(D)\over \mathbf{s}_\Lambda(r)}\right)\right]^\beta\frac{\mathbf{c}_\Lambda(r)\mathbf{s}_\Lambda(r)^{\alpha-p}}{(\mathbf{c}_\Lambda(r)-K\mathbf{s}_\Lambda(r))^{n}}\dvol_g\leq 2C \int^\epsilon_0\left[\log\left({\mathbf{s}_\Lambda(D)\over \mathbf{s}_\Lambda(r)}\right)\right]^{ \beta}\mathbf{s}_\Lambda(r)^{s_2}\mathbf{c}_\Lambda(r){\dd}r<+\infty.
\end{align*}
On the other hand, by choosing an enough large $C>0$, the continuity of $\det \mathcal {A}(t,\mathbf{n})$ implies
\[
\det \mathcal {A}(t,\mathbf{n})\leq C<+\infty, \quad \text{ for }(t,\mathbf{n})\in([\epsilon,R]\times \mathcal {V}S\Sigma)\cap \F^{-1}(\mathscr{K}).
\]
 Thus, Proposition \ref{basicloganalysis} again furnishes
\begin{align*}
&\int_{\mathscr{K}\backslash \T_\epsilon}\left[\log\left({\mathbf{s}_\Lambda(D)\over \mathbf{s}_\Lambda(r)}\right)\right]^{p+\beta}\frac{\mathbf{c}_\Lambda(r)^{1-p}\mathbf{s}_\Lambda(r)^{\alpha}}{(\mathbf{c}_\Lambda(r)-K\mathbf{s}_\Lambda(r))^{n}}\dvol_g
\leq  C^2 \int^R_\epsilon\left[\log\left({\mathbf{s}_\Lambda(D)\over \mathbf{s}_\Lambda(r)}\right)\right]^{p+\beta}\mathbf{s}_\Lambda(r)^{\alpha}\mathbf{c}_\Lambda(r){\dd}r<+\infty,\\
&\int_{\mathscr{K}\backslash \T_\epsilon}\left[\log\left({\mathbf{s}_\Lambda(D)\over \mathbf{s}_\Lambda(r)}\right)\right]^\beta\frac{\mathbf{c}_\Lambda(r)\mathbf{s}_\Lambda(r)^{\alpha-p}}{(\mathbf{c}_\Lambda(r)-K\mathbf{s}_\Lambda(r))^{n}}\dvol_g
\leq  C^2 \int_\epsilon^R\left[\log\left({\mathbf{s}_\Lambda(D)\over \mathbf{s}_\Lambda(r)}\right)\right]^{ \beta}\mathbf{s}_\Lambda(r)^{\alpha-p}\mathbf{c}_\Lambda(r){\dd}r<+\infty.
\end{align*}
From these estimates, \eqref{biggerlogesitamate} follows and hence,
Assumption \ref{Basciasumptionfunction3}  holds.
Therefore, \eqref{hardy-log-weight-global} directly follows by \eqref{pp-0p-001}, \eqref{pp-0p-011}   and Theorem \ref{Hardy-thm-2}.

We now focus on the  sharpness of \eqref{hardy-log-weight-global}.
If $\Sigma\subset\Omega$, then \eqref{paramesss}  combined with \eqref{sharplogcondi}   implies $s=-(s_1+1)$ and $s_2=-1$. Thus, it follows  that both \eqref{inequality-lemm-global} and  \eqref{pp-0p-001} are equalities.
In this way, due to Theorem \ref{sharp-u-general} and Remark \ref{conditioncheck}, in order to show the sharpness of \eqref{hardy-log-weight-global}, it suffices to prove \eqref{infiniteintegral}. In fact, the compactness of
 $\Sigma$ implies $c_{\mathcal{V}}(\Sigma)>0$. Let  $\delta_\Lambda\in (0,\min\{D,c_{\mathcal{V}}(\Sigma)\})$ be small enough such that
 \[
\min_{t\in [0,\delta_\lambda]}  \left(\frac{\mathbf{c}_\Lambda(t)}{\mathbf{c}_\Lambda(t)-K\mathbf{s}_\Lambda(t)}\right)^n{\geq \frac12}.
 \]
Thus, for any  $\delta\in (0,\delta_\Lambda)$,
  a direct calculation together with the equality form of \eqref{inequality-lemm-global} yields
\begin{align*}
&\lim_{\e\rightarrow0^+}\int_\e^{\delta}\phi(t)^{p-1}|(\log \psi(t))'| \mathbf{c}_\Lambda(t)^n\mathbf{s}_\Lambda(t)^{m-n-1} {\dd}t=\lim_{\e\rightarrow0^+}\int_\e^{\delta} \left(\frac{\mathbf{c}_\Lambda(t)}{\mathbf{c}_\Lambda(t)-K\mathbf{s}_\Lambda(t)}\right)^n\frac{\psi'(t)}{\psi(t)} {\dd}t\nonumber\\
&{\geq \frac12}s\lim_{\e\rightarrow0^+}\int_\e^{\delta}\left[\log\Big({\mathbf{s}_\Lambda(D)\over \mathbf{s}_\Lambda(t)}\Big)\right]^{-1}\mathbf{s}_\Lambda(t)^{-1}\mathbf{c}_\Lambda(t) {\dd}t=+\infty,\nonumber
\end{align*}
which shows the validity of \eqref{infiniteintegral}, and hence  the sharpness of \eqref{hardy-log-weight-global}.

In the other case, i.e., $\Sigma=\partial\Omega$, again  both  \eqref{inequality-lemm-global} and \eqref{pp-0p-001} become equalities and  $c_{\mathcal{V}}(\Sigma)>0$. Let $s_0\in (0,\min\{D, c_{\mathcal{V}}(\Sigma)\})$ with
\begin{align*}
\det \mathcal {A}(t,\mathbf{n})\leq 2\, \mathbf{s}_\Lambda(t)^{m-n-1}=2, \text{ for }(t,\mathbf{n})\in[0,s_0]\times \mathcal {V}S\Sigma.
\end{align*}
Since $D<\mathfrak{c}_{\Lambda,K}$ and $n=m-1$, both $\psi(r)^{1+\varepsilon}$ and $\phi(r)^{p-1}$ are bounded on $\T_{s_0}\cap \Omega$. Thus, \eqref{dv-g-Def} yields
\begin{align*}
\int_{{\T_{s_0}}\cap\,\Omega}  \psi(r)^{1+\varepsilon}\phi(r)^{p-1}|\psi'(r)|\dvol_g\leq C \int_{{\T_{s_0}}\cap\,\Omega}\psi'(r)\dvol_g\leq C\int^{s_0}_0\psi'(t){\dd}t=C\psi(s_0)<+\infty,
\end{align*}
i.e., $\psi(r)^{1+\varepsilon}\phi(r)^{p-1}|\psi'(r)|\in L^1({\T_{s_0}}\cap\,\Omega)$.
On the other hand, by \eqref{inequality-lemm-global} and \eqref{pp-0p-011} we have
\begin{align*}
\psi(r)^{-(1+\varepsilon)}\phi(r)^{p-1}|\psi'(r)|=s^{(1+\varepsilon)}\left[\log\Big({\mathbf{s}_\Lambda(D)\over \mathbf{s}_\Lambda(r)}\Big)\right]^{-(1+\varepsilon)(1+s_1)+\beta} \frac{\mathbf{s}_\Lambda(r)^{-(1+\varepsilon)(s_2+1)+\alpha-p }\mathbf{c}_\Lambda(r)}{(\mathbf{c}_\Lambda(r)-K\mathbf{s}_\Lambda(r))^{n}}.
\end{align*}
Since $-(1+\varepsilon)(1+s_1)+\beta=-1-\varepsilon(s_1+1)>-1$ for any $\varepsilon>0$, Proposition \ref{basicloganalysis} yields  $\psi^{-(1+\varepsilon)}(r)\phi(r)^{p-1}|\psi'(r)|\in L^1(\Omega\backslash{\T_{s_0}})$.
Thus, \eqref{Ambrosio-condition-1} holds, which together with Corollary \ref{geomertricmeaningofsharpness2} implies the sharpness of \eqref{hardy-log-weight-global}.

The non-existence of extremal functions in \eqref{hardy-log-weight-global} is a direct consequence of Theorem  \ref{nonexiglobal}.
\end{proof}



Note that $\Omega\subset \T_D$ implies $\sup_{x\in \Omega}r(x)\leq  D$; however, the latter inequality in general is not enough to follow the argument from the proof of Theorem \ref{log-weight-thm-global} as the next example shows.

\begin{example}Let
$M:=\mathbb{R}\times \mathbb{S}$ be a cylinder equipped by the canonical metric and let $(t,s)\in (-\infty,+\infty)\times [0,2\pi)$ be a natural coordinate system of $M$.  Let $\Sigma:=\mathbb{R}\times \{0\}$, $A:=[-1,1]\times \{\pi\}$ and $\Omega:=M\backslash A$.

 Thus, $\Omega$ is a noncompact domain with $\sup_{x\in \Omega}r(x)=\pi$. Moreover, it is obvious that
 \[
 r(x)=\pi, \quad \forall\,x\in (\mathbb{R}\backslash[-1,1])\times \{\pi\}\subset \Omega.
 \]
Choose $\lambda=\Lambda=0=K=\kappa$ and $\alpha=p-(m-n)=p-1$ and let $(\phi,\psi)$ be as in the proof of Theorem \ref{log-weight-thm-global}. For the compact set $\mathscr{K}:=[2,3]\times \mathbb{S}\subset \Omega$, since $\beta<-1$,  by \eqref{pp-0p-011}
 we have
\[
 \int_{\mathscr{K}}\phi(r)^{p-1}|\psi'(r)|\dvol_g
= 2\int_0^\pi\left[\log\left({\pi\over r}\right)\right]^{ \beta}r^{-1} {\dd}r=+\infty,
\]
i.e., $\phi(r)^{p-1}|\psi'(r)| \notin L_{\loc}^1(\Omega)$.
\end{example}

\section{Application II: Sharp Hardy inequalities on manifolds with lower curvature bounds}\label{applicI}

In this section we study Hardy inequalities on manifolds with curvature bounded from below. The generic setting of the present section is described in  the following definition.

\begin{definition}\label{Curvature-Below} Let $(M,g)$ be an $m$-dimensional Riemannian manifold, $i:\Sigma\hookrightarrow M$ be an $n$-dimensional closed submanifold and $\Omega$ be a non-empty domain in $M$ satisfying \eqref{natural-domain}. Given $\lambda,\kappa\in \mathbb{R}$, the triple
 $(M,\Sigma,\Omega)$ satisfies the {\it lower curvature bound condition $(\lambda,\kappa)$}  if one of the following conditions holds:
\begin{enumerate}[{\rm (a)}]
\item\label{0diem}    ${\Ric}_M\geq \lambda$ and $n=0$;

\smallskip

\item\label{m-1diem}   ${\Ric}_M\geq \lambda$, $n=m-1$ and $\tr \mathfrak{A}^\mathbf{n}\geq n\kappa$ for any inward normal vector $\mathbf{n}$  w.r.t. $\Omega$;

\smallskip

\item\label{other-dime}  $\sec_M\geq \lambda$, $1\leq n\leq m-2$ and  $\tr \mathfrak{A}^\mathbf{n}\geq n\kappa$ for any $\mathbf{n}\in \mathcal {V}S\Sigma$.

\end{enumerate}
\end{definition}

\begin{remark}\label{remarkcurvaturebleow}We use $n\kappa$ to denote the lower bound  of $\tr \mathfrak{A}^\mathbf{n}$. More precisely,
\begin{itemize}
\item \eqref{0diem} means that $\Sigma$ is a point and hence, the Weingarten map always vanishes, when $\kappa$ can be chosen as any real number;

\item if $\kappa=0$ in \eqref{m-1diem}, then $\Sigma$ is called {\it weakly mean convex} (cf. \cite{LLL,CLZ});

\item  \eqref{other-dime} implies $\kappa\leq 0$, because $n\kappa\leq\tr \mathfrak{A}^{-\mathbf{n}}=-\tr \mathfrak{A}^{\mathbf{n}}\leq -n\kappa$ for any $\mathbf{n}\in \mathcal {V}S\Sigma$.
Moreover, if $\kappa=0$, the above inequality yields
\[
0=\tr \mathfrak{A}^{\mathbf{n}}=n\langle H,\mathbf{n}\rangle,\quad \forall\,\mathbf{n}\in \mathcal {V}S\Sigma,
\]
which implies that $H=0$, i.e., $\Sigma$ is a minimal submanifold.

\end{itemize}
 As in the previous section, although we do not emphasize the sign of $\kappa$, we assume that $\kappa\leq 0$ whenever $1\leq n\leq m-2$.
\end{remark}

\subsection{
Hardy inequalities involving double-curvature weights
} \label{section-6-1}

\begin{theorem}\label{distance-weight-hardy-general} Let $(M,g)$ be an $m$-dimensional complete Riemannian manifold, $i:\Sigma \hookrightarrow M$ be an $n$-dimensional closed submanifold  and $\Omega$ be a non-empty domain in $M$ satisfying \eqref{natural-domain}.  Denote $r(x):=d(\Sigma,x)$  the distance function from $\Sigma$.
Suppose that $(M,\Sigma,\Omega)$ satisfies the lower curvature bound condition $(\lambda,\kappa)$.
Given any pair $(\Lambda,K)\leq (\lambda,\min\{\kappa,0\})$ with $\Lambda<\lambda/4$ if $\lambda>0$,
the following statements hold$:$
\begin{enumerate}[{\rm (i)}]
\item\label{basicriccboundedbelow} for $p>1$, $\beta<-(m-n)$, there holds
\begin{equation}\label{Hardy-Ineq-curvature-condition51}
\int_{\Omega_\Sigma}|\na u|^p\frac{\mathbf{c}_\Lambda(r)^{1-p}\mathbf{s}_\Lambda(r)^{p+\beta}}{(\mathbf{c}_\Lambda(r)-K \mathbf{s}_\Lambda(r))^n}\dvol_g\geq \Big|{\beta+m-n\over p}\Big|^p\int_{\Omega_\Sigma} |u|^p\frac{\mathbf{c}_\Lambda(r)\mathbf{s}_\Lambda(r)^\beta}{(\mathbf{c}_\Lambda(r)-K \mathbf{s}_\Lambda(r))^n} \dvol_g,
\end{equation}
 for  any $u\in C^\infty_0(\Omega_\Sigma)$.
In particular, if $\Sigma$ is compact and either $\Omega=M$ or $\partial\Omega=\Sigma$, then \eqref{Hardy-Ineq-curvature-condition51} is sharp and there are no extremal functions in $C^\infty_0(\Omega_\Sigma);$

\smallskip

\item\label{basicriccboundedbelowe2} for $p>1$, $\beta<-(m-n)$ with $p+\beta>-(m-n)$, inequality \eqref{Hardy-Ineq-curvature-condition51} holds for every $u\in C^\infty_0(\Omega,\Sigma)$.
%
Furthermore, if $\Sigma$ is compact and either $\Omega=M$ or $\partial\Omega=\Sigma$, then \eqref{Hardy-Ineq-curvature-condition51} is sharp over $C^\infty_0(\Omega,\Sigma)$ and there are no extremal functions in $C^\infty_0(\Omega,\Sigma)$.

\end{enumerate}
\end{theorem}

\begin{proof}\eqref{basicriccboundedbelow} Let us define
 \begin{equation}\label{newphipsi}
\phi(t):=(\mathbf{c}_\Lambda(t)-K\mathbf{s}_\Lambda(t))^{-\frac{n}{p-1}}\mathbf{s}_\Lambda(t)^{-\frac{m-n-1}{p-1}}, \quad \psi(t):=\mathbf{s}_\Lambda(t)^{\beta+m-n}, \quad t\in (0,\min\{\mathfrak{r}_\Lambda,\mathfrak{t}_{\Lambda,K} \}).
\end{equation}
In addition,  for any $\lambda,\kappa\in \mathbb{R}$ and
   $\mathbf{n}\in \mathcal {V}S\Sigma$, it follows from \eqref{Lap-deta-controll}, \eqref{tlkdefin}, \eqref{relationcandr} and \eqref{footcomprison}      that
\begin{equation}\label{cnvcontroll}
 c_\mathcal {V}(\mathbf{n})\leq  \mathfrak{t}_{\lambda,\kappa}\leq 2\mathfrak{r}_{\lambda}
 \begin{cases}
<\mathfrak{r}_{\Lambda}= \min\{\mathfrak{r}_\Lambda,\mathfrak{t}_{\Lambda,K}\},& \text{ if } \lambda>0;\\
\\
=\mathfrak{r}_{\Lambda}= \min\{\mathfrak{r}_\Lambda,\mathfrak{t}_{\Lambda,K}\}=+\infty,& \text{ if } \lambda\leq 0.
\end{cases}
\end{equation}
 Therefore, Assumption \ref{Basciasumptionfunction2}/\eqref{smoothofphi2}\eqref{smoothofpsi2}\eqref{negligbileset2} directly follow.
 Moreover,  Theorem \ref{Laplacian-comparison-Sigma} and Proposition \ref{gestiamtegra} provide
\begin{align*}
\Delta r\leq G_{\lambda,\kappa}(r)\leq G_{\Lambda,K}(r)=-(p-1)\frac{\phi'(r)}{\phi(r)},\quad \forall\,x\in  \mathcal {V}\mathscr{D}\Sigma\backslash \Sigma,
\end{align*}
which implies together with \eqref{plaplacephi} Assumption \ref{Basciasumptionfunction2}/\eqref{nansigcontroll2}.
Since $\phi(r), \psi(r)$ are locally bounded and $\psi'(r)$ is locally Lipschitz  in $M\backslash\Sigma$, it is a direct consequence that
\[
\phi(r)^{p-1}, \phi(r)^{p-1}\psi(r),\phi(r)^{p-1}|\psi'(r)|\in L^1_{\loc}(\Omega_\Sigma).
\]
In particular, $|\psi'(r)|$ is always nonzero and locally bounded in $M\backslash\Sigma$. Thus, a direct calculation yields
\begin{align}
\frac{\phi(r)^{p-1}|\psi'(r)|}{|\nabla\log(\psi(r))|^p}= \frac{\phi(r)^{p-1}\psi(r)^p}{|\psi'(r)|^{p-1}}\in L^1_{\loc}(M\backslash\Sigma),\label{lqinteg}
\end{align}
which shows the validity of Assumption \ref{Basciasumptionfunction}/\eqref{integrable}. Thus, Assumption \ref{Basciasumptionfunction} holds as it is weaker than Assumption \ref{Basciasumptionfunction2}. Accordingly, \eqref{Hardy-Ineq-curvature-condition51} follows  directly by Theorem \ref{Hardy-thm-general}.

According to Corollary \ref{geomertricmeaningofsharpness2}, in order to obtain the sharpness, it suffices to prove  \eqref{Ambrosio-condition-2}. First,
Proposition \ref{gestiamtegra} implies for any $t\in (0,\mathfrak{t}_{\lambda,\kappa})$ and small $\varepsilon>0$ that
\begin{align*}
&\left.\log\left(\left[ \mathbf{c}_\lambda(s)-\kappa\mathbf{s}_\lambda(s) \right]^n \mathbf{s}^{m-n-1}_\lambda(s)\right)\right|_{s=\varepsilon}^{s=t}=\int^{t}_{\e}G_{\lambda,\kappa}(s){\dd}s\\
\leq& \int^{t}_{\e}G_{\Lambda,K}(s){\dd}s=\left.\log\left(\left[ \mathbf{c}_\Lambda(s)-K\mathbf{s}_\Lambda(s) \right]^n \mathbf{s}^{m-n-1}_\Lambda(s)\right)\right|_{s=\varepsilon}^{s=t},
\end{align*}
which furnishes
\begin{equation}\label{cscontrola11}
\left[ \mathbf{c}_\lambda(t)-\kappa\mathbf{s}_\lambda(t) \right]^n \mathbf{s}^{m-n-1}_\lambda(t)\leq \left[ \mathbf{c}_\Lambda(t)-K\mathbf{s}_\Lambda(t) \right]^n \mathbf{s}^{m-n-1}_\Lambda(t).
\end{equation}
Due to \eqref{cnvcontroll} and the compactness of $\Sigma$,   we have
\[
0<c_\mathcal {V}(\Sigma)\leq \sup_{\mathbf{n}\in \mathcal {V}S\Sigma}c_\mathcal {V}(\mathbf{n})\leq \mathfrak{t}_{\lambda,\kappa}\leq 2\mathfrak{r}_{\lambda} \leq \mathfrak{r}_{\Lambda},
\]
where we used the assumption that $\Lambda<\lambda/4$ whenever $\lambda>0$. Therefore, for any $s_0\in (0,c_\mathcal {V}(\Sigma))$,
 by \eqref{Lap-deta-controll} and  \eqref{cscontrola11} we obtain
\begin{align*}
&\int_{{\T_{s_0}}\cap\Omega}\psi(r)^{-\e}\phi(r)^{p-1}|\nabla(\log\psi(r))| \dvol_g\\
=&\int_{\Sigma}{\dvol}_{i^*g}(x)\int_{\mathcal {V}S_{x}\Sigma}d\nu_x(\mathbf{n})\int_0^{s_0} |\alpha| \frac{\mathbf{c}_\Lambda(t)\mathbf{s}_{\Lambda}(t)^{-(m-n)-\alpha\varepsilon} }{(\mathbf{c}_\Lambda(t)-K\mathbf{s}_\Lambda(t))^n}  \det \mathcal {A}(t,\mathbf{n}){\dd}t\\
\leq & |\alpha| \int_{\Sigma}{\dvol}_{i^*g}(x)\int_{\mathcal {V}S_{x}\Sigma}d\nu_x(\mathbf{n})\int^{s_0}_0 \mathbf{s}_\Lambda(t)^{-\alpha \varepsilon -1} \mathbf{c}_\Lambda(t){\dd}t\\
=& \frac{1}{\varepsilon} \vol(\mathbb{S}^{m-n-1})\vol_{i^*g}(\Sigma)\mathbf{s}_\Lambda(s_0)^{-\alpha\varepsilon}<+\infty,
\end{align*}
where $\alpha:=\beta+m-n<0$. A similar computation yields
\begin{align*}
\int_{{\T_{s_0}}\backslash\Omega}\psi(r)^{2+\e}\phi(r)^{p-1}|\nabla(\log\psi(r))| \dvol_g\leq  |\alpha| \int_{\Sigma}{\dvol}_{i^*g}(x)\int_{\mathcal {V}S_{x}\Sigma}d\nu_x(\mathbf{n})\int_{s_0}^{c_\mathcal {V}(\mathbf{n})} \mathbf{s}_\Lambda(t)^{\alpha(2+\varepsilon) -1} \mathbf{c}_\Lambda(t){\dd}t<+\infty.
\end{align*}
Therefore, \eqref{Ambrosio-condition-2} holds, which implies the sharpness of \eqref{Hardy-Ineq-curvature-condition51}. In addition,   Theorem \ref{achieve-thm} implies the non-existence of extremal functions for \eqref{Hardy-Ineq-curvature-condition51} in $C^\infty_0(\Omega_\Sigma)$.

\medskip

\eqref{basicriccboundedbelowe2} We first show \eqref{Hardy-Ineq-curvature-condition51} for $C^\infty_0(\Omega,\Sigma)$ and its sharpness.
According  to Theorem \ref{generalhardclosed} and Remark \ref{sharpnonextronclosed}, it suffices to show that Assumption \ref{Basciasumptionfunction2} is satisfied. Moreover, comparing Assumption \ref{Basciasumptionfunction} and Assumption \ref{Basciasumptionfunction2}, we only need to prove that
\begin{equation}\label{qglobalinteg}
\frac{\phi(r)^{p-1}|\psi'(r)|}{ |\nabla(\log \psi(r))|^{p}}\in L^{1}_{\loc}(M), \quad\text{ $\frac{\partial}{\partial r}\log\left(\frac{\phi(r)^{p-1}|\psi'(r)|}{ |\nabla(\log \psi(r))|^{p}}\right)$ is locally bounded in $M\backslash\Sigma$}.
\end{equation}
 The latter property from \eqref{qglobalinteg} easily follows, since $\mathbf{c}_\Lambda(r)$, $\mathbf{s}_\Lambda(r)$ and $(\mathbf{c}_\Lambda(r)-K \mathbf{s}_\Lambda(r))$ are nonzero and locally bounded in $M\backslash\Sigma$.
For the former property of \eqref{qglobalinteg},
due to \eqref{lqinteg}, we only need to check its local integrability around $\Sigma$. Therefore, without loss of generality, we may assume that $\Sigma$ is compact; otherwise, we choose a compact subset of $\Sigma$.
Now, the first limit of \eqref{smallsestimatedatA} together with the compactness of $\mathcal {V}S\Sigma$ provides a small $\epsilon\in (0,c_\mathcal {V}(\Sigma))$ such that
\begin{equation}\label{standarddetestimate}
\det  \mathcal {A}(t,\mathbf{n})\leq 2t^{m-n-1}, \quad \forall\,(t,\mathbf{n})\in  (0,\epsilon)\times \mathcal {V}S\Sigma.
\end{equation}
Since $\mathbf{s}_\Lambda(t)\sim t$ as $t\to 0^+$, by choosing eventually a smaller $\epsilon$,  we may assume
\begin{equation}\label{weightedestimate}
\frac{\phi(t)^{p-1}|\psi'(t)|}{ |(\log \psi(t))'|^{p}}= \frac{\mathbf{c}_\Lambda(t)^{1-p}\mathbf{s}_\Lambda(t)^{p+\beta}}{(\mathbf{c}_\Lambda(t)-K \mathbf{s}_\Lambda(t))^n}\leq 2 t^{p+\beta}, \quad \forall\,t\in (0,\epsilon).
\end{equation}
Due to $p+\beta>-(m - n)$, the volume expression \eqref{dv-g-Def} combined with \eqref{standarddetestimate} and \eqref{weightedestimate} yields
\begin{align*}
0\leq \int_{\T_\epsilon}   \frac{\phi(r)^{p-1}|\psi'(r)|}{ |\nabla(\log \psi(r))|^{p}}  \dvol_g\leq 4\vol(\mathbb{S}^{m-n-1})\vol_{i^*g}(\Sigma)\int^{\epsilon}_0  t^{p+\beta+m-n-1}{\dd}t<+\infty,
\end{align*}
which together with \eqref{lqinteg} implies $\frac{\phi(r)^{p-1}|\psi'(r)|}{ |\nabla(\log \psi(r))|^{p}}\in L^{1}_{\loc}(M)$.

 It remains to show the nonexistence of   extremal functions. Owing to the compactness of $\Sigma$ and \eqref{smallsestimatedatA}, choose a small $\delta\in (0, c_\mathcal {V}(\Sigma))$ such that
 \[
 \frac{\det \mathcal {A}(t,\mathbf{n})}{(\mathbf{c}_\Lambda(t)-K\mathbf{s}_\Lambda(t))^n}\geq \frac12 \mathbf{s}_{\Lambda}(t)^{m-n-1}, \quad \text{ for }(t,\mathbf{n})\in[0,\delta]\times \mathcal {V}S\Sigma.
 \]
Thus,
for any $\varepsilon>0$ we have
\begin{align*}
\int^\delta_\varepsilon \phi(t)^{p-1}|(\log \psi(t))'| \det \mathcal {A}(t,\mathbf{n}){\dd}t\geq \frac{|\alpha|}2\int^\delta_\varepsilon \mathbf{s}_{\Lambda}(t)^{-1} \mathbf{c}_{\Lambda}(t){\dd}t\rightarrow +\infty, \quad \text{ as }\varepsilon\rightarrow 0^+,
\end{align*}
 which implies Condition \eqref{limteexideoin3} in Remark \ref{sharpnonextronclosed} and   the nonexistence of   extremal functions.
\end{proof}

\begin{remark}
It should be noticed that \eqref{cnvcontroll} implies  $M\subset \overline{\T_{\mathfrak{t}_{\lambda,\kappa}}}$. In fact, for every $x\in \mathcal {V}\mathscr{C}\Sigma \backslash\Sigma$, there exists a unique minimal normal geodesic from $\Sigma$ to $x$ with $\dot{\gamma}(0)=\mathbf{n}\in \mathcal {V}S\Sigma$, see Lemma \ref{prop-1-3}. Clearly, the length of $\gamma$ is not  larger than  $c_\mathcal {V}(\mathbf{n}) $, which combined with \eqref{cnvcontroll} implies  $M\subset \overline{\T_{\mathfrak{t}_{\lambda,\kappa}}}$.  Obviously, this result can be considered as a submanifold version of the Bonnet-Myers theorem.
\end{remark}

\begin{remark}\label{lambdapostivevcase}
In the assumption of Theorem \ref{distance-weight-hardy-general} we require that $\Lambda<\lambda/4$ if $\lambda>0$; however,
 this condition can be relaxed to ``$\Lambda\leq \lambda/4$ if $\lambda>0$" whenever $n>0$. In fact, such an assumption is needed to establish \eqref{cnvcontroll}, i.e.,
\begin{equation}\label{estinjectiradi}
c_\mathcal {V}(\mathbf{n})< \min\{\mathfrak{r}_\Lambda,\mathfrak{t}_{\Lambda,K}\},
\end{equation}
 which guarantees the fact that the functions $\phi,\psi$ defined in \eqref{newphipsi} satisfy Assumption \ref{Basciasumptionfunction2}.
When $n>0$ and $\lambda>0$, in view of \eqref{estimatecklambda} and \eqref{tlkdefin}, we have
\[
c_\mathcal {V}(\mathbf{n})\leq  \mathfrak{t}_{\lambda,\kappa}=\mathfrak{c}_{\lambda,\kappa}< 2\mathfrak{r}_{\lambda}\leq \mathfrak{r}_{\Lambda} =\min\{\mathfrak{r}_\Lambda,\mathfrak{t}_{\Lambda,K}\},
\]
i.e., \eqref{estinjectiradi} remains valid, and the same argument as above can be followed.
Furthermore, if $n>0$, $\kappa=0$ and $\lambda>0$,  a similar argument applies whenever $\Lambda<\lambda$.
\end{remark}

\begin{remark} Although one could expect that \eqref{Hardy-Ineq-curvature-condition51} can be extended to $C^\infty_0(\Omega)$ (i.e., Theorem \ref{Hardy-thm-2}), it is easy to see that Assumption \ref{Basciasumptionfunction3}/\eqref{derivativeintegr} fails. Moreover, if $\Omega=M$ is a compact manifold, constant functions  belong to $C^\infty_0(\Omega)$, for which \eqref{Hardy-Ineq-curvature-condition51} clearly fails.
\end{remark}

An interesting application of Theorem \ref{distance-weight-hardy-general} to spheres is as follows, providing three different kinds of Hardy inequalities.

\begin{example}\label{example-sinh-spheres}Let $\MM_\lambda=\mathbb{S}^m$ and let $\Sigma$ be the great sphere $\mathbb{S}^{m-1}\subset\mathbb{S}^m$ (i.e., $n=m-1$).
Let $\Omega$ be an open hemisphere with $\partial \Omega=\Sigma$. Obviously, $(M,\Sigma,\Omega)$ satisfies the (non-optimal) lower curvature bound condition $(\lambda,\kappa)=(1,0)$.

Given $p>1$ and $\beta<-1$, by choosing different pairs of $(\Lambda,0)\leq (1,0)$ (i.e., $\Lambda=1/4,0,-1$), Theorem \ref{distance-weight-hardy-general} together with Remark \ref{lambdapostivevcase} implies for every  $u\in C^\infty_0(\Omega)$ that
\begin{align*}
\int_{\Omega}|\nabla u|^p r^{\beta+p}\dvol_g&\geq \left|  \frac{\beta+1}{p} \right|^p \int_{\Omega}|u|^p r^\beta \dvol_g,\\
\int_{\Omega}|\na u|^p\cos\left(\frac{r}2\right)^{2-m-p}\sin\left(\frac{r}2\right)^{p+\beta}\dvol_g&\geq \Big|{\beta+1\over {2p}}\Big|^p\int_{\Omega} |u|^p\cos\left(\frac{r}2\right)^{2-m}\sin\left(\frac{r}2\right)^\beta \dvol_g,\\
\int_{\Omega}|\na u|^p\cosh(r)^{2-m-p}\sinh(r)^{p+\beta}\dvol_g&\geq \Big|{\beta+1\over p}\Big|^p\int_{\Omega} |u|^p\cosh(r)^{2-m}\sinh(r)^\beta \dvol_g.
\end{align*}
In particular, all of them are sharp and no extremal functions exist in $C^\infty_0(\Omega)$.
\end{example}

\begin{remark}\rm
	A similar observation can be made as in Remark \ref{barbatis-remark}. Indeed, the domain $\Omega$ in Theorem \ref{distance-weight-hardy-general} is arbitrary  but
	the sharpness  depends on $c_\mathcal {V}(\Sigma)>0$. In addition, if
	\begin{equation*}\label{r finite-0}
		\sup_{x\in \Omega} r(x)<+\infty,
	\end{equation*}
	holds instead of $c_\mathcal {V}(\Sigma)>0$, one can prove formally the same improved sharp Hardy inequality as \eqref{newsharpss344} on the  spaces $C^\infty_0(\Omega_\Sigma)$ and $C^\infty_0(\Omega,\Sigma)$  instead of $C^\infty_0(\Omega)$, which is another natural extension of the result of Barbatis et al. \cite{BFT} (see also Chen et al. \cite{CLZ}).
\end{remark}

\subsection{Hardy inequalities involving logarithmic double-curvature weights
} \label{section-6-2}

The pair $(\phi,\psi)$ is defined as follows.
Given $\Lambda, K,s_1,s_2\in \mathbb{R}$ and $D\in (0,\min\{\mathfrak{r}_\Lambda,\mathfrak{t}_{\Lambda,K}\})$, $L\in(0,D]$, let
\begin{equation}
\phi(t):=(\mathbf{c}_\Lambda(t)-K\mathbf{s}_{\Lambda}(t))^{-{n\over p-1}}\mathbf{s}_\Lambda(t)^{-{m-n-1\over p-1}},\quad  \psi(t):=\psi(t,s_1,s_2):=\int^L_t \left[\log\left({\mathbf{s}_\Lambda(D)\over \mathbf{s}_\Lambda(\tau)}\right)\right]^{s_1} \mathbf{s}_\Lambda(\tau)^{s_2}\mathbf{c}_\Lambda(\tau){\dd}\tau.\label{def-psi-log-general}
\end{equation}

%
%
%
%
%
%
%
%
%
%
%
%
%
%
%

\begin{lemma}\label{computation-lem}
	The function $\psi$ in \eqref{def-psi-log-general} is well-defined and strictly decreasing in $(0,L]$ either when
	 $L<D$ and  $s_1,s_2\in \mathbb{R}$, or when $L=D$ and   $s_1>-1$ and $s_2\in \mathbb{R}$.
%
%
%
%
%
%
%
%
%
%
%
%
%
%
%
	Moreover, given $L\in (0,D]$, if
$s>0$, $s_1>-1$ and $s_2\leq -1$ satisfy
\[
s_1+1\geq s+(s_2+1)\log{\mathbf{s}_\Lambda(D)\over \mathbf{s}_\Lambda(L)},
\]
then
\begin{equation}
\psi(t)\leq s^{-1}\left[\log\left({\mathbf{s}_\Lambda(D)\over \mathbf{s}_\Lambda(t)}\right)\right]^{s_1+1} \mathbf{s}_\Lambda(t)^{s_2+1}=s^{-1}\left[\log\left({\mathbf{s}_\Lambda(D)\over \mathbf{s}_\Lambda(t)}\right)\right]{\mathbf{s}_\Lambda(t)\over \mathbf{c}_\Lambda(t)}|\psi'(t)|,\ \ t\in (0,L], \label{inequality-lemm-general}
\end{equation}
with equality if $s_2=-1$, $s=s_1+1$ and $L=D$.
\end{lemma}
\begin{proof}
The well-definiteness of $\psi$ is based on Proposition \ref{basicloganalysis}; moreover, it is a decreasing positive function, as
	\begin{equation}\label{derivapsi}
		\psi'(t)=-\left[\log\left({\mathbf{s}_\Lambda(D)\over \mathbf{s}_\Lambda(t)}\right)\right]^{s_1} \mathbf{s}_\Lambda(t)^{s_2}\mathbf{c}_\Lambda(t)<0,\quad t\in (0,L).
	\end{equation}
	Let
 \[
 F(t):=\psi(t)-s^{-1}\left[\log\left({\mathbf{s}_\Lambda(D)\over \mathbf{s}_\Lambda(t)}\right)\right]^{s_1+1} \mathbf{s}_\Lambda(t)^{s_2+1}.
 \]
  Since $\mathbf{s}_\Lambda$ is  increasing  in $[0,D]$, a direct calculation together with \eqref{derivapsi} yields
\[
F'(t)=s^{-1}\left[\log\left({\mathbf{s}_\Lambda(D)\over \mathbf{s}_\Lambda(t)}\right)\right]^{s_1} \mathbf{s}_\Lambda(t)^{s_2}\mathbf{c}_\Lambda(t)\left[s_1+1-s-(s_2+1)\log\left({\mathbf{s}_\Lambda(D)\over \mathbf{s}_\Lambda(t)}\right)\right]\geq0.
\]
Clearly, $\psi(L)=0$ and $F(L)\leq0$. Thus,  $F(t)\leq F(L)\leq0$ for any $t\in (0,L]$, which is exactly (\ref{inequality-lemm-general}). In particular,  if $s_2=-1$, $s=s_1+1$ and $L=D$, then $F'(t)=0$ and $F(t)=F(L)=F(D)=0$.
\end{proof}

In addition, elementary computations imply the fact that $\Lambda\mapsto \frac{\mathbf{s}_\Lambda(D)}{\mathbf{s}_\Lambda(L)}$ is decreasing in $\Lambda\in (-\infty,\lambda)$ whenever $L\in (0,D]$; therefore, we have the following monotonicity property which plays a crucial role in the sequel.
\begin{lemma}\label{monotonocity-a-function} Given $\lambda\in \mathbb{R}$, $D\in (0,\mathfrak{r}_\lambda)$ and $L\in(0,D]$,  we have
$$
\frac{\mathbf{s}_\lambda(D)}{\mathbf{s}_\lambda(L)}\leq \frac{\mathbf{s}_\Lambda(D)}{\mathbf{s}_\Lambda(L)},\quad \forall\,\Lambda\leq \lambda.
$$
\end{lemma}


We now derive the following Hardy inequality with logarithmic weights.

\begin{theorem}\label{log-weight-thm-general} Let $(M,g)$ be an $m$-dimensional complete Riemannian manifold, $i:\Sigma \hookrightarrow M$ be an $n$-dimensional  closed submanifold and $\Omega$ be a non-empty domain in $M$ satisfying \eqref{natural-domain} and $\sup_{x\in \Omega}r(x)<+\infty$.     Suppose  $(M,\Sigma,\Omega)$ satisfies the lower curvature bound condition $(\lambda,\kappa)$. If $D\in(0,\min\{\mathfrak{r}_\lambda,\mathfrak{t}_{\lambda,\kappa}\})$ and  $s,\alpha,\beta\in \mathbb{R}$ are such that
\[
\sup_{x\in \Omega}r(x)\leq D,\quad   s>0,  \quad \beta>-1,\quad \alpha-p+m-n\leq 0, \quad \beta+1\geq s+(\alpha-p+m-n)\log\left({\mathbf{s}_\lambda(D)\over \sup_{x\in \Omega}\mathbf{s}_\lambda(r(x))}\right),
\]
 then for any $(\Lambda,K)\leq (\lambda,\min\{\kappa,0\})$ with $\Lambda<\lambda/4$ if $\lambda>0$, we have
\begin{align} & \int_{\Omega_\Sigma}|\na u|^p\left[\log\left({\mathbf{s}_\Lambda(D)\over \mathbf{s}_\Lambda(r)}\right)\right]^{p+\beta}\frac{\mathbf{c}_\Lambda(r)^{1-p}\mathbf{s}_\Lambda(r)^{\alpha}}{(\mathbf{c}_\Lambda(r)-K\mathbf{s}_\Lambda(r))^n} \dvol_g\nonumber\\
\geq& \left({s\over p}\right)^p\int_{\Omega_\Sigma}  |u|^p \left[\log\left({\mathbf{s}_\Lambda(D)\over \mathbf{s}_\Lambda(r)}\right)\right]^\beta \frac{\mathbf{c}_\Lambda(r)\mathbf{s}_\Lambda(r)^{\alpha-p}}{(\mathbf{c}_\Lambda(r)-K\mathbf{s}_\Lambda(r))^n} \dvol_g,\quad \forall\, u\in C^\infty_0(\Omega_\Sigma).\label{hardy-log-weight-general}
\end{align}
Moreover, if either $\Omega=M$ or $\Sigma=\partial\Omega$, and
\begin{equation}\label{compsharpnessloginequ}
\text{$\Sigma$ is compact},\quad \alpha=p-(m-n),\quad s=\beta+1,\quad \sup_{x\in \Omega}r(x)=D,
\end{equation}
then \eqref{hardy-log-weight-general} is sharp.
\end{theorem}
\begin{proof} In view of the assumption,  choose
\begin{equation}\label{constlogfixed}
L:=\sup_{x\in \Omega}r(x)\leq D<\min\{\mathfrak{r}_\lambda,\mathfrak{t}_{\lambda,\kappa}\},\quad s_1:=\beta>-1, \quad s_2:=\alpha-p+(m-n)-1\leq -1.
\end{equation}
 Due to Lemma \ref{computation-lem}, the pair of functions $(\phi,\psi)$ constructed by \eqref{def-psi-log-general} is well-defined on $(0,L]$. Note that \eqref{cnvcontroll} remains valid and hence, Assumption \ref{Basciasumptionfunction}/\eqref{smoothofphi}\eqref{smoothofpsi}\eqref{negligbileset} hold. Moreover, according to Theorem \ref{Laplacian-comparison-Sigma} and Proposition \ref{gestiamtegra}, we have
\begin{align*}
\Delta r\leq G_{\lambda,\kappa}(r)\leq G_{\Lambda,K}(r)=-(p-1)\frac{\phi'(r)}{\phi(r)},\quad \forall\,x\in  \mathcal {V}\mathscr{D}\Sigma,
\end{align*}
which combined with \eqref{plaplacephi} and \eqref{derivapsi} yields  Assumption \ref{Basciasumptionfunction}/\eqref{nansigcontroll}.
Now we are going to show that Assumption \ref{Basciasumptionfunction}/\eqref{integrable} holds. Note that $\phi(r),\psi(r)$ are locally bounded and the latter is  locally Lipschitz in $\Omega_\Sigma$. Thus, one has
$
\phi(r)^{p-1}, \phi(r)^{p-1}\psi(r),\phi(r)^{p-1}|\psi'(r)|\in L^1_{\lo}(\Omega_\Sigma).
$
Moreover, by the assumption and Lemma \ref{monotonocity-a-function} we obtain
\begin{align*}
s_1+1&=\beta+1\geq s+(\alpha-p+m-n)\log\left({\mathbf{s}_\lambda(D)\over \sup_{x\in \Omega}\mathbf{s}_\lambda(r(x))}\right)\\
&=s+(s_2+1)\log\left({\mathbf{s}_\lambda(D)\over  \mathbf{s}_\lambda(L)}\right)\geq s+(s_2+1)\log\left({\mathbf{s}_\Lambda(D)\over  \mathbf{s}_\Lambda(L)}\right),
\end{align*}
which combined with Lemma \ref{computation-lem} yields \eqref{inequality-lemm-general}. Thus,
 we have
\begin{align}
\frac{\phi(r)^{p-1} |\psi'(r)|}{|\nabla(\log \psi(r))|^p}=\,&(\mathbf{c}_\Lambda(r)-K\mathbf{s}_\Lambda(r))^{-n}\mathbf{s}_\Lambda(r)^{-(m-n-1)}\frac{\psi(r)^p}{|\psi'(r)|^{p-1}}\notag\\
\leq &
s^{-p}\left[\log\left({\mathbf{s}_\Lambda(D)\over \mathbf{s}_\Lambda(r)}\right)\right]^{p+\beta}\frac{\mathbf{c}_\Lambda(r)^{1-p}\mathbf{s}_\Lambda(r)^{\alpha}}{(\mathbf{c}_\Lambda(r)-K\mathbf{s}_\Lambda(r))^n}, \label{pp-0p-00}
\end{align}
where the latter is obviously locally bounded in $\Omega_\Sigma$. Thus, Assumption \ref{Basciasumptionfunction}/\eqref{integrable} holds.
From above, Assumption \ref{Basciasumptionfunction} is satisfied and hence,  \eqref{hardy-log-weight-general} follows by \eqref{Hardy-Inequ-general} and \eqref{pp-0p-00}.

Assume either  $\Omega=M$ or $\Sigma=\partial\Omega$; we show the sharpness of \eqref{hardy-log-weight-general} under \eqref{compsharpnessloginequ}.
Due to \eqref{constlogfixed}, we have
\[
s_1=\beta>-1,\quad s_2=-1,\quad s=\beta+1>0,\quad L=\sup_{x\in \Omega}r(x)=D,
\]
which combined with
Lemma \ref{computation-lem} yields
\begin{equation}
\psi(t)=(\beta+1)^{-1}\left[\log\left({\mathbf{s}_\Lambda(D)\over \mathbf{s}_\Lambda(t)}\right)\right]^{\beta+1},\quad \forall\, t\in (0,D],\label{def-psi-log-general-special}
\end{equation}
which furnishes $\psi(0^+)=+\infty$ and $\psi'(t)<0$. Since now $\Omega$ is a bounded domain, in view of Corollary \ref{geomertricmeaningofsharpness2}, it remains to show \eqref{Ambrosio-condition-2}. Indeed,
the compactness of $\Sigma$ implies $c_\mathcal {V}(\Sigma)>0$. Choose $s_0\in (0,\min\{c_\mathcal {V}(\Sigma),D\})$.
Thus, \eqref{def-psi-log-general-special} together with  \eqref{Lap-deta-controll} and \eqref{cscontrola11} (i.e., the integral version of Proposition \ref{gestiamtegra}) yields
\begin{align} &\int_{\T_{s_0}\cap\,\Omega} \psi(r)^{-(1+\e)}\phi(r)^{p-1}|\psi'(r)|\dvol_g\nonumber\\
\leq&(\beta+1)^{1+\e}\int_{r<s_0}\left[\log\left({\mathbf{s}_\Lambda(D)\over \mathbf{s}_\Lambda(r)}\right)\right]^{-(1+\e)(\beta+1)+\beta}\frac{\mathbf{c}_\Lambda(r)\mathbf{s}_\Lambda(r)^{-(m-n)}}{(\mathbf{c}_\Lambda(r)-K\mathbf{s}_\Lambda(r))^n}\dvol_g\nonumber\\
\leq& (\beta+1)^{1+\e}\vol(\mathbb{S}^{m-n-1})\vol_{i^*g}(\Sigma) \int_0^{s_0}\left[\log\left({\mathbf{s}_\Lambda(D)\over \mathbf{s}_\Lambda(r)}\right)\right]^{-(1+\e)(\beta+1)+\beta}\mathbf{s}_\Lambda(r)^{-1}\mathbf{c}_\Lambda(r){\dd}r<+\infty,\nonumber
\end{align}
where the last inequality follows by $-(1+\e)(\beta+1)+\beta<-1$ and  Proposition \ref{basicloganalysis}. In the same way, by  $(1+\e)(\beta+1)+\beta>-1$ and Lemma \ref{computation-lem}  we have
\begin{align}
&\int_{\Omega\backslash{\T_{s_0}}} \psi(r)^{1+\e}\phi(r)^{p-1}|\psi'(r)|\dvol_g\leq(\beta+1)^{-(1+\e)}\int_{r>s_0} \left[\log\left({\mathbf{s}_\Lambda(D)\over \mathbf{s}_\Lambda(r)}\right)\right]^{(1+\e)(\beta+1)+\beta}\frac{\mathbf{c}_\Lambda(r)\mathbf{s}_\Lambda(r)^{-(m-n)}}{(\mathbf{c}_\Lambda(r)-K\mathbf{s}_\Lambda(r))^n}\dvol_g\nonumber\\
&\leq(\beta+1)^{-(1+\e)} \vol(\mathbb{S}^{m-n-1})\vol_{i^*g}(\Sigma)\int_{s_0}^{D} \left[\log\left({\mathbf{s}_\Lambda(D)\over \mathbf{s}_\Lambda(r)}\right)\right]^{(1+\e)(\beta+1)+\beta}\mathbf{s}_\Lambda(r)^{-1}\mathbf{c}_\Lambda(r){\dd}r<+\infty.\nonumber
\end{align}
 Thus,  \eqref{Ambrosio-condition-2}   and the sharpness follow.
\end{proof}

\begin{remark}
On account of Lemma \ref{Ambrosio-lem2} and Corollary \ref{finsharpcorollary}, by the same calculation one can show that \eqref{hardy-log-weight-general} is also sharp when $\Sigma$ is compact with $\Sigma\subset\Omega$, $\partial\Omega =\{r=D\}$, $\alpha=p-(m-n)$ and  $s=\beta+1$.
\end{remark}

\appendix

\section{Monotonicity of $G_{\lambda,\kappa}$}\label{monotG}

In this section, we provide the proof of Proposition \ref{gestiamtegra} that states some important monotonicity properties of the function $G_{\lambda,\kappa}$ defined by \eqref{Gfunctiondefined}, i.e.,
\begin{align}\label{Gdefexrpession}
G_{\lambda,\kappa}(t)=\left\{
\begin{array}{lll} \sqrt{\lambda}\left[-n{\sqrt{\lambda} +\kappa \cot(\sqrt{\lambda}t)\over \sqrt{\lambda}\cot(\sqrt{\lambda}t)-\kappa }+(m-n-1)\cot(\sqrt{\lambda}t)\right], && \text{ if }\lambda>0,\\
\\
 -{n\kappa\over 1-\kappa t}+{m-n-1\over t}, && \text{ if }\lambda=0,\\
\\
\sqrt{-\lambda}\left[-n{-\sqrt{-\lambda} +\kappa \coth(\sqrt{-\lambda}t)\over \sqrt{-\lambda}\coth(\sqrt{-\lambda}t)-\kappa }+(m-n-1)\coth(\sqrt{-\lambda}t)\right], && \text{ if }\lambda<0.
\end{array}
\right.
\end{align}
Although the proof is elementary, for the sake of completeness, we provide its proof.

\begin{proof}[Proof of Proposition \ref{gestiamtegra}] The monotonicity in $\kappa$   follows by a similar calculation as in  the proof of Theorem \ref{Laplacian-comparison-Sigma} (see Step 1). Hence, we focus on the monotonicity in $\lambda$.

\smallskip

\noindent{\it Case 1}: $\lambda>0$. For convenience, set $z:=\sqrt{\lambda}$. Thus, in view of \eqref{Gdefexrpession},
we study the monotonicity of
\begin{equation*} z
\mapsto z\left[-n{z+\kappa \cot(zt)\over z\cot(zt)-\kappa }+(m-n-1)\cot(zt)\right]=:\mathcal {G}_{\kappa}(z,t). \label{w-mapsto-Laplaican-0}
\end{equation*}
A direct calculation yields
\begin{align}
\frac{\partial \mathcal {G}_{\kappa}}{\partial z}(z,t)=&-n{z^2[\cot(zt)+zt\text{csc}^2(zt)]+\kappa^2[zt\text{csc}^2(zt)-\cot(zt)]-2\kappa z\over (z\cot(zt)-\kappa)^2 }\notag\\
&+(m-n-1)[\cot(zt)-zt\csc^2(zt)].\label{particG}
\end{align}
On the other hand, we have the identities
\begin{align}
A(x):=x\csc^2(x)+\cot(x)=\frac{x+\sin(x)\cos(x)}{\sin^2(x)}=\frac{x+\frac12\sin (2x)}{\sin^2(x)}>0, \quad \forall\,x>0.\notag\\
B(x):=x\csc^2(x)-\cot(x)=\frac{x-\sin(x)\cos(x)}{\sin^2(x)}=\frac{x-\frac12\sin (2x)}{\sin^2(x)}>0, \quad \forall\,x>0.\label{lasttemnumber}
\end{align}
Since  $\sin^2x\leq x^2$, we have $x^2\csc^2x-1\geq 0$ and hence,
\begin{align*}
z^2A(x)+\kappa^2B(x) &\geq 2z|\kappa|A(x)B(x)
=  2z|\kappa| \left( x^2\csc^4x-\cot^2x\right)\\
&=2z|\kappa|\left( \csc^2x (x^2\csc^2x-1)+1  \right)
\geq  2z|\kappa|\geq 2\kappa z,
\end{align*}
which together with  \eqref{particG} and \eqref{lasttemnumber} implies $\frac{\partial \mathcal {G}_{\kappa}}{\partial z}(z,t)<0$, i.e., $G_{\lambda,\kappa}(t)$ is decreasing in $\lambda$.
Furthermore, by using $\lim_{\lambda\rightarrow 0^+}\sqrt{\lambda}\cot(\sqrt{\lambda}t)=1/t$, we obtain
\begin{align*}
\lim_{\lambda\rightarrow 0^+}G_{\lambda,\kappa}(t)=-{n\kappa\over 1-\kappa t}+{m-n-1\over t}=G_{0,\kappa}(t).
\end{align*}

\noindent{\it Case 2}: $\lambda<0$.
By setting $w:=\sqrt{-\lambda}$ in \eqref{Gdefexrpession}, we need to study the monotonicity of
\begin{equation*}
w\mapsto w\left[-n{-w+\kappa \coth(wt)\over w\coth(wt)-\kappa }+(m-n-1)\coth(wt)\right]=:\mathcal {G}_{\kappa}(w,t).
\end{equation*}
Similarly, we have
\begin{align}
\frac{\partial \mathcal {G}_{\kappa}}{\partial w}(w,t)=&n{w^2[wt \text{csch}^2(wt)+\coth(wt)]+\kappa^2[\coth(wt)-wt \text{csch}^2(wt)]-2\kappa w\over (w\coth(wt)-\kappa)^2 }\notag\\
&+(m-n-1)[\coth(wt)-wt\text{csch}^2(wt)].\label{w-mapsto-Laplaican}
\end{align}
Let
\begin{align}
A(x):=\coth(x)+x\text{csch}^2(x)>0, \quad
B(x):=\coth(x)-x \text{csch}^2(x)>0, \quad \forall\,x>0.\label{Bshchcontroll}
\end{align}
Since  $1+\text{csch}^2x=\coth^2x$ and $\sinh^2x\geq x^2$ for any $x\geq0$, we have
\begin{align*}
w^2A(x)+\kappa^2B(x)&\geq   2w|\kappa|A(x)B(x)=  2w|\kappa|(\coth^2x-x^2\text{csch}^4x)\\
&=2w|\kappa|(\text{csch}^2x(1-x^2\text{csch}^2x)+1)\geq 2w|\kappa|\geq 2w \kappa.
\end{align*}
which together with \eqref{w-mapsto-Laplaican} and \eqref{Bshchcontroll} yields  $\frac{\partial \mathcal {G}_{\kappa}}{\partial w}(w,t)>0$, which concludes the proof.  
\end{proof}

%

\section{Weighted Sobolev spaces}\label{Soblevspace}
 The results in this section -- which are of interest of their own right --  are natural extensions of the ones from Chen et al. \cite[Appendix A]{CLZ}, where we establish
 some  properties of the weighted Sobolev spaces defined by Definition \ref{Sobolevspaces} under Assumption \ref{Basciasumptionfunction2}. Throughout the section, we assume the validity of Assumption \ref{Basciasumptionfunction2}.

A standard argument together with  Remark \ref{assp2reason} yields the following result.
\begin{lemma}\label{nullmeaure}
Let ${\dmu}:=\frac{\phi(r)^{p-1}|\psi'(r)|}{ |\nabla(\log \psi(r))|^{p}}{\dvol}_g$. Then ${\dmu}$ is a $\sigma$-finite measure on $M$. Moreover, $\mu(E)=0$ if and only if $\vol_g(E)=0$
for any Borel set $E\subset \Sigma\cup \Omega$. In particular, $\mu(\Sigma)=0$.
\end{lemma}

In the sequel, we prefer to use ${\dmu}$ rather than $\dvol_g$ to investigate weighted Sobolev spaces. Let $U\subset M$ be an open set.
 For any $s\in (1,\infty)$, we define
  $L^s({U},\phi,\psi)$ (resp., $L^s(T{U},\phi,\psi)$) as the completion of $C_0({U})$ (resp., $\Gamma_0(T{U})$, i.e., the space of  continuous tangent vector fields with compact support in ${U}$) under the norm
\begin{align*}
[u]_{s,\mu}:=\left( \int_{{{U}}}|u|^s{\dmu}\right)^{\frac1s}, \ \ \  \left(\text{ resp., }  [X]_{s,\mu}:=\left( \int_{{{U}}}|X|^s{\dmu}\right)^{\frac1s} \right).
\end{align*}

The standard theory of Lebesgue spaces yields the following result.
\begin{theorem}\label{LPreflexive}
Both $L^s({U},\phi,\psi)$ and $L^s(T{U},\phi,\psi)$ are reflexive. Hence, if a sequence $\{u_i\}$ $($resp., $\{X_i\}$$)$ is bounded in
$L^s({U},\phi,\psi)$ $($resp., $L^s(T{U},\phi,\psi)$$),$ then there exists a subsequence $\{u_{i_k}\}$ $($resp., $\{X_{i_k}\}$$)$ and $u\in L^s({U},\phi,\psi)$ $($resp., $X\in L^s(T{U},\phi,\psi)$$)$ such that $u_{i_k}\rightarrow u$
weakly in $L^s({U},\phi,\psi)$ $($resp., $X_{i_k}\rightarrow X$
weakly in $L^s(T{U},\phi,\psi)$$)$ as $k\rightarrow \infty$.
\end{theorem}

Let  $W^{1,p}(U, \phi,\psi)$  be  as in Definition \ref{Sobolevspaces}; thus,  for any $u\in C^\infty({U})$, we have $\|u\|^p_{p}=[u]^p_{p,\mu}+[\nabla u]^p_{p,\mu}$.
\begin{remark}\label{smoothdivege}It is easy to check that the weak gradient in the sense of $W^{1,p}({U},\phi,\psi)$ satisfies the following properties:
\begin{enumerate}[{\rm (a)}]

\item\label{nablaproperty} $\nabla(\lambda  u+\zeta  v)=\lambda\nabla  u+\zeta\nabla  v$ for any $\lambda,\zeta\in \mathbb{R}$ and $u,v\in W^{1,p}({U},\phi,\psi)$;

\item\label{weakdiverlemma} For any smooth vector field $X$ with compact support in ${U}_\Sigma$,   one has
\begin{equation}\label{newdivelemm}
\int_{{{U}}} g( \nabla u ,X) {\dmu}=-\int_{{{U}}}u \di_\mu X {\dmu},\quad \forall\,u\in W^{1,p}({U},\phi,\psi),
\end{equation}
where
$\di_\mu (X) {\dd}\mu:={\dd}(X\rfloor {\dd}\mu) $, i.e.,
\begin{equation}\label{mudivergence}
 \di_\mu(X)=\di(X)+\langle X,\nabla r\rangle \frac{\partial}{\partial r}\log\left( \frac{\phi(r)^{p-1}|\psi'(r)|}{ |\nabla(\log \psi(r))|^{p}}  \right).
\end{equation}
In particular, $|\di_\mu(X)|<+\infty$ follows by Assumption \ref{Basciasumptionfunction2}/\eqref{localboundedofnu} and the compactness of $\supp X$.

\end{enumerate}
\end{remark}

In the sequel, $U$ is either $M$ or $\Omega_\Sigma$.
The same argument as in Corollary \ref{finsharpcorollary} yields the following result.
\begin{lemma}\label{lipshccompact}
For
any globally Lipschitz function $u$  with compact support in $U$, we have
$u\in W_0^{1,p}({U},\phi,\psi)$.
\end{lemma}

Due to Lemma \ref{lipshccompact}, the following result can be proved in the same way as in Hebey \cite[Theorem 2.7]{Hebey}.
\begin{theorem}\label{MandM0isequal}
$W^{1,p}(M,\phi,\psi)=W_0^{1,p}(M,\phi,\psi)$.
\end{theorem}

Theorem \ref{MandM0isequal} combined with the same proof of Zhao \cite[Lemma B.4]{Z3} produces the following result.
\begin{corollary}\label{compwp}
If $u\in W^{1,p}(M,\phi,\psi)$ with compact support in ${\Omega_\Sigma}$, then $u|_{\Omega_\Sigma}\in W_0^{1,p}(\Omega_\Sigma,\phi,\psi)$.
\end{corollary}

In order to show the reflexivity of $W^{1,p}(M,\phi,\psi)$, we
recall Mazur's lemma (cf. Renardy and Rogers \cite[Lemma 10.19]{RR}).

\begin{lemma}[Mazur's lemma]\label{Marzurlemma} Assume that $X$ is a Banach space and that
$x_i\rightarrow x$ weakly in $X$ as $i \rightarrow \infty$. Then there exists a sequence of convex
combinations
\[
\tilde{x}_i=\sum_{j=i}^{m_i}a_{i,j}x_j,\ a_{i,j}\geq 0,\ \sum_{j=i}^{m_i}a_{i,j}=1,
\]
such that $\tilde{x}_i\rightarrow x$ strongly in $X$ as $i\rightarrow \infty$.
\end{lemma}

\begin{theorem}\label{weakconvergece}
$W^{1,p}(M,\phi,\psi)$ is reflexive, i.e.,  given  a bounded sequence $\{u_i\}$  in $W^{1,p}(M,\phi,\psi)$, there exists a subsequence $\{u_{i_k}\}$ and $u\in  W^{1,p}(M,\phi,\psi)$ such that $u_{i_k}\rightarrow u$ weakly in $L^p(M,\phi,\psi)$  and $\nabla u_{i_k}\rightarrow \nabla u$ weakly in $L^p(TM,\phi,\psi)$ as $k\rightarrow \infty$.
\end{theorem}

\begin{proof}
Since $\{u_i\}$ is a bounded sequence in $W^{1,p}(M,\phi,\psi)$,  Theorem \ref{LPreflexive} yields    a subsequence $\{u_{i_k}\}$ and $u\in L^p(M,\phi,\psi)$ and a vector field $X\in  L^p(TM,\phi,\psi)$ such that $u_{i_k}\rightarrow u$ weakly in $L^p (M,\phi,\psi)$ and $\nabla u_{i_k}\rightarrow X$ weakly in $L^p(T M,\phi,\psi)$ as $k\rightarrow \infty$.
Consider the Banach space $L^p(M,\phi,\psi)\times L^p(TM,\phi,\psi)$ endowed with the product metric. Obviously,
\[
(u_{i_k},\nabla u_{i_k})\rightarrow (u,X) \text{ weakly in }L^p(M,\phi,\psi)\times L^p(TM,\phi,\psi).
\]
Thus, by Lemma \ref{Marzurlemma} we have a  sequence of convex combinations such that
\[
\sum_{k=l}^{m_l}a_{l,k}(u_{i_k},\nabla u_{i_k})\rightarrow (u,X) \text{ strongly in }L^p(M,\phi,\psi)\times L^p(TM,\phi,\psi), \text{ as }l\rightarrow \infty.
\]
Set $v_l:=\sum_{k=l}^{m_l}a_{l,k}u_{i_k}\in W^{1,p}(M,\phi,\psi)$. Then both $v_l\rightarrow u$ strongly  and $\nabla v_l\rightarrow X$ strongly in the corresponding $[\cdot]_{p,\mu}$-norms. Hence, $(v_l)$ is a Cauchy sequence in $W^{1,p}(M,\phi,\psi)$, which implies $\nabla u=X$ and $u\in W^{1,p}(M,\phi,\psi)$.
\end{proof}

Theorem \ref{weakconvergece} together with the dominated convergence theorem implies the following  two consequences; we omit their proofs because they are standard.

\begin{corollary}\label{maxminlenew}
If $u\in {W^{1,p}}(M,\phi,\psi)$, then $u_+:=\max\{u,0\}$, $u_-:=-\min\{u,0\}$ and $|u|=u_+-u_-$ are all in ${W^{1,p}}(M,\phi,\psi)$.  In particular, $\nabla|u|=\sgn(u) \nabla u$.
\end{corollary}

\begin{corollary}\label{trunctionfun}
Given $u\in W^{1,p}(M,\phi,\psi)$, we have $\max\{0,\min\{u,1\}\}\in W^{1,p}(M,\phi,\psi)$. Moreover,
for any $\lambda>0$, set $u_\lambda:=\max\{-\lambda,\min\{u,\lambda\}\}$. Then $u_\lambda\rightarrow u$ in $ W^{1,p}(M,\phi,\psi)$ as $\lambda \rightarrow +\infty$.
\end{corollary}

\begin{proposition}\label{usefulproweakdre}Given $u,v\in W^{1,p}(M,\phi,\psi)$, the  following statements hold$:$
\begin{enumerate}[{\rm (i)}]

\item\label{maxgrad} $\max\{u,v\}\in W^{1,p}(M,\phi,\psi)$ and $|\nabla \max\{u,v\}|\leq \max\{|\nabla u |, |\nabla v | \}$;

\item\label{Neunirule} if  $\|u\|_\infty+\|v\|_\infty<+\infty$, then $uv\in W^{1,p}(M,\phi,\psi)$ and $\nabla (uv)=u\nabla v+v\nabla u$.

\end{enumerate}

\end{proposition}
\begin{proof}
  Since $\max\{a,b\}=\frac12(a+b+|a-b|)$,  \eqref{maxgrad} immediately follows by Corollary \ref{maxminlenew} and Remark \ref{smoothdivege}/\eqref{nablaproperty}.
In order to prove  \eqref{Neunirule}, by Theorem \ref{MandM0isequal} we have two sequences $\{u_i\},\{v_i\}$ in $C^\infty_0(M)$ such that $u_i\rightarrow u$ and $v_i\rightarrow v$ under $\|\cdot\|_{p}$. Let $u^*_i:= \max\{-\lambda,\min\{u_i,\lambda\}\}$ and $v^*_i:= \max\{-\lambda,\min\{v_i,\lambda\}\}$, where $\lambda$ is a constant satisfying $\|u\|_\infty+\|v\|_\infty\leq \lambda$. Thus,  $u^*_iv^*_i$ is a Lipschitz function with compact  support, which together with Lemma \ref{lipshccompact} and Theorem \ref{MandM0isequal} implies $u^*_iv^*_i\in W^{1,p}(M,\phi,\psi)$.
    Moreover,
    it follows from Corollary \ref{trunctionfun}  that $u^*_i,v^*_i\in W^{1,p}(M,\phi,\psi)$ and  $u^*_i\rightarrow u$, $v^*_i\rightarrow v$ under $\|\cdot\|_{p}$. Since  $\|u^*_i\|_\infty+\|v^*_i\|_\infty\leq 2\lambda$,
  it is easy to check $u^*_iv^*_i\rightarrow uv$ and $\nabla(u^*_iv^*_i)\rightarrow v\nabla u+u\nabla v$ under the corresponding $[\cdot]_{p,\mu}$-norms.
  Hence, $\{u^*_iv^*_i\}$ is a Cauchy sequence in $W^{1,p}(M,\phi,\psi)$. We conclude the proof by $u^*_iv^*_i\rightarrow uv$ in $\|\cdot\|_{p}$.
  \end{proof}

\begin{theorem}\label{pointswisestrongly}
Suppose that $\{u_i\}$ is a bounded sequence in $W^{1,p}(M,\phi,\psi)$ and $u_i\rightarrow u$ pointwise a.e.\ in ${M}$. Then $u\in W^{1,p}(M,\phi,\psi)$, $u_i\rightarrow u$ weakly in $L^p(M,\phi,\psi)$ and $\nabla u_i\rightarrow \nabla u$ weakly in $L^p(TM,\phi,\psi)$.
\end{theorem}
\begin{proof}Choose an arbitrary subsequence $\{u_{i_k}\}$ of $\{u_i\}$.
According to Theorem \ref{weakconvergece},  there is a subsequence  of $\{u_{i_k}\}$, still denoted by $\{u_{i_k}\}$, and a function $v\in W^{1,p}(M,\phi,\psi)$ such that $u_{i_k}\rightarrow v$ weakly in $L^p(M,\phi,\psi)$ and $\nabla u_{i_k}\rightarrow \nabla v$ weakly in $L^p(TM,\phi,\psi)$. We now show that $v=u$ a.e.\ in ${M}$.

  Lemma \ref{Marzurlemma} together with the proof of Theorem \ref{weakconvergece} provides a sequence of convex combinations with
\[
u^*_{k}:=\sum_{j=k}^{m_k}a_{k,j}u_{i_j}\rightarrow v,\ \ \nabla u^*_k=\sum_{j=k}^{m_k}a_{k,j}\nabla u_{i_j}\rightarrow \nabla v, \text{ as }k\rightarrow \infty,
\]
in the corresponding $[\cdot]_{p,\mu}$-norms. Thus, the standard theory of $L^p$-space together with Lemma \ref{nullmeaure} yields a subsequence of $\{u^*_k\}$ converging  pointwise  to $v$ a.e. in ${M}$. On the other hand, since $\{u_j\}$ converges to $u$ point-wisely, we see that $\{u^*_k\}$ converges to $u$ point-wisely as well. Consequently,
  $v=u$ a.e. and therefore $\nabla u=\nabla v$ a.e. in ${M}$.

Therefore, we proved that for every subsequence of $\{u_i\}$, there is another subsequence $\{u_{i_k}\}$ such that  $u_{i_k}\rightarrow u$ weakly in $L^p(M,\phi,\psi)$ and $\nabla u_{i_k}\rightarrow \nabla u$ weakly in $L^p(TM,\phi,\psi)$, which indicates that the original sequence
$\{u_i\}$ also satisfies such a property.
\end{proof}

\begin{corollary}\label{maxepcro}
 Given $u\in W^{1,p}(M,\phi,\psi)$, we have $\max\{u-\epsilon,0\}\in  W^{1,p}(M,\phi,\psi)$ for any $\epsilon>0$.
\end{corollary}
\begin{proof}
Due to Theorem \ref{MandM0isequal}, there exists a sequence $\{u_i\}$ in $C^\infty_0(M)$ such that $u_i\rightarrow u$ under $\|\cdot\|_{p}$ and $\|u_i\|_{p}\leq \|u\|_{p}+1$. By passing to a subsequence, we may assume that $\{u_i\}$ converges to $u$ point-wisely. Let $v_i:=\max\{u_i-\epsilon,0\}$. Since $v_i$ is a Lipschitz function with compact support, Lemma \ref{lipshccompact} yields $v_i\in W^{1,p}(M,\phi,\psi)$. Moreover, note that $\|v_i\|_{p}\leq \|u_i\|_{p}\leq  \|u\|_{p}+1$ and
$\{v_i\}$ converges to $\max\{u-\epsilon,0\}$ point-wisely, which  implies $\max\{u-\epsilon,0\}\in  W^{1,p}(M,\phi,\psi)$ through  Theorem \ref{pointswisestrongly}.
\end{proof}

Now we introduce the capacity with respect to $\|\cdot\|_{p}$.

\begin{definition}The {\it $p$-capacity} of a set $E\subset M$ is defined by
\begin{align*}
\Ca_{p}(E):=\inf_{u\in \mathscr{A}(E)}\|u\|^p_{p},
\end{align*}
where $\mathscr{A}(E)=\{ u\in {W^{1,p}}(M, \phi,\psi)\,:\,u\geq 1 \text{ on  a neighborhood of }E \}$.
In particular, set $\Ca_{p}(E):=\infty$ if $\mathscr{A}(E)=\emptyset$.
\end{definition}

Proceeding as in the proof of Kinnunen and Martio \cite[Remark 3.1]{KM} and using Corollary \ref{trunctionfun}, we have the following result.

\begin{proposition}
Let $\mathscr{A}'(E):=\{u\in {W^{1,p}}(M, \phi,\psi)\,:\,0\leq u\leq 1 \text{ and } u=1\text{ on  a neighborhood of }E \}$. Then
$
\Ca_{p}(E)=\inf_{u\in \mathscr{A}'(E)}\|u\|^p_{p}.
$
\end{proposition}

The following result collects the basic properties of  $\Ca_{p}$; since the proofs are standard, based on Theorem \ref{pointswisestrongly}, Proposition \ref{usefulproweakdre} and Lemma \ref{Marzurlemma}, we omit them.

\begin{theorem}\label{propoertyfocap}
The  $p$-capacity is an outer measure, that is
\begin{enumerate}[{\rm (i)}]
\item\label{Cap1} $\Ca_{p}(\emptyset)=0$;
\item\label{Cap2} $\Ca_{p}(E_1)\leq \Ca_{p}(E_2)$ if $E_1\subset E_2$;
\item\label{Cap3}
$\Ca_{p}\left( \bigcup_{i=1}^\infty E_i  \right)\leq \sum_{i=1}^\infty\Ca_{p}(E_i)$;
\item\label{Cap4} $\Ca_{p}(\cdot)$ is outer regular, i.e., $\Ca_{p}(E)=\inf\left\{ \Ca_{p}(O)\,:\, O \text{ is open with }E\subset O \right\}$.
\end{enumerate}
\end{theorem}

\begin{definition}
We say that a property holds {\it $p$-quasieverywhere in a set $E$}, if it holds except for a subset of $p$-capacity zero in $E$.
In addition,
a function $u:M\rightarrow [-\infty,+\infty]$ is {\it $p$-quasicontinuous} if for any $\epsilon>0$, there is a set $E$ such that $\Ca_{p}(E)<\epsilon$ and the restriction  $u|_{M\backslash E}$ is continuous.
\end{definition}

A suitable modification of the argument from Chen et al. \cite[Theorem A.16]{CLZ}, based on the the above notions, provides the following result:

\begin{theorem}\label{maintheoreminaapendix}
If $u\in W^{1,p}(M,\phi,\psi)$ is $p$-quasicontinuous and $u=0$ $p$-quasieverywhere in $M\backslash {\Omega_\Sigma}$, then ${u|_{{\Omega_\Sigma}}}\in W^{1,p}_0(\Omega_\Sigma,\phi,\psi)$.
\end{theorem}

\textbf{Declarations.} 

\textit{Conflict of interest.} The authors state that there is no conflict of interest.

\end{document}